\newcounter{lil1}
\newenvironment{steps}
{\begin{list} { \sl Step (\Roman{lil1})}
		{ \usecounter{lil1}
			\setlength{\leftmargin}{0.0cm}
			\setlength{\topsep}{0.2cm}
			\setlength{\itemsep}{0.0cm}
			\setlength{\parsep}{0.1cm}
			\setlength{\itemindent}{0.8cm}
			\setlength{\parskip}{0.0cm}}}
	{\end{list}}

\newcounter{lil33}
\newenvironment{stepsinner}
{\begin{list} { \sl Step (\alph{lil33})}
		{ \usecounter{lil33}
			\setlength{\leftmargin}{0.0cm}
			\setlength{\topsep}{0.2cm}
			\setlength{\itemsep}{0.0cm}
			\setlength{\parsep}{0.1cm}
			\setlength{\itemindent}{0.8cm}
			\setlength{\parskip}{0.0cm}}}
	{\end{list}}

\newenvironment{step}
{\begin{list} { \bf Step (\Roman{lil1})}
		{ \usecounter{lil1}
			\setlength{\leftmargin}{0.0cm}
			\setlength{\topsep}{0.2cm}
			\setlength{\itemsep}{0.0cm}
			\setlength{\parsep}{0.1cm}
			\setlength{\itemindent}{0.8cm}
			\setlength{\parskip}{0.0cm}}}
	{\end{list}}

\newenvironment{stepinner}
{\begin{list} { \bf Step (\alph{lil33})}
		{ \usecounter{lil33}
			\setlength{\leftmargin}{0.0cm}
			\setlength{\topsep}{0.2cm}
			\setlength{\itemsep}{0.0cm}
			\setlength{\parsep}{0.1cm}
			\setlength{\itemindent}{0.8cm}
			\setlength{\parskip}{0.0cm}}}
	{\end{list}}

\newcommand{\ns}{{n^\ast}}

\newcommand{\BH}{{\mathbb{H}}}

\newcommand{\uk}{u_\kappa}
\newcommand{\vk}{v_\kappa}
\newcommand{\buk}{\bar u_\kappa}
\newcommand{\bvk}{\bar v_\kappa}

\newcommand{\n}{\kappa}
\newcommand{\CMM}{{\mathcal{M}}}
\newcommand{\cmm}{\mathcal{M}}
\newcommand{\mynegspace}{\hspace{-0.13em}}
\newcommand{\lvvvert}{\rvert\mynegspace\rvert\mynegspace\rvert}
\newcommand{\rvvvert}{\rvert\mynegspace\rvert\mynegspace\rvert}

\renewcommand{\n}{\kappa}

\newcommand{\LLm}{{s^\ast}}
\newcommand{\pst}{{p^\ast}}
\renewcommand{\ns}{{q}}
\newcommand{\alneu}{{\rho}}

\newcommand{\pp}{p+1}

\newcommand{\TTend}{T]} 
\newcommand{\MA}{\mathfrak{A}}

\newcommand{\Law}{\mbox{Law}}

\newcommand{\essup}{\mathop{\mathrm{sup}}}

\newcommand{\eps}{\varepsilon}
\newcommand{\Lve}{\lVert}
\newcommand{\Rve}{\rVert}

\newcommand{\oO}{{\mathcal{O}}}

\newcommand{\pmat}{\begin{matrix}}
\newcommand{\emat}{\end{matrix}}

\newcommand{\baray}{\begin{array}{rcl}}
\newcommand{\earay}{\end{array}}
\newcommand{\barray}{\begin{array}{rcl}}
\newcommand{\earray}{\end{array}}

\newcommand{\FE}{\aleph}
\newcommand{\FEplus}{\frac{\aleph}{2}}

\newcommand{\Si}{(\rho, \FE)}

\newcommand{\Sitilde}{(\rho, \FE)}

\newcommand\dela[1]{}
\newcommand\casedrei[1]{}

\newcommand{\bcase}{\begin{cases}}
\newcommand{\ecase}{\end{cases}}

\newcommand{\spK}{\mathcal{K} _\MA(K_1,K_2,K_3) }

\newcommand\Leb{\mbox{Leb}}
\newcommand\FA{\mathfrak{A} }

\newcommand\unbekannt{\mu}

\newcommand\del[1]{}

\del{

\newcommand\del[1]{}
}

\newcommand{\embed}{\hookrightarrow}

\def\eps{\varepsilon}

\newcommand{\lk}{\left}
\newcommand{\lqq}{\lefteqn}
\newcommand{\rk}{\right}
\newcommand{\la}{{\langle}}

\newcommand{\ra}{{\rangle}}

\newcommand{\LL}{{\rm I \kern -0.2em L}}

\newcommand{\ep} {\varepsilon }

\newcommand{\be} {\begin{enumerate} }
\newcommand{\ee} {\end{enumerate} }

\newcommand{\CO}{{{ \mathcal O }}}
\newcommand{\CK}{{{ \mathcal K }}}

\newcommand{\CC}{{{ \mathcal C }}}

\newcommand{\CV}{{{ \mathcal V }}}

\newcommand{\CM}{{{ \mathcal M }}}

\newcommand{\BF}{{{ \mathbb{F} }}}
\newcommand{\CF}{{{ \mathcal F }}}

\newcommand{\CN}{{{ \mathcal N }}}

\newcommand{\RR}{{\mathbb{R}}}

\newcommand{\WW}{{\mathbb{W}}}

\newcommand{\dd}{\mathbb{d}}
\newcommand{\NN}{\mathbb{N}} 

\newcommand{\PP}{{\mathbb{P}}}

\newcommand{\EE}{ \mathbb{E} }

\newcommand{\DEQS}{\begin{eqnarray*}}
\newcommand{\EEQS}{\end{eqnarray*}}
\newcommand{\DEQSZ}{\begin{eqnarray}}
\newcommand{\EEQSZ}{\end{eqnarray}}
\newcommand{\DEQ}{\begin{eqnarray}}
\newcommand{\EEQ}{\end{eqnarray}}

\newcommand{\rr}{r}

\documentclass[11pt,fleqn,english]{amsart}
\usepackage{xcolor}
\usepackage{dsfont}
\usepackage{amsmath,mathtools,amsfonts,amssymb}

\usepackage{enumitem}
\usepackage{geometry}

\usepackage{bbm}
\usepackage[linkcolor=black,citecolor=black]{hyperref}

\definecolor{apricot}{rgb}{0.98, 0.81, 0.69}
\definecolor{babyblue}{rgb}{0.54, 0.81, 0.94}
\usepackage{tikz}
\usetikzlibrary{shapes.geometric, arrows}

\usepackage{color}
\usepackage{xcolor}

\definecolor{DarkGreen}{rgb}{0.1,0.6,0.2}   
\definecolor{Yellow}{rgb}{0.7,0.7,0.7}   
\usepackage{nicefrac}

\newcommand{\tesfalem}[1]{{#1}}  
\usepackage{ulem}
\usepackage{amsmath}   
\usepackage{amsfonts,amssymb,color}
\usepackage{graphicx}

\newcounter{michicounter}

\usepackage{enumitem}
\usepackage{latexsym}
\input colordvi
\usepackage{bbm}

\usepackage{amsfonts}

\usepackage{todonotes}
\newtheorem{theorem}{Theorem}[section]

\newtheorem{lemma}[theorem]{Lemma}
\newtheorem{corollary}[theorem]{Corollary}
\newtheorem{hypo}[theorem]{Assumption}
\newtheorem{assumption}[theorem]{Assumption}
\newtheorem{tlemma}{Technical Lemma}[section]
\newtheorem{definition}[theorem]{Definition}
\newtheorem{remark}[theorem]{Remark}

\newtheorem{proposition}[theorem]{Proposition}
\newtheorem{tproposition}[theorem]{Technical Proposition}

\usepackage{todonotes}
\allowdisplaybreaks

\numberwithin{equation}{section}

\tikzstyle{Setting} = [rectangle, minimum width=3cm, minimum height=1cm, text centered, text width=8cm,  draw=black, fill=blue!30]
\tikzstyle{Theorem} = [rectangle, minimum width=3cm, minimum height=1cm, text centered, text width=3cm,  draw=black, fill=blue!30]
\tikzstyle{Proposition} = [rectangle, minimum width=3cm, minimum height=1cm, text width=3cm, text centered, draw=black, fill=green!30]
\tikzstyle{InText} = [rectangle, minimum width=3cm, minimum height=1cm, text width=6cm, text centered, draw=black, fill=white!30]
\tikzstyle{Technical Proposition} = [rectangle, minimum width=3cm, minimum height=1cm, text width=3cm, text centered, draw=black, fill=green!30]
\tikzstyle{TechnicalLemma} = [rectangle, minimum width=3cm, minimum height=1cm, text centered, text width=3cm,draw=black, fill=green!30]
\tikzstyle{Lemma} = [rectangle, minimum width=3cm, minimum height=1cm, text centered, text width=3cm,draw=black, fill=green!30]
\tikzstyle{arrow} = [thick,->,>=stealth]
\tikzstyle{Hypothesis} = [rectangle, minimum width=3cm, minimum height=1cm, text width=5cm, text centered, draw=black, fill=orange!30]
\tikzstyle{Definition} = [rectangle, minimum width=3cm, minimum height=1cm, text width=5cm, text centered, draw=black, fill=yellow!30]

\begin{document}

\title[Coupled reaction-diffusion systems; \today]
{{Weak solutions for} coupled reaction-diffusion systems 
{with} pattern formation {by a stochastic} fixed point theorem}

\author{E. Hausenblas}
\address{E. Hausenblas, Department of Mathematics,
Montanuniversitaet Leoben, Austria.}
\email{erika.hausenblas@unileoben.ac.at}

\author{M.A. H\"ogele}
\address{M.A. H\"ogele, Departamento de Matem\'aticas, Universidad de los Andes, Bogot\'a, Colombia}
\email{ma.hoegele@uniandes.edu.co}

\author{T.A. Tegegn}
\address{T.A. Tegegn, Department of Mathematics And Applied Mathematics, University of Pretoria, South Africa.}
\email{Tesfalem.Tegegn@up.ac.za}

\begin{abstract}
Chemical and biochemical reactions
can exhibit surprisingly different behaviours, ranging from multiple steady-state solutions to oscillatory solutions and chaotic behaviours.
These types of systems are often modelled by a  system of reaction-diffusion
equations coupled by a nonlinearity. In the article, we study the existence
of stochastically perturbed equations of this type.
In particular, we show the existence of {a probabilitic weak} solution of the following stochastic system 
\DEQS
\dot {u} & =& r_1\,\Delta   u+ a_1\, u + b_1  -c_1\, u\cdot v^q+\sigma_1\, g_1(u)\circ \dot W_1 , 
\\
\dot{v}& = &r_2\,A v  + a_2\, v + b_2 +c_2\, u\cdot v^q+\sigma_2\, g_2(v)\circ \dot W_2 ,
\EEQS
where
$r_i,b_i,c_i, \sigma_i>0$, $a_i\in\RR$, and  $g_i$ are linear, $i=1,2$, and
the exponent $q\geq 1$. The operator $A=-(-\Delta)^{\FE/2}$ is a fractional power of the Laplacian, ${1<}\FE\le2$.
The main result is obtained by a Schauder--Tychonoff type fixed point theorem
for the controlled versions of the laws of the respective (infinite dimensional) 
Ornstein-Uhlenbeck system, from which we infer the existence of a martingale solution of the coupled system.
\end{abstract}

\maketitle



\bigskip 
\section{\textbf{Introduction}}\label{one}

\noindent Pattern formation is an  umbrella  term for the emergence of characteristic variations of a quantity {of interest} as a function of a spatial parameter over time. It is a natural phenomenon that may occur in different contexts and shapes, such as the height {above sea level} or colour of landscapes, the repartition of vegetation species, the mixing patterns of different colour spots on animal skin, or the colouring of flowers. In his pioneering work in 1952, \cite{turing},  Alan Turing
investigates the possibility of such spatial instabilities occurring in purely dissipative systems involving chemical reactions far from equilibrium under the transport process of diffusion. Pattern formation via diffusion-driven instabilities plays a vital role in biology, see \cite{meinhardt, MinchevaRoussel12, murray2}. Besides, Turing's diffusive-instability mechanism has also been proposed for modelling self-organised vegetation pattern formation in arid environments \cite{hardenberg, klausmeier,levejer}, for modelling pigmentation patterns in fishes \cite{kondo}, and for population dynamics \cite{population}. For more detailed references, see, for instance, \cite{ghergu,meron,wei}. 

\noindent The Gray-Scott model is a typical nonlinear reaction-diffusion system exhibiting pattern formation for certain parameter constellations, whose stochastic counterpart is studied in this article. It consists of an activator component $u$, which satisfies a linear dissipative equation coupled to an inhibitor component $v$, with a feedback component proportional to a nonlinear power of $v$. 
In \cite{You08}, the author establishes the existence of a finite-dimensional global attractor for the Gray-Scott system. More general necessary conditions for the appearance of Turing instabilities are studied in  \cite{elragigtownley}. In \cite{Dilao05, RicardMischler09}, the authors detect Turing patterns if the system enters the vicinity of a Hopf bifurcation. In general, the mathematically precise understanding of pattern formation is still at its beginning. 

\noindent Noisy random fluctuations are ubiquitous in the real world.
Randomness leads to a variety of new phenomena and may have a non-trivial impact on the behaviour of the solution. The presence of the stochastic term (or noise) in the model often leads to qualitatively new 
{features}, which help to understand the {real-world phenomena} and often produces more realistic outcomes. Adding noise in a system of partial differential equations (PDEs) induces, { among others, the following }new and important phenomena:
due to the interplay of the noise and the nonlinearity, noise-induced transitions, stochastic resonance, metastability, or noise-induced chaos may appear.
Numerical experiments indicate that noise in stochastic Turing patterns expands the range of parameters in which Turing patterns appear. So, the stochastic extension of the deterministic system leads to pattern formation with less strict parameter conditions, see \cite{karig}. {There,} the authors explore whether the stochastic extension leads to a larger range of parameters with Turing patterns by a genetically engineered synthetic bacterial population in which the signalling molecules form a stochastic activator{\textendash}inhibitor system. 
The effect of noise on Turing patterns in the case of the Belousov-Zhabotinsky reaction-diffusion model was analysed in \cite{noise1}. Butler and Goldenfeld demonstrate in \cite{butler} that noisy demographic perturbations can induce persistent spatial pattern formation and temporal oscillations in the Levin-Segel predator-prey model for plankton-herbivore population dynamics. Biancalani et al. in \cite{noise2} {confirm} the impact of noise on the Turing pattern on several examples and show that with random noise, the range of parameters where Turing patterns may appear enlarged. In \cite{mechtild}, it is shown that adding a multiplicative noise term to the Grey-Scott system breaks the symmetry and changes the dependence on the parameter, and, depending on the noise, also changes the long-term asymptotics.
In \cite{noise3}, spatiotemporal chaos or/and spatial intermittency in the setting of SPDEs is investigated. In \cite{noise5,noise4,noise4A}, multistability and noise-induced transitions between different states are addressed.
In \cite{veraar}, the authors studied some special case of the generalised system.

\noindent In this article, we propose the implementation of a new method of proving rigorously the existence of a weak solution for a class of semilinear activator-inhibitor systems, exemplified by the stochastically perturbed Gray-Scott system with a nonlocal activator diffusion operator, which is based on the stochastic Schauder fixed point theorem shown in \cite{klausmeier}. 
Mathematically, nonlocal operators appear naturally as the continuum limit of the transition operators of random walks of L\'evy flight type in which, at a given position and time, an individual is assumed to ``jump'' to a new, often non-neighbouring position, for instance with a power tail law. On the phenomenological side, since Einstein in 1905 \cite{Ei05}, a Brownian motion is interpreted as a particle which moves freely in a liquid under continuous excitation by the much smaller particles of the substrate. 
This perspective has been used to model free movements in biology, chemistry, ecology and many other disciplines, highlighting the importance of studying reaction-diffusion equations. 
Remarkably, there are situations, in particular, in the presence of macromolecules, in which
the random movements do not only happen within local neighbourhoods, 
but from time to time also on a macroscopic range, on which jump increments are perceived, see, e.g. \cite{SSH+08, SHH+19}. In these scenarios, the (continuous) Brownian motion is no longer an appropriate concept to describe such movements and should be replaced by a (discontinuous) L\'evy process. Consequently, the (negative) Laplace operator $\Delta$, which is the generator of the Brownian motion, should be replaced by a non-local operator $A$, which generates the Levy process under consideration. 
{The most} representative models of non-local operators include fractional Laplacians, 
which have been intensively studied in the last decades (see, e.g. \cite{EDNV12} and references therein), 
and {are the Markov generators of stochastic} $\alpha$-stable L\'evy {processes} \cite{ST94}. Furthermore, L\'evy flights and nonlocal operators have been detected to appear naturally in various power law thresholds situations such as locomotion, foraging and nonlocal search, paleoclimate and optimisation, see \cite{Dit99, GHKM17, SSH+08} and references therein. 

\noindent One of the particular features of such a coupled activator-inhibitor system is the lack of global dissipativity. In particular, the nonlinearities of the inhibitor components have negative coefficients, which allow for dissipativity estimates, while the sign of the activator nonlinearity on the right-hand side is positive, which destroys the global dissipativity. Hence, a priori estimates of particular finite-dimensional Faedo-Galerkin expansions are increasingly challenging to obtain and technically involved and have to be carried out for each system and implementation anew. Applying a stochastic Schauder fixed point theorem, which yields nonlinear local (cutoff) martingale solutions, allows a stronger modularisation of the steps of the proof, as carried out in the roadmap outlined below.

\noindent In \cite{AC99}, the authors study deterministic semilinear parabolic partial differential equations over a bounded smooth domain with Neumann boundary conditions with an unmitigated positive nonlinear ``activator'' reaction term. They obtain critical exponents on the algebraic nonlinearity depending on the (spatial) integrability of the initial conditions and the dimension, such that all smaller exponents yield global solutions, while larger nonlinearities trigger blow-ups. 
Such a behaviour is, essentially, what we obtain { - a priori - } for the activator in the Gray-Scott system, and our results are valid only for bounded values, which seem to be close to optimal. Nevertheless, the Gray-Scott system, as well as more coupled general activator-inhibitor systems, exhibit the following intrinsic moderating feedback, which yields that the solution of the deterministic solution exhibits a global attractor \cite{You08}, which is finite-dimensional. {A posteriori, the situation looks rather different:} Imagine the values of the inhibitor $v$ become very large. Hence, the inhibitor is strongly mitigated since it receives the feedback term proportional to $-u v^q$ on the right-hand side, which, due to the nonnegativity of $u,v$ and the negative sign, leads to very small values of $u$. These small values of $u$ then enter the feedback term of the activator $v$, which is proportional to $u v^q$, such that, combined with the dissipativity of the respective linear term, yields smaller values of $v$. The stochastic terms applied to our system do not change this picture. Consequently, the confinement of deterministic dynamics in the sense of \cite{You08} does not come as a surprise. 

\noindent The proof strategy is as follows.
First, we consider the solution operator of the respective linear equations, where the nonlinearities
\begin{enumerate}
\item[(1)] are cut off smoothly (with respect to a cutoff parameter $\kappa$) and, in particular, continuously with respect to the path space topology,
\item[(2)] are parametrized by external controls $(\eta, \xi)$.
\end{enumerate}
We consider the operator $\CV^\kappa$ which sends for each cutoff level $\kappa$ the pair $(\eta, \xi)$ of parametrizations to the solution $(u_\kappa, v_\kappa)$ of the linear  system
\begin{align*}
d{\uk }(t)&= [r_1 \,\Delta \uk (t)+ a_1 \uk (t) + b_1 - c_1 \phi_\n(\xi,t)\cdot \uk(t)\cdot \,\xi^q(t)]\, dt +\sigma_1 \, g_{\gamma_1}(\uk (t))\, dW_1(t)\\
d{\vk  }(t)&= [r_2 \,A \vk  (t)+ a_2 \vk  (t)+ b_2+ c_2\phi_\n(\xi, t)\cdot \uk(t)\cdot \xi^q(t)]\, dt + \sigma_2 \,g_{\gamma_2}(\vk  (t)) \, dW_2(t).
\end{align*}
Such an operator is well-defined due to the smoothing of the diffusion operators for both the inhibitor and the activator.
Furthermore, we show that the image set is continuous and has a compact closure. This { step} would be trivial for {purely } dissipative nonlinearities. {In order to control the non-dissipative } activator part we use the (dimension-dependent) compact embedding of the nonlinearity {in negative Sobolev spaces}.
Then, applying the Stochastic Schauder theorem yields a global martingale solution for the cutoff system, which, up to a $\kappa$-hitting time, coincides with the original problem. The consecutive glueing of local solutions along a sequence of hitting times yields a global solution for each given time interval $[0, T]$.

\noindent We stress that our equations and the solution methods are generic in the choice of parameters insofar as no specific parameter constellations of $r_i,b_i,c_i, \sigma_i>0$ and $a_i\in\RR$ \tesfalem{for $i=1,\,2$}, are taken into account. In particular, no term cancellations by summing or subtraction of the equations are applied whatsoever. 

\noindent {The novelty of this article lies in the following three aspects: 
First we lay out a complete \textit{modularized road map} to prove the existence of a martingale solution for an important, nontrivial model equation, which avoids to choose specific bases, as in the classical Faedo-Galerkin methods. Secondly, in our main result in Theorem~\ref{thm_main}, we extend the parameter range from quadratic in dimension $2$ (which is known in the literature \cite{HT22plus} for the Laplacian $\aleph =2$ in front of the activator term) 
to a tradeoff between to nonlinearities $1 < q \leq 2$, the non-local power of the Laplacian $1 < \aleph \leq 2$, the dimension $d$ and the integrability of the initial conditions $p^*$ and allows for a class of fractional powers for the (non-dissipative) activator component. As a third aspect it is worth mentioning the quantitative regularity statement in Corollary~\ref{cor: main} obtained by an interpolated version of the Strook-Varopoulos inequality for the inhibitor part, which allows to shift regularity from the inhibitor to the activator part. 
}

\paragraph{\textbf{Notation: }}
For convenience we define $H := L^2(\oO)$ for a domain $\CO\subset \RR^d$ in $d\in \{1,2\}$. Furthermore, we interpret the function spaces $H^s_p(\oO)$ and $L^p(\oO)$ spaces in terms of their equivalent Triebel spaces as it appears in \cite[Theorem 2.5.6]{triebel_1983}. That is to say, for $s\in\mathbb{R},\,1<p<\infty$, we have
\DEQSZ\label{triebel}
H^s_p = F^s_{p,2} \qquad \mbox{ and } \qquad L^p = F^0_{p,2}.
\EEQSZ
For the definitions and further reading on the space $F^s_{p,2}$, we refer to \cite[Chapter 2]{runst} and \cite{triebel_1983}.

\noindent Let $\mathcal{A}$ be a generator of a $\mathcal{C}_0$-semigroup $(e^{\mathcal{A} t})_{t\geq 0}$
in a separable Banach space,
associated with the mild solution of the Cauchy problem
$$\frac{d u}{d t} = \mathcal{A} u, \qquad u(0) = u_0.$$
We refer to \cite{pazy} for details.

\bigskip 
\section{\textbf{Notation, Preliminaries and Main Result}}\label{s:mainresult}

\paragraph{\textbf{The phenomenological setting: }}
Let $\MA=(\Omega, \CF,\BF,\PP)$ be a complete probability space
with filtration $\BF=(\CF_t)_{t\ge 0}$  satisfying the usual conditions, i.e.,
$\mathbb{P}$ is complete on $(\Omega, \CF)$,
for each $t\geq 0$, $\CF_t$ contains all $(\CF,\mathbb{P})$-null sets, and
the filtration $(\CF_t)_{t\ge 0}$ is right-continuous.
Here, the underlying spatial domain $\CO\subset \RR^d$,  $d\in \{1,2\}$, is either a bounded domain with $C^\infty$ boundary, or the torus $\CO=[0,1]^d$.  Let $W_1$ and $W_2$ be two independent cylindrical Wiener processes on $H$ 
 defined over the probability space
$\mathfrak{A}$.

\noindent We consider a generic class of coupled nonlinear reaction-diffusion systems,
a stochastic two-component activator-inhibitor system which may generate Turing patterns,
covering the Schnakenberg and the Gray-Scott system with noise.
In particular, for $T>0$ we consider
\DEQSZ\label{e:u1}
d {u}(t) & =& (r_1 \Delta u(t)+ a_1 u(t) + b_1  -c_1 u(t)\cdot v^q(t))\, d t +\sigma_1\,g_{\gamma_1}(u(t))\circ d W_1(t) 
,\, \, t\in [0,T],
\EEQSZ
\DEQSZ\label{e:v1}
d {v}(t)& = & (r_2 A v(t)  +a_2 v(t) + b_2 + c_2 u(t)\cdot v^
q(t))\, d t +\sigma_2\,g_{\gamma_2}(v(t))\circ d W_2(t) ,
\,\, t\in [0,T],
\EEQSZ
with given initial conditions $u(0) = u_0$ and $v(0) = v_0$, where $r_i, a_i, b_i, c_i, \sigma_i$  are real parameters, $r_i, c_i, \sigma_i>0$, $i=1,2$.
The operator $\Delta$ denotes the Laplace operator with periodic, or Neumann boundary conditions and $A = -(-\Delta)^{\FE /2}$, $\FE \in (1,2]$,
the (non-local) fractional Laplacian with a Neumann exterior domain condition. 
For $\FE = 2$, the operator becomes a Laplace operator $A = \Delta$ with Neumann boundary condition.
The diffusion coefficient $g_{\gamma_i}(u)[h] = u \cdot (-\Delta)^{-\nicefrac{\gamma_i}{2}} h$,  $h\in H$, 
is given by a bilinear multiplication operator defined in \eqref{g_sr}.

\noindent Our main result states the existence of a weak solution of the system \eqref{e:u1} and \eqref{e:v1}, depending on the main parameters $q, \FE$, $\gamma_1$, $\gamma_2$, and $d$.

\medskip
\noindent For clarity, we define the notion of a mild solution to the system \eqref{e:u1}--\eqref{e:v1} for a given probability space $\MA$.

\begin{definition}\label{d:mild}
Fix $T>0$ and some  $\alpha\in(-\frac{d}{2},\frac{\aleph}{2} -\frac d2)$. We call a pair of processes $(u,v)$ a \textnormal{mild solution} on the given filtered probability space $\MA$ with given cylindrical Wiener processes $(W_1, W_2)$ to the system \eqref{e:u1} and \eqref{e:v1} with initial conditions $(u_0, v_0)$, if $u:[0,T] \times \Omega \rightarrow H$ 
and $v:[0,T] \times \Omega \rightarrow H^{\alpha}_2(\CO)$ 
are $\BF$--progressively  measurable processes  such that $\PP$-a.s.\ $(u,v)\in C_b([0,T]; H\times H^{\alpha}_2(\CO))) $ 
and satisfy $\PP$--a.s. for all $t\in [0,T]$ the stochastic integral equations
\begin{align*}
u(t)=e^{(r_1\Delta+ a_1I)t}u_0 &+\int_0^te^{(r_1\Delta+a_1I)(t-s)}b_1\,d s-c_1\int_0^te^{(r_1\Delta+a_1I)(t-s)}u(s)\,v^q(s)\,d s \\
&+ \sigma_1\int_0^te^{(r_1\Delta+a_1I)(t-s)}g_{\gamma_1}(u(s))\circ d W_1(s),
\\
v(t)=e^{(r_2A+a_2I)t}v_0 &+\int_0^te^{(r_2A+a_2I)(t-s)}b_2\,d s+ c_2\int_0^te^{(r_2A+a_2I)(t-s)}u(s)\,v^{[q]}(s)\,d s\\
&+\sigma_2\int_0^te^{(r_2A+a_2I)(t-s)}g_{\gamma_2}(v(s))\circ d W_2(s).
\end{align*}
\end{definition}
\noindent The stochastic integral terms are defined in the Stratonovich sense; for details, we refer to Appendix~\ref{Appendix_A_noise}.
However, the proof of our main theorem is based on compactness arguments,
such that the preceding notion of a (probabilistically) strong solution cannot be maintained since we lose the original probability space. The solution will be a martingale solution in the following (weak) probabilistic sense.
\begin{definition}\label{Def:mart-sol}
A  \textnormal{probabilistic weak solution} to the problem
\eqref{e:u1}--\eqref{e:v1} for parameters $q, \aleph$, $\gamma_1$, $\gamma_2$, and $d$
is given by the tuple
$$
\left(\Omega ,{{\mathcal{F}}},\mathbb{P},{\mathbb{F}},
(W_1,W_2), (u,v)\right)
$$
such that
\begin{itemize}
\item  $\mathfrak{A}:=(\Omega ,{{\mathcal{F}}},{\mathbb{F}},\mathbb{P})$ is a complete filtered
probability space with a filtration ${\mathbb{F}}=\{{{\mathcal{F}}}_t:t\in [0,T]\}$ satisfying the usual conditions,
\item $W_1$ and $W_2$ are two cylindrical Wiener processes on $H$ over the probability space
$\mathfrak{A}$, and  
\item $(u, v)$
is a {mild} solution to the system \eqref{e:u1}--\eqref{e:v1} over the probability space $\mathfrak{A}$ with the Wiener processes $W_1$ and $W_2$ in the sense of Definition~\ref{d:mild}.
\end{itemize}
\end{definition}

\medskip

\noindent We now list a set of sufficient conditions on the parameters $q, \aleph, \gamma_1, \gamma_2$ and $d$ for the statement of the main result.

\begin{hypo}[Parameters of the underlying spaces] \label{init_w}\hfill\\
The parameters  $(q,\FE,\alpha)$ satisfy
$$
q<\frac {\min\{\FE+d,2d\}}{2d-\FE}, \quad d\lk(\frac 12 -\frac 1q\rk) <\alpha < \frac \aleph 2  -\frac{d}{2},
$$
or $d=2$, $\FE=2$, and $q=2$ is satified.
Let  $p^\ast_0$ satisfy 
\DEQSZ\label{condpastzero}
p^\ast_0> \max\lk(\frac {2d}{ \FE +2d-dq+2q\alpha}, 4\rk).
\EEQSZ
In case $2\le q< {\FE+1}$ and $d=1$, let us put $\tau:=\frac 12-\frac 1q-\alpha$ and $p_1^\ast :=\frac \FE{q\tau}$.\\
In case $1\le q<2$ let us put $p_1^\ast:= 1/(2-q)$. Set $p^\ast:=\max\{p_0^\ast,2p_1^\ast\}$.
\end{hypo}

\begin{hypo}[Parameters of the underlying Wiener processes] \label{wiener}\hfill
\label{A_mult}
\\
Let $\oO$ be a bounded domain with $\mathcal{C}^\infty$-boundary in $\RR^d$. Let us assume that the mappings $g_{\gamma_j}$, $j=1,2$, 
are defined by
equation \eqref{g_sr} with
$$\gamma_1 > 
\frac d2+\frac d{p^\ast}  -\min(\frac 2{p^\ast},\frac d{p^\ast})$$
and
$$
\gamma_2 > 
\begin{cases} 
1-\frac \aleph2  &\mbox{ if } ~d=1,\, \alpha\ge 0, \\
&\mbox{ or } d=1, ~\alpha<0, ~\alpha +\frac{\aleph}{2} \leq \frac{1}{2},
\\
d-\frac \aleph 2 &\mbox{ if } d= 2.
\end{cases}
$$
\end{hypo}

The stage is now set to present the main result of this article.

\begin{theorem}\label{thm_main}
Given the parameters $q, \aleph, \gamma_1, \gamma_2$, and $d$, we choose the parameters $\alpha, p^\ast$ and $m$ according to  Assumption~\ref{init_w} and the Wiener process
Assumption~\ref{wiener}.
Then, for any initial conditions $u_0$ and $v_0$ being $\mathcal{F}_0$-measurable random vectors over $\MA$ and satisfying the following initial conditions
\begin{equation}\label{e: initheomain}
\EE |u_0|_{L^{2}}^2
+ \EE |u_0|^{2p^\ast_1}_{L^{p^\ast_1}}+ \EE |u_0|^{p^\ast_0}_{L^{p^\ast_0}} 
<\infty  \qquad \mbox{ and }
\qquad \EE|v_0|_{H_{2}^\alpha}^2 <\infty
\end{equation}
the system \eqref{e:u1}--\eqref{e:v1} has a probabilistic weak  solution
$$
\left(\bar \Omega ,\bar{\mathcal{F}},\bar{\mathbb{P}},\bar{\mathbb{F}},
(\bar W_1,\bar W_2), (u,v)\right).
$$
In addition, there exist positive constants $C_1 = C_1(T)$ and $C_2 = C_2(T)$ such that
\DEQSZ
&&\bar\EE\sup_{0 \le s\le T} |u(s)|_{L^2}^2 \le C(1+\EE|u_0|_{L^2}^2), \label{e:est1}\\
&&\bar\EE\sup_{0 \le s\le T} |u(s)|_{L^{p^\ast}}^{p^\ast}+{p^\ast}({p^\ast}-1)\int_0^ T \bar \EE |u^{\frac {p^\ast}2-1}(s)\nabla u(s)|_{L^2}^{2}\, ds \le C(1+\EE|u_0|_{L^{p^\ast}}^{p^\ast}), \label{e:est33}\\
&&\bar \EE\sup_{0 \le s\le T} |v(s)|_{H^{\alpha}_2}^2   + 2\bar \EE\int_0^T |v(s)|_{H^{\alpha+ \frac{\aleph}{2}}_2}^2 ds
\le C_1\Big( \EE|v_0|_{H^{\alpha}_2}^2
+ \,\EE|u_0|_{L^{p^\ast_1}}^{2p_1^\ast } +1\Big).\label{e:est2}
\EEQSZ
\end{theorem}

\begin{corollary}\label{cor: main}
Let us assume the hypotheses of Theorem \ref{thm_main}.
For any $p, m\ge 1$, where $\gamma_1$ and $p$ satisfy Assumption \ref{wiener}, there exist positive constants $C_1$, $C_2$ and $C_3$,
such that for $u_0$ and $v_0$ being $\mathcal{F}_0$-measurable random vectors over $\MA$ and satisfying the following initial conditions
satisfying $\EE[|u_0|_{L^{p^\ast}}^{p^\ast m}]< \infty$, we have
\DEQS
\lqq{	\bar{\mathbb{E}}\sup\limits_{0\leq t\leq T}|u (t)|_{L^p}^{p^\ast m}\,dx
+ C_1 \bar{\mathbb{E}}\lk[ \int_0^T\int_\CO u^{p^\ast}(s, x) v^q(s,x)\,dx\,ds\rk]^m }
&&
\\
&&{}\qquad \qquad + C_2 \bar{\mathbb{E}}\lk[\int_0^T\int_\CO|u (t,x)|^{p^\ast-2}|\nabla u (t,x)|^2\,dx\,ds \rk]^m \leq   C_3\,(\mathbb{E}|u_0 |^{p^\ast m}_{L^{p^\ast}}+1).
\EEQS
\del{In particular,
$$
\bar \EE\sup_{0 \le s\le T} |u(s)|_{L^{p^\ast}}^{\frac{2p^\ast}{p^\ast-q(p^\ast-1)}} \leq   C_3\,(\mathbb{E}|u_0 |^{\frac{2p^\ast}{p^\ast-q(p^\ast-1)}}_{L^{p^\ast}}+1).
$$}
\end{corollary}
\noindent For the proof, we refer to Proposition~\ref{reg_u2}.

\begin{remark}
The initial conditions \eqref{e: initheomain} are formulated in terms of the original probability space $\MA$, which has the same law as $\bar \MA$. Hence, they have the same moments in \eqref{e:est1} and \eqref{e:est2}. The same holds in the estimates of Corollary~\ref{cor: main}.
\end{remark}

\begin{remark} Note that the previous results imply that on the product space $L^{p^\ast}(\CO) \times H^{\alpha}_2(\CO)$ the integrability of the initial conditions of order $\frac{2p^\ast}{p^\ast-q(p^\ast-1)}$ for $u$ and $2$ for $v$, respectively, are preserved.
\end{remark}

\bigskip
\section*{\textbf{Top-level proof of the main result}}\label{five}

\noindent The stochastic integral in the system \eqref{e:u1}--\eqref{e:v1} is interpreted in the Stratonovich sense due to the physical interpretation,
as a limit of a fluctuating forcing term with a finite amplitude and finite timescale. 
To compensate for the lack of martingale property
the proof is shown as the It\^ o stochastic integral with the additional quadratic variation operator
to the It\^o form.

The proof of the main result consists of several layers.
In this section, we lay out the structure of the top layer.
Since this part consists of a series of steps,
we outline the separate steps in the following \textit{road map}.

\begin{steps}
\item Definition of the underlying spaces, depending on a parameter $\rho$.

\item Definition of the nonlinear equation with the cut-off at level $\kappa$. 

\item Introduction of the corresponding integral operator $\mathcal{V}_\kappa$ and the subset $\mathcal{K}_\kappa$, where the operator acts onto, depending on the parameter $\rho$ and $p^\ast$.

\item Properties of the fixed point operator: 

\begin{stepsinner}
\item {{wellposedness of the operator $\mathcal{V}_\kappa$}}
\item {{identification of the invariant subset $\mathcal{K}_\kappa$}}
\item {{continuity of the operator $\mathcal{V}_\kappa$}}
\item {{precompactness of the operator $\mathcal{V}_\kappa$}}
\end{stepsinner}

\item {{Application of the Schauder Theorem for any $\kappa$. }} 

\item {{Glueing of concatenated local solutions. }} 

\item {{Uniform bound in the cutoff parameter $\kappa$ of the activator depending on the parameter $\alpha\ge \rho$.   }} 

\item {Extension of the local solution to a global solution. }

\end{steps}

\medskip

\noindent We start with introducing the notation and
formulating the detailed steps of the proof 
and keep the numbering of the preceding \textit{road map}.
\begin{step}

\item  {\textbf{Definition of the underlying  spaces: }}
Let us assume firstly 
$$q<\frac {d+\FE}{2d-\FE}.
$$
First, let us fix the parameters used in the next Steps.
Let $\rho\in\RR$ satisfy
\DEQSZ\label{cond00pstar}
\frac d2-\frac \FE{2q}-\frac d{2q}<\rho\le \frac \FE 2-\frac d2,
\EEQSZ 
and  let $p^\ast$ be arbitrary but
$$
p^\ast > \max\{\frac {4d}{ \FE +d-dq+2q\rho}, 4\}.
$$

\noindent We start by introducing the activator process's main path space, depending on $\rho$.
For any $\rho \in \RR, \FE \in (1, 2]$ and $t>0$ set
\DEQSZ\label{d: mainpathspace}
\mathbb{H}^t_{\rho, \FE} &:=& L^\infty(0,t; H^\rho_2(\CO))\cap L^2(0,t;H^{\rho +\FEplus}_2(\CO))
\EEQSZ
equipped with norm
$$ \|w\|_{\BH_{\rho, \FE}^t}:= \sup_{0\le s\le t}|w(s)|_{H^\rho_2}+ \|w\|_{L^2(0,t;H^{\rho +\FEplus}_2)},\quad w\in\BH^t_{\rho, \FE} .
$$
For the globally fixed $t= T>0$, we omit the parameter and set 
$\mathbb{H}_{\rho, \FE} :=\mathbb{H}_{\rho, \FE}^T$
and denote the norm by
$$ \|w\|_{\BH_{\rho, \FE}}:= \|w\|_{\BH_{\rho, \FE}^T}.
$$
\hfill

\noindent Recall that $\FA=(\Omega,\CF, {\mathbb{F}},\mathbb{P})$ is a complete filtered probability
space with a right-continuous filtration ${\mathbb{F}}=\{{{\mathcal{F}}}_t:t\in [0,T]\}$,
carrying the Wiener processes $W_1$ and $W_2$.

\noindent Fix $\rho \in (-\frac{d}{2}, \frac{\aleph}{2}-\frac{d}{2})$.
The space of function-valued processes for the subsequent stochastic Schauder fixed point theorem
will be  $\CMM_{\MA, \Si}^2(0,T) $ 
defined by
\DEQS
\lqq{  \mathcal{M}^2_{\MA, \Sitilde}(0,T):=\Big\{(\eta,\xi) :\Omega\times[0,T]\times \CO\to\mathbb{R}\times \mathbb{R}: \quad \mbox{$\xi $ and $\eta$}
}
&& \\
&& \mbox{ are progressively measurable over $ {\mathfrak{A}}$ and}
\quad\mathbb{E}\int_0^T |\eta (s)|_{L^2}^{2}\, ds +\EE \|\xi\|_{\BH_{\rho, \FE}}^2<\infty\,
\Big\} ,
\EEQS
and being 
equipped with the norm
$$
\lvvvert
(\eta,\xi) \rvvvert_\cmm :=\Big\{ \mathbb{E}\int_0^T|\eta (s)|^{2}_{L^{2}}\, ds\Big\}^\frac 1{2}+\Big\{ \mathbb{E} \|\xi \|^{2}_{\BH_{\rho, \FE}}\Big\} ^\frac 1{2}, \quad(\eta,\xi)\in \CMM^2_{\MA, \Si}(0,T).
$$

\medskip

\item {\textbf{Definition of the nonlinear equation with the cut-off: }}
Let $\psi \in C_c^{\infty}(\mathbb{R})$ be a cut-off function which satisfies
$$
\psi(x):= \bcase 0 &\mbox{ if } |x|\ge 2,
\\ \in [0,1] &\mbox{ if } |x|\in(1,2),
\\1 &\mbox{ if } |x|\le 1,
\ecase
$$
and let $\psi_\n(x):= \psi(x/\n)$, $x\in\RR,\,\n \in \mathbb{N}$. For any pair of processes $(\eta,\xi) \in \mathcal{M}^2_{\MA, \Sitilde}(0,T)$
let us define for $t \in [0,T]$
\DEQSZ \label{h12.def}
h((\eta, \xi),t )&:=& \|\xi\|_{\mathbb{H}^t_{\rho, \FE}}, \qquad (\eta, \xi)\in \mathcal{M}^2_{\MA, \Sitilde}(0,T).
\EEQSZ
For clarity and simplicity, we omit the first (mute) component.
Finally, let $\phi_\kappa(t,(\eta, \xi)):=\psi_\n(h((\eta, \xi),t))$. 
We now state the cutoff nonlinear equation as follows:
\DEQSZ\label{nonlinearcutoffu}
\\\nonumber
d{\uk }(t)&=& [r_1 \,\Delta \uk (t)+a_1 \uk(t) + b_1 - c_1 \phi_\n((\uk, \vk),t)\cdot \uk(t)\cdot \,\vk^q(t)]\, dt +\sigma_1 \, g_{\gamma_1}(\uk (t))\, dW_1(t),\nonumber
\\
&\hfill \label{nonlinearcutoffv}
\\ \nonumber
d{\vk  }(t)&=& [r_2 \,A \vk  (t)+a_2 \vk  (t)+b_2+ c_2\phi_\n((\uk, \vk),t)\cdot \uk(t)\cdot \vk^q(t)]\, dt + \sigma_2 \,g_{\gamma_2}(\vk  (t)) \, dW_2(t), \nonumber
\hfill
\EEQSZ
with $\uk (0)=u_0$ and $\vk  (0)=v_0$ for initial data $(u_0, v_0)$ satisfying
\begin{equation}\label{initialinner}
\EE|u_0|_{H_2^{-2}}^2 <\infty,\qquad \EE|u_0|_{L^{p^*}}^{p^*} < \infty, \qquad \mbox{ and } \qquad \EE|v_0|_{H^{\rho}_2}^2 < \infty.
\end{equation}

\item {\textbf{Introduction of the fixed point operator and its invariant subset: }}

\noindent 
Firstly, we define the following subset of $\CMM_{\MA, \Si}^2(0,T)$.
Fix three constants $K_1,K_2, K_3>0$ and $\lambda>0$,
and introduce $\CK_{\MA}^\n(K_1,K_2,K_3)\subset \CMM_{\MA, \Si}^2(0,T)$ given by\footnote{{The space $\mathbb{H}_{0,2}$ is defined as $L^\infty(0,T;L^2)\cap L^2(0,T;H_2^{\frac{\aleph}{2}})$.}}
\del{\begin{align}\label{def_K_KKK}
\mathcal{K} _\MA:=	\CK_\MA(K_1,K_2,K_3) &:= \Big\{(\eta,\xi) \in \CMM_{\MA, \Si}(0,T)\mid \xi \geq0 \quad {\PP}\times \mbox{Leb}-\mbox{a.s.},\phantom{\Bigg|} 
\\ &\nonumber 
\mathbb{E}\,\|\eta\|^{2}_{\mathbb{H}_{0,2}}\leq K_1, \quad
\mathbb{E}\sup_{0< s\le T}e^{-\lambda s}|\eta(s)|_{L^{p^\ast}}^{p^\ast}\le K_2,\quad  {\rm and } \quad  \mathbb{E}\|\xi\|_{\mathbb{H}_{\rho,\aleph}}^2 \leq K_3
\Big\}.
\end{align}}
\begin{align}\label{def_K_KKK}
\lqq{ 	\mathcal{K} _\MA:=	\CK_\MA(K_1,K_2,K_3) := \Big\{(\eta,\xi) \in \CMM_{\MA, \Si}(0,T)\mid 
}
	\\ 	\qquad &\qquad \nonumber 
	\mathbb{E}\,\|\eta\|^{2}_{\mathbb{H}_{0,2}}\leq K_1, \quad
	\mathbb{E}\sup_{0< s\le T}e^{-\lambda s}|\eta(s)|_{L^{p^\ast}}^{p^\ast}\le K_2,\quad  {\rm and } \quad  \mathbb{E}\|\xi\|_{\mathbb{H}_{\rho,\aleph}}^2 \leq K_3
	\Big\}.
\end{align}

\noindent We define the mapping  $\CV^\kappa_{\MA,W}$ by
\DEQSZ\label{def_operator}
\mathcal{V}^\kappa_{\MA,W}: \mathcal{K}_\MA(K_1,K_2,K_3)  &\to& \CMM_{\MA, \Si}^2 (0,T),
\\ (\eta,\xi)&\mapsto &\mathcal{V}^\kappa_{\MA,W}[(\eta,\xi)],
\nonumber\EEQSZ
with   $\mathcal{V}^\kappa_{\MA,W}[(\eta,\xi)]=:(\uk,\vk)$,
where $(\uk, \vk)$ is the unique solution to the linear, decoupled system
\DEQSZ
d{\uk }(t)&= [r_1 \,\Delta \uk (t)+ a_1 \uk (t) + b_1 - c_1 \phi_\n(\xi,t)\cdot \eta(t) \cdot \,|\xi|^q(t)]\, dt +\sigma_1 \, g_{\gamma_1}(\uk (t))\, dW_1,\quad
 \nonumber\\
&\hfill \label{sysu0}\\
d{\vk  }(t)&= [r_2 \,A \vk  (t)+ a_2 \vk  (t)+ b_2+ c_2\phi_\n(\xi, t)\cdot \eta(t)\cdot |\xi|^q(t)]\, dt + \sigma_2 \,g_{\gamma_2}(\vk  (t)) \, dW_2(t), \nonumber\\
&\hfill\label{sysw0}
\EEQSZ
with initial conditions $\uk (0)=u_0$ and $\vk  (0)=v_0$. 
\begin{remark}\label{postiv}
	Here, we took for $\xi^q$ the modulos $|\xi|^q$. Later on, once we have verified the existence of a solution to \eqref{sysu0} and \eqref{sysw0},
	we show the non-negativity of the solution. Then, $\PP$-a.s. $\vk=|\vk|$.
\end{remark}
\item \textbf{Properties of the fixed point operator: }

\begin{stepinner}
\item {{\bf Wellposedness of the operator $\mathcal{V}^\n_{\MA,W}$ (Appendix~\ref{ss:wellposedness}): }}

Due to Proposition~\ref{reg_uk}, for any $\kappa\in\NN$, $K_1,K_2, K_3>0$ and for any
initial condition $(u_0,v_0)$ satisfying the constraints of Step (II),
for any $(\eta,\xi)\in \mathcal{K} _\MA^\n$
there exists a  unique pair of solutions  $\uk$ and $\vk$ to the system  \eqref{sysu0} --\eqref{sysw0}    such that $(\uk,\vk)\in\CM^{2}_\MA(0,T)$. 

\bigskip

\item {{\bf The operator $\mathcal{V}^\n_{\MA,W}$ maps $\CK_\MA^\n$ into itself (Appendix~\ref{s:invariance}): }}

In the next step, we will show that for any $\kappa\in\NN$ there exist numbers $K_1=K_1(\n),K_2=K_2(\n)$, and $K_3= K_3(\kappa)>0$
such that $\mathcal{V}^\n_{\MA,W}$ maps $\CK_\MA(K_1(\n),K_2(\n),K_3(\n))$ into itself.
Observe, we know from Proposition~\ref{reg_uk} that for all $\kappa$, there exists  constants $C_1,C_2>0$ such that we have for all $\kappa\in\NN$
\DEQSZ
\EE\sup_{0\le t\le T} |\uk(t)|_{L^2}^2+2\EE\int_0^T |\uk(t)|_{H^{1}_2}^2\, dt\le C(T)\lk( \EE|u_0|^2_{L^2}+C(\kappa)K^\frac 2{p^\ast}_2e^{2\lambda T/p^\ast}\rk)  . \label{e:trouble new} 
\EEQSZ
and
\DEQSZ
\EE\sup_{0\le t\le T} |\vk(t)|_{H^{\rho}_2}^2+2\EE\int_0^T |\vk(t)|_{H^{  \FE /2+\rho}_2}^2\, dt\le C(T)\lk(\EE|v_0|^2_{H^\rho_2} +C(\kappa) K^\frac 2{p^\ast}_2e^{2\lambda T/p^\ast }\rk). \label{e:vtrouble new}
\EEQSZ
From Proposition~\ref{limituk}, it follows that for any $\n\in\NN$, there exists a $\lambda>0$ (depending on $\n$) and two constants $C_1,C_2>0$, such that we have 
$$
\EE\sup_{0\le s\le T} e^{-\lambda s}|\uk(s)|_{L^{p^\ast}}^{p^\ast}\le C
\lk(1+e^{C_2T}\rk) |u_0|_{L^{p^\ast}}^{p^\ast}+\frac 12 \sup_{0\le s\le T} e^{-\lambda s}|\eta(s)|_{L^{p^\ast}}^{p^\ast}
$$
It follows for
\begin{align}\label{def_constants} &  K_1:=C(T)\lk( \EE|u_0|^2_{L^2}+C(\kappa)K^\frac 2{p^\ast}_2e^{2\lambda T/p^\ast}\rk),
	\\
&K_2:= 2\lk(1+e^{C_2T}\rk) |u_0|_{L^{p^\ast}}^{p^\ast}, \notag 
\\
& K_3:= C(T)\lk(\EE|v_0|^2_{H^\rho_2} +C(\kappa) K^\frac 2{p^\ast}_2e^{2\lambda T/p^\ast }\rk),\notag 
\end{align}
that the operator $\CV^\n_{\MA,W_1}$ maps the set $\CK_\MA^\kappa := \CK_\MA(K_1,K_2,K_3)$ into itself.
In the following, whenever we refer to $K_1,K_2$, and $K_3$,
we assume that the constants are given via \eqref{def_constants}.  Similarly, when we write $\CK_\MA^\kappa$,
we mean $ \CK_\MA(K_1,K_2,K_3)$.

\item {{\bf The operator $\mathcal{V}^\n_{\MA,W}$ is continuous from $\CK_\MA^\kappa$ into $\mathcal{M}^2_{\MA, \Sitilde}(0,T)$ (Appendix~\ref{ss:continuity}): }}

The continuity for the operator $\CV^\n_{\MA,W_1}:\CK_\MA^\kappa\to\mathcal{M}^2_{\MA, \Sitilde}(0,T)$ is shown in Proposition~\ref{contu}. In this step, it is essential that the argument of the cutoff function respects the norm of the underlying space.

\newcommand{\allp}{\nu}
\item \textsl{{\bf The operator $\mathcal{V}^\n_{\MA,W}$ maps the set  $\CK_\MA^\kappa
$ into a precompact set of $\mathcal{M}^2_{\MA, \Sitilde}(0,T)$ (Appendix~\ref{ss:compactness}): }}

Going back to the definition of $\mathcal{M}^2_{\MA, \Sitilde}(0,T)$ we see that
$$
\mathbb{X} = L^{2}(0, T; L^2(\CO))\times \mathbb{H}_{\rho, \aleph}^2,
$$
where $  \mathbb{H}_{\rho, \aleph}^2 = L^\infty(0,T; H^\rho _2(\CO))\cap L^2(0,T;H^{\rho +\FEplus}_2(\CO))$
is the underlying deterministic path space. Fix arbitrary constants $\allp, \eps\in (0,1)$. Any bounded subset in $\mathbb{X}'$ given by
$\mathbb{X}' = \mathbb{Y}^u_{\allp, \eps} \times \mathbb{Y}^v_{\allp, \eps}$, where
\begin{align*}
&\hspace{-1cm}\mathbb{Y}^u_{\allp, \eps} := L^2(0, T; H^{\eps}_{2}(\CO))\cap \mathbb{W}_2^\allp(0, T; H^{-2}_2(\CO)),\\
&\hspace{-1cm}\mathbb{Y}^v_{\allp, \eps} := L^2(0, T; H^{\rho+\frac{\aleph}{2}+\eps}_{2}(\CO))\cap
\mathbb{W}_2^\allp(0, T; H^{-2}_2(\CO))
\cap \mathcal{C}([0, T], H^{\rho}_2(\CO))\cap \mathcal{C}^{\allp}([0, T], H^{-2}_2(\CO)).
\end{align*}
Due to the Lemma of Prokhorov (see \cite[Theorem 2.2, p.\ 104]{MR838085})
combined with the Chebychev inequality, Theorem~\ref{th-gutman},
it remains to show that for any $K_1,K_2,K_3>0$
there exist  constants $C,\allp_0>0$ and Banach spaces $B_1,B_0,B_1'$, and $B_0'$, with $B_0\subset L^2(\CO) \subset B_1$, $B_0\hookrightarrow L^2(\CO)$ compactly embedded,
$B'_0\subset H^{\rho+\aleph/2}_2(\CO) \subset B_1'$, and  $B_0'\hookrightarrow H^{\rho+\aleph/2}_2(\CO) $ compactly embedded,
such that for any 
$\xi\in \CK_\MA(K_1,K_2,K_3) $, %
we have
$$
\EE \|u_\n\|_{L^2(0,T;B_0)},\EE \|u_\n\|_{\WW ^{\allp_0}_2 (0,T;B_1)}\le C\quad \mbox{and}\quad
\EE \|v_\n\|_{L^2(0,T;B'_0)},\EE \|v_\n\|_{\WW ^{\allp_0}_2 (0,T;B'_1)}\le C,
$$
where $(u_\n,v_\n)=\mathcal{V}^\kappa_{\MA,W}(\xi)$.
In addition, to get compactness also in $C_b^{(0)}(0,T;H^{\rho}_2(\CO))$ we have to show again a compact containment condition and a time regularity condition, see \cite[Theorem 7.2, p.\ 128]{MR838085} and  \cite{Aubin, Simon, Lions}, respectively, and Theorem~\ref{th-gutman2}.
To show the compact containment condition, we have to show that  for any $K_1,K_2,K_3>0$
there exists a Banach space
$B_0''$ with $ B_0''\hookrightarrow H^{\rho}_2(\CO) $ compactly, and a constant $C_0>0$ 
such that for any $t\in[0,T]$ 
and
all $\xi\in   \CK_\MA(K_1,K_2,K_3) $, 
we have
$$\sup_{0\le t\le T}\EE|v(t)-e^{t\Delta^{\aleph/2}}v_0|_{B_0''}\le C.
$$
Since $v_0$ is a $H^{\rho}_2(\CO)$--valued random variable, $H^{\rho}_2(\CO)$ a separable Banach space, the distribution of $v_0$ is tight and for all $\ep>0$ there exists a compact set $K_\ep\subset H^{\rho}_2(\CO)$ such that
$\PP\lk( v_0\not\in K_\ep\rk)\le \ep$. Since the semigroup is compact, $e^{-\Delta^{\aleph/2}t}$ maps bounded sets into compact sets.
Hence, we subtract $e^{\Delta^{\aleph/2}t}v_0$ and analyze the difference.
This procedure allows us to avoid higher regularities on the initial conditions.

To show the time regularity,
we have to show that
for any $K_1,K_2,K_3>0$,
there exist a  Banach space $B_1''$ with $H^{\rho}_2(\CO) \subset B_1''$ and constants  $C,\alpha_0>0$,
such that
$$\EE \|v_\n\|_{C^{(\alpha_0)}_b(0,T;B_1'')}\le C.
$$
Summarising the reasoning above, it is sufficient to prove the existence of some $\alpha$, $\eps$, $C>0$, and $m>0$ such that for all $(\xi, \eta) \in\CK^\n_\MA$
\begin{align*}
\mathbb{E}[\|(u_\kappa, v_\kappa)\|^m_{\mathbb{Y}^u_{\alpha, \eps} \times \mathbb{Y}^v_{\alpha, \eps}}]\leq C, \qquad \mbox{ where }\qquad (u_\kappa, v_\kappa) = \mathcal{V}^\n_{\MA,W}(\xi, \eta).
\end{align*}
This is shown in Proposition~\ref{semigroup_compact}.
\end{stepinner}
Note that Steps (IV) (a)-(d) correspond to the assumptions of the Schauder theorem.

\item {\textbf{Application of the stochastic Schauder fixed point theorem, \cite[Theorem 2.1]{HT22plus}: }}
For any $\kappa\in\NN$ and $(u_0, v_0)$ satisfying Assumtion~\ref{initialinner},
there exists a probability space $\tilde{\mathfrak{A}}_\n =(\tilde \Omega_\n ,\tilde{\CF}_\n ,\tilde{\mathbb{F}}_\n ,\tilde{\PP}_\n )$, a Wiener process $\tilde W^\n =(\tilde W^\n _1,\tilde W^\n _2)$ over $\tilde {\mathfrak{A}}_\n $, and elements
$$	
(\tilde u_{\kappa},\tilde v_{\kappa})\in\mathcal{K}_{\tilde{\mathfrak{A}}_\n}^\n,
$$
such that we have $\tilde \PP$-a.s.\ a fixed point
$$
\mathcal{V}_{\tilde{\mathfrak{A}}_\kappa, \tilde W^\n}^\n(\tilde u_{\kappa}(t),\tilde v_{\kappa}(t))=(\tilde u_{\kappa}(t),\tilde v_{\kappa}(t))
$$
for every $t\in[0,T]$.
Due to the construction of
$\mathcal{V}_{\tilde{\mathfrak{A}}_\n, \tilde W^\n}^\n$, the pair $(\tilde u_{\kappa},\tilde v_{\kappa})$
solves the system \eqref{nonlinearcutoffu}--\eqref{nonlinearcutoffv}
over the stochastic basis $\tilde{\mathfrak{A}}_\n $ with the Wiener noise $\tilde W^\n $.

\item   \textbf{Non-negativity of the solution:}
Later on, in Step (VII), the non-negativity of the solutions to system \eqref{nonlinearcutoffu}-\eqref{nonlinearcutoffv} is essential.  However, by a standard technique we prove in Proposition \ref{non_negativity}, that the solution is $\PP$-a.s. non-negative.

\item \textbf{Definition of the solution: independent glueing of the local solutions: }\label{ss: glueing}

\del{
The aim is here that by glueing to get a solution $\{(\buk, \bvk):k\in\NN\}$, where $(\buk, \bvk)$ solves the system
\DEQSZ\label{sysu_gesamt}
\\\nonumber
d{\buk  }(t)= [r_1 \,\Delta \buk  (t)+ a_1 \buk  (t)+ b_1 - c_1\phi_\n(\bvk,t)\buk (t)\vk^q(t)]\, dt +\sigma_1 \, g(\buk  (t))\, dW_1,\: \buk  (0)=u_0.
\EEQSZ
and $\bar v_\n$ is the unique solution to
\DEQSZ\label{sysw_gesamt}
\\\nonumber
d{\bvk  }(t)= [A \, \bvk  (t)+ a_2 \bvk  (t) + b_2+ c_2\phi_\n(\bvk ,t)\buk  (t)\bvk ^q(t)]\, dt + \sigma_2 \,g(\bvk  (t)) \, dW_2(t), \quad \bvk  (0)=v_0.
\EEQSZ}
\noindent By the following argument, we will show that
there exists a probabilitic weak  solution  of the system \eqref{nonlinearcutoffu}-\eqref{nonlinearcutoffv}.
That is, there exists a filtered probability space
$\bar{\mathfrak{A}} := (\bar\Omega, \bar{\mathcal{F}}, (\bar{\mathcal{F}}_t)_{t\geq 0}, \bar{\mathbb{P}})$,
Wiener processes $(W_1, W_2)$ over $\bar{\mathfrak{A}}$,
and a couple of processes $(u, v)$ being a strong solution
of the system \eqref{e:u1}-\eqref{e:v1} over $\bar{\mathfrak{A}}$.

\medskip

\noindent To do so, we construct a family of solutions $\{(\bar u_{\kappa} ,\bar v_{\kappa} ):\kappa\in\NN\}$  following the solution to the original problem until a certain stopping time ${\bar \tau}_\kappa$ defined in \eqref{stopp_time}. In particular,  we will introduce for each $\kappa\in\NN$ a new pair of processes $({\bar u}_{\kappa} , {\bar v}_{\kappa} )$ following the Gray-Scott system to the stopping time ${\bar \tau}_\kappa$. Besides, we will have
$({\bar u}_{\kappa} , {\bar v}_{\kappa} )|_{[0,{\bar \tau}_\kappa)}=({\bar u}_{\n +1},{\bar v}_{\n +1})|_{[0,{\bar \tau}_\kappa)}$.

\noindent Let us start with $\kappa =1$. From Step (V) 
we know that there exists a martingale solution consisting of a probability space $\tilde{\mathfrak{A}}_1=(\tilde{\Omega}_1,\tilde{\CF}_1,\tilde{\mathbb{F}}_1,\PP_1)$, two independent Wiener processes $(\tilde {W}^1_1,\tilde {W}^1_2)$ defined over  $\tilde{\mathfrak{A}}_1$, and a couple of processes $(\tilde u_1,\tilde v_1)$ solving $\tilde{\PP}_1$--a.s. \ the system
\DEQSZ\label{e:u1n}
\lk. \barray  && d\tilde u_1(t) = \left[r_1 \Delta \tilde u_1(t) +a_1 \tilde u_1(t) + b_1-c_1 \,\phi_1(\tilde v_1,t)\, \tilde u_1 (t)\,\tilde v_1^2(t)\right] dt+ \sigma_1 g_{\gamma_1}(\tilde u_1(t))\,d {\tilde W}^1_1(t),
\\
&& d\tilde v_1(t)=\left[r_2 A \tilde v_1(t)+a_2 \tilde v_1(t)+ b_2 + c_2\,\phi_1(\tilde v_1,t)\, \tilde u_1 (t)\,\tilde v_1^2(t)
\right]\,dt+\sigma_2 g_{\gamma_2}(\tilde v_1(t))\, d {\tilde W}^1_2(t),
\\ &&(\tilde u_1(0),\tilde v_1(0))=(u_0,v_0).
\earray\rk. \EEQSZ
Let us define now the stopping time
$\tilde \tau_1^\ast:=\inf \{s\ge 0\;\colon\; \|\tilde v_1\|_{\mathbb{H}_{\rho}^{s}} \ge 1 \}$  on $\tilde{\mathfrak{A}}_1$.
Observe, on the time
interval $[0,\tilde \tau_1^\ast)$, the pair $(\tilde u_1,\tilde v_1)$ solves the system given in
\eqref{e:u1n}. 
Now, we define a new pair of processes $(\bar u_1,\bar v_1)$  following $(\tilde u_1,\tilde v_1)$  on $[0,\tilde \tau_1^\ast)$ and extend this processes to the whole interval $[0,T]$ in the following way.
First, we put $\overline{\mathfrak{A}}_1:=\tilde{\mathfrak{A}}_1$ and $\bar {{W}}_j^1:={\tilde W}_j^1$, $j=1,2$.
Then let us introduce the processes $y_1$ and $y_2$
being a strong solution over $\overline{\mathfrak{A}}_1$ to
\DEQSZ
d y_1(t, \tilde u_1(\tilde \tau^\ast_1),\sigma) &=&  r_1\Delta y_1(t, \tilde u_1(\tilde \tau^\ast_1),\sigma)\,dt + \sigma_1 y_1(t, \tilde u_1( \tilde \tau^\ast_1),\sigma) \,d(\theta_{\sigma} \bar{{W}}^1_1)(t)
\nonumber \\ \label{eq1} y_1(0,\tilde u_1(\tilde \tau^\ast_1))&=&\tilde u_1(\tilde \tau^\ast_1){}
,\quad t\ge 0 ,\\ \nonumber
y_2(t, \tilde  v_1( \tilde \tau^\ast _1),\sigma) &=&e^{r_2\Delta t} \tilde v_1(\tilde \tau^\ast_1) 
\\
\lqq{{} + \int_{0}^t  e^{(r_v\Delta )(t-s)} \sigma_2 y_2 (s, \tilde v_1(\tilde \tau^\ast _1),\sigma) \,d(\theta_{\sigma} \bar{{W}}^1_2)(s)
,\quad t\ge 0.} && \label{eq2}
\EEQSZ
Here, $\theta_\sigma$ is the shift operator which maps ${\bar W}_j(t)$ to ${\bar W}_j(t+\sigma)-\bar W_j(\sigma)$.
Since the couple $(\tilde u_1,\tilde v_1)$  is continuous in $H_2^{-1}(\CO)\times H_2^\rho(\CO)$, we know that  $\tilde u_1(\tilde \tau^\ast_1)$ and $\tilde v_1(\tilde \tau^\ast_1)$ are well defined random variables belonging $\bar \PP_1$-a.s.\ to $L^2(\CO)$ and $H_2^\rho(\CO)$, respectively.
By \cite[Theorem 2.5.1]{BDPR2016} the existence of a unique solution $y_1$ over $\overline{\mathfrak{A}}_1$ to \eqref{eq1} in $H^{-1}_2(\CO)$ is given. Since $(e^{t(r_v\Delta -\alpha \operatorname{Id})})_{t\ge 0}$ is an  analytic semigroup on  $H_2^\rho(\CO)$ for $\alpha>0$,
the existence of a unique solution $y_2$  over $\overline{\mathfrak{A}}_1$ to \eqref{eq2} in $H_2^\rho(\CO)$
can be shown by standard methods, cf. \cite{DaPrZa}. Furthermore, verifying that the initial conditions' assumptions are satisfied is straightforward.
Now, let us define two  processes $\bar u_1 $ and $\bar v_1$
being identical to $\tilde u_1$ and $\tilde v_1$, respectively,  on the time interval $[0,\tilde \tau^\ast_1)$ and
following the heat equation (with lower order terms) with noise and without nonlinearity, i.e.,   $y_1(\cdot,\tilde u_1(\tilde \tau^\ast_1),\tilde \tau^\ast_1)$ and $y_2(\cdot,\tilde v_1(\tilde \tau^\ast_1),\tilde \tau^\ast_1)$, afterwards.
In particular, let
$$
\bar u_1  (t) = \bcase \tilde u_1(t) & \mbox{ for } 0\le t< \tilde \tau^\ast_1,\\
y _1(t,\tilde u_1(\tilde \tau^\ast_1),\tilde \tau^\ast_1) & \mbox{ for } \tilde \tau^\ast_1\le  t \le T,\ecase
$$
and
$$
\bar v_1  (t) = \bcase \tilde v_1 (t) & \mbox{ for } 0\le t< \tilde \tau_1^\ast,\\
y_2 (t,\tilde v_1(\tilde \tau^\ast_1),\tilde \tau_1^\ast
)  & \mbox{ for } \tilde \tau_1^\ast \le  t \le T.\ecase
$$
Let us now construct the probability space and the processes for the next time interval.

Let $\MA = (\Omega, \mathcal{F}, \mathbb{F}, \mathbb{P})$ be a filtered
probability space where two Wiener processes $\tilde W_1^1, \tilde W_2^1$, cylindrical on $H$, are given, such that
$\mathcal{F}_0 = \tilde{\mathcal{F}}_{\tilde \tau_1}$. Let $u_0 = \tilde u(\tilde \tau_1^\ast)$ and $v_0 = \tilde v(\tilde \tau_1^\ast)$.
Observe, $u_0$ and $v_0$ are $\mathcal{F}_0$-measurable random variables satisfying Assumption \ref{initialinner}. Then, due to Step (V), there exists a martingale solution of system \eqref{nonlinearcutoffu} - \eqref{nonlinearcutoffv} for $\kappa = 2$.
That is, 
there exists a filtered probability space $\tilde{\mathfrak{A}}_2=(\tilde{\Omega}_2,\tilde{\CF}_2,\tilde{\mathbb{F}}_2,\tilde{\PP}_2)$, and
a pair of independent Wiener processes $({\tilde W}_1^2,{\tilde W}_2^2)$,
a couple of processes $(\tilde u_2,\tilde v_2)$ solving $\tilde{\PP}_2$-a.s.\ the system
\DEQS
\lk\{ \barray  d\tilde u_2(t) &=&\left[r_1 \Delta \tilde u_2(t) - c_1 \,\phi_2(\tilde v_2,t)\, \tilde u_2 (t)\,\tilde v_2^q(t)\right] dt+ \sigma_1 g_{\gamma_1}(\tilde u_2(t)) \, d\tilde{W}^2_1(t),
\\[0.5cm]
d\tilde v_2(t)&=&\left[r_v \Delta \tilde v_2(t)+\gamma_2\phi_2(\tilde v_2,t)\, \tilde u_2 (t)\,\tilde v_2^q(t)
\right]\,dt+\sigma_2 g_{\gamma_2}(\tilde v_2(t))\, d \tilde{W}^2_2(t),
\\[0.5cm]
(\tilde u_2(0),\tilde v_2(0))& = & (u_0, v_0).
\earray\rk.\EEQS
Let us define now the stopping times on $\tilde{\mathfrak{A}}_2$,
$$\tilde{\tau}_2^\ast:=\inf \{s\ge 0: \|\tilde v_2\|_{\mathbb{H}_{\rho}^{s}} \ge 2 \},$$
and $\overline{\mathfrak{A}}_1:=\tilde{\mathfrak{A}}_1$ and $\bar {{W}}_j^1:={\tilde W}_j^1$, $j=1,2$,
and in particular $\overline{\mathbb{F}}_1:=(\tilde{\CF}^1_t)_{t\in [0,T]}$.
We continue by defining $\overline{\Omega}_2:=\overline{\Omega}_1\times\tilde{\Omega}_2$, $\overline{\CF}_2:=\overline{\CF}^1\otimes \tilde{\CF}^2$, $\overline{\PP}_2:=\overline{\PP}_1\otimes\tilde{\PP}_2$ and set
$\overline{\mathbb{F}}_2:=(\overline{\CF}^2_t)_{t\in [0,T]}$, where
$$\overline{\CF}^2_t:=\bcase \overline{\CF}^1_t, & \mbox{if} \quad t<\tilde{\tau}_1^\ast,
\\ \tilde{\CF}^2_{t-\tau_1^\ast}, & \mbox{if}\quad  t\ge \tilde{\tau}_1^\ast.
\ecase
$$
Let $\overline{\mathfrak{A}}_2:=(\overline{\Omega}_2,\overline{\CF}_2,\overline{\mathbb{F}}_2,\overline{\PP}_2)$.
Finally, let us set for $j=1,2$
$$\bar{{W}}_j^2(t):=\bcase \bar{{W}}^1_j(t), & \mbox{if} \quad t<\tilde \tau_1^\ast,
\\   {\tilde W}^2_j({t-\tilde \tau_1^\ast})+\bar{{W}}^1_j(\tilde \tau_1^\ast), & \mbox{if}\quad  t\ge \tilde \tau_1^\ast,
\ecase
$$
which gives independent Wiener processes for $j=1,2$, w.r.t. the filtration $\overline{\mathbb{F}}_2$.

Now, let us define two  processes $\bar u_2 $ and $\bar v_2$
being identical to $\bar u_1$ and $\bar v_1$, respectively,  on the time interval $[0,\tilde\tau^\ast_1)$, being identical to $\tilde u_2$ and $\tilde v_2$  on the time interval $[\tilde \tau^\ast_1,\tilde \tau^\ast_1+\tilde \tau_2^\ast)$ and
following the heat equation  (with lower order terms) with multiplicative noise afterwards.
Let us note that for any initial condition having distribution equal to $\bar u_2(\tilde \tau_2^\ast)$ and $\tilde v_2(\tilde \tau^\ast_2)$ there exists a strong solution $y_1(\cdot ,\cdot ,\tilde \tau^\ast_2+\tilde \tau^\ast_1)$ and $y_2(\cdot,\cdot ,\tilde \tau^\ast_2+\tilde \tau^\ast_1)$ of the systems \eqref{eq1} and \eqref{eq2}, respectively, on $\overline{\mathfrak{A}}_2$.
Let for $t\in[0,T]$
$$
\bar u_2  (t) = \bcase \bar u_1(t) & \mbox{ for } 0\le t< \tilde \tau^\ast_1,\\
\tilde u_2(t-\tilde \tau_1^\ast ) & \mbox{ for }  \tilde \tau^\ast_1\le t\le \tilde \tau^\ast_1+\tilde \tau_2^\ast,\\
y _1(t-(\tilde \tau^\ast_1+\tilde \tau^\ast_2),\tilde u_2(\tilde \tau^\ast_2),\tilde \tau^\ast_1+\tilde \tau^\ast_2) & \mbox{ for } \tilde \tau^\ast_2+\tilde \tau^\ast_1\le  t \le T,\ecase
$$
$$
\bar v_2  (t) = \bcase \bar v_1 (t) & \mbox{ for } 0\le t< \tilde \tau_1^\ast,\\
\tilde v_2(t-\tilde \tau_1^\ast ) & \mbox{ for } \tilde \tau^\ast_1\le t\le \tilde \tau^\ast_1+\tilde \tau_2^\ast,\\
y_2 (t-(\tilde \tau^\ast_1+\tilde \tau^\ast_2),v_2(\tilde \tau^\ast_2),\tilde \tau_1^\ast+\tilde \tau^\ast_2
)  & \mbox{ for } \tilde \tau_1^\ast+\tilde \tau^\ast_2 \le  t \le T.\ecase
$$
In the same way we will construct for any $\kappa \in\NN$ a probability space $\overline{\mathfrak{A}}_\kappa $, a pair of independent Wiener processes $(\bar{{W}}^1_\kappa ,\bar{{W}}^2_\kappa )$, over $\overline{\mathfrak{A}}_\kappa $ and a pair of processes $(\bar{u}_\kappa,\bar{v}_\kappa)$ starting at $(u_0,v_0)$ and solving system  \eqref{e:u1}-\eqref{e:v1} up to time
%
%
%
\DEQSZ\label{stopp_time}\bar \tau_\kappa :=\tilde{\tau}_1^\ast+\cdots +\tilde{\tau}_\kappa ^\ast
\EEQSZ and following the heat equation afterwards.
Besides, we know that.
$$(\bar u_{\kappa} ,\bar v_{\kappa} )|_{[0,\bar \tau _{\kappa -1})}=(\bar u_{\kappa -1},\bar v_{\kappa -1})|_{[0,\bar \tau_{\kappa -1})}.
$$

\medskip

So far the space and the function $(\bar u_\kappa, \bar v_\kappa)$
were defined with the help of the parameter $\rho$, which provides a local solution.
In order to obtain uniform bound independent of $\kappa$ on $(\bar u_\kappa, \bar v_\kappa)$
we obtain further restrictions. For that purpose, we introduce a new parameter $\alpha \geq \rho$, which replaces $\rho$, satisfying additional restrictions and which implies new restrictions on the nonlinearity $q$.

\item {\bf A uniform bound on $\mathbb{E}\| \bvk  \|_{\mathbb{H}_\alpha}^{2}$ with respect to $\kappa$ (Appendix~\ref{s:uniform}): }

The parameters $q$, $\aleph$, $d$ are given. The parameters
$p^*$ and $\rho$ chosenin Step (I).
Now, if
 $(q,\FE,\alpha)$ and $p$ satisfy
$$
q<\frac {2d}{2d-\FE}, \quad \alpha < \frac \aleph 2  -\frac{d}{2}
, \quad p\in(2,2.25),
$$
or $d=2$, $\FE=2$, and $q=2$ is satified,
then,
in Proposition~\ref{propvarational} 
we have shown that
there exists a constant $C>0$ such that for all $\kappa\in\NN$ we have
 for any $r^*\geq 1$ there exists positive constant $C>0$ such that for all $\kappa$ 
\DEQS
\mathbb{E}\| \vk \|_{\mathbb{H}_{\alpha, \aleph}}^{r^\ast}
&\leq C\bigg(1 +  \mathbb{E}|\Delta^ {\frac{\alpha}{2} }v_0|_{L^2}^{2r^\ast}
+ \Big(\EE|u_0|_{L^p}^{\frac{r^\ast}{p-q(p-1)}}\Big)^2\bigg)^\frac{1}{2}.
\EEQS
However, for the next step, we have to ensure that we have additionally  $\alpha\ge \rho$.

\medskip

\item {\bf Extension to a global  solution: }
Let us define 
$$
\bar \Omega := \lim\limits_{k\to\infty} \bar \Omega_\kappa, \qquad \bar{\mathcal{F}}_t := \lim\limits_{\kappa \to \infty} \bar{\mathcal{F}}^\kappa_{t}, \quad t\in [0, T], \qquad  \qquad \bar\PP := \lim\limits_{\kappa \to\infty} \bar \PP_\kappa.
$$
and the Wiener processes, $j=1,2$,
$$\bar{W}_j(t) := \lim\limits_{\kappa \to \infty} \bar{W}_j^{\kappa}(t).$$
Observe that $(\bar W_j(t))_{t\in [0, T]}$, $j=1,2$, is a Wiener processes with
respect to the filtration $\bar{\BF} := (\bar{\mathcal{F}}_t)_{t\in [0, T]}$
over $\bar{\mathfrak{A}} = (\bar{\Omega}, \bar{\mathcal{F}}, \bar{\BF}, \bar\PP)$.
Define the $\bar{\BF}$-stopping time
$$
\bar \tau_\kappa^\ast:=\inf \lk\{0\le s\le T~|~ \| \bvk \|_{\mathbb{H}^s_{\alpha, \aleph}}\ge \kappa\rk\}.
$$
Note, that on the time interval $[0,\bar \tau_\kappa^\ast]$ the processes $(\buk,\bvk )$ solves the system \eqref{e:u1}-\eqref{e:v1}.
By the uniform bound we show that for all $T>0$ with probability $1$ we have  $\lim\limits_{\kappa\to\infty}\bar \tau_k^\ast\ge T$.

\medskip
\noindent For any $\kappa\in\NN$, let us define the set
$$A_\kappa :=
\lk\{ \omega\in \bar{\Omega}\;\colon\;\bar \tau^\ast_\kappa(\omega)\ge T\rk\}.
$$
From before, we know that there exists a progressively measurable process $(\bar u,\bar v)$, $\bar u := \lim\limits_{\kappa \to\infty} \bar u_\kappa$ and $\bar v := \lim\limits_{\kappa \to\infty} \bar v_\kappa$ over $\bar{\mathfrak{A}}$
such that $(\bar u,\bar v)$ solves $\bar \PP$--a.s. the system \eqref{e:u1}-\eqref{e:v1} up to time $T$ on the set $A_\kappa$. In particular, we have the conditional probability
\begin{align}\label{e:conditionalsol}
\PP\lk( \{ \mbox{there exists a solution $(\bar u,\bar v)$ to  \eqref{e:u1}-\eqref{e:v1} } \} \mid A_\kappa\rk)=1.
\end{align}
Set $\bar \Omega_0=\bigcup_{\kappa=1}^{\infty} A_\kappa$.
Then, since by construction $A_\kappa \subset A_{\kappa+1}$, we have
\DEQS\lqq{
\qquad\bar\PP\lk( \lk\{ \mbox{there exists a solution $(\bar u, \bar v)$ to \eqref{e:u1}-\eqref{e:v1} }
\rk\}\cap\bar{\Omega}_0\rk) }
\\
&=& \lim_{\kappa\to \infty} \bar \PP\lk(  \{ \mbox{there exists a solution $(\bar u,\bar v)$ to \eqref{e:u1}-\eqref{e:v1}} \} \cap A_\kappa \rk)
\\
&=& \lim_{\kappa\to \infty} \bar \PP\lk(  \{ \mbox{there exists a solution $(\bar u,\bar v)$ to \eqref{e:u1}-\eqref{e:v1}} \} \mid A_\kappa \rk)\cdot \bar \PP\lk( A_\kappa\rk).
\EEQS
Due to \eqref{e:conditionalsol}
it remains to show that\
${\lim\limits_{\kappa\to\infty}}$ $\PP( A_\kappa)=1$. In this case, $A_\kappa\subset A_{\kappa+1}$ for $\kappa\in\NN$, implies that
$$\PP\lk( \lk\{ \mbox{there exists a solution $(\bar u,\bar v)$ to \eqref{e:u1}-\eqref{e:v1}}
\rk\}\cap\bar{\Omega}\rk)=1.
$$
However, from Step (VIII), we know there exists a ($\kappa$-independent) constant $C(T)>0$ such that we have
$$
\EE\left[ \|\bvk \|_{\mathbb{H}^T_{\alpha, \aleph}}\right]\le C(T),\quad \kappa\in\NN,
$$
%
thus, we have the Markov inequality and the definition of  $\bar \tau^\ast_\n$
$$\PP\lk(\bar \Omega_0\setminus  A_\kappa\rk) \le \PP\lk(\bar \tau^\ast_\n\ge T\rk)\le \frac {C(T)}{\kappa }\to 0,\qquad \mbox{ as }\kappa \to \infty.
$$
Thus the solution process is well defined on $\bar \Omega=\bigcup_{\kappa=1}^{\infty} A_\kappa$ with $\PP(\bar \Omega)=1$.
This finishes the proof. $\square$

\end{step}
\section*{Acknowledgements} 

\noindent EH gratefully acknowledges the support of the Austrian Science Foundation, Project number: P 34681. 
MAH thanks the Austrian Academy of Sciences for funding the academic visit at Montanuniversit\"at Leoben in June / July 2022 in the framework of a JESH project 2019.

\newpage 
\appendix

\bigskip
\section{\textbf{Technical preliminaries on the noise and the stochastic integral}}\label{Appendix_A_noise}

\noindent {This section is dedicated to gather the main results for the stochastic integral and the stochastic convolution. The final goal is to obtain a maximal inequality of stochastic convolution process.} 

\medskip 

\paragraph{\textbf{Stratonovich v. It\^{o} noise: }} 
The stochastic integral in the system \eqref{e:u1}--\eqref{e:v1} is interpreted in the Stratonovich sense due to the physical interpretation as the limit of a real fluctuating forcing term with a finite amplitude and finite timescale.
To compensate for the lack of martingale property 
the proof is given for the It\^ o stochastic integral with the additional quadratic variation operator $\frac 12 g(u)Dg(u)\,dt$ (for the definition of the linear operator $g$ see \eqref{g_sr}). Due to the linearity and boundedness of $\frac 12 g(u)Dg(u)\,dt$ under the assumption of the main theorem, it is omitted and covered by the linear terms $a_1 u$ and $a_2 v$ up to a small modification. 

\medskip

\paragraph{\textbf{The cylindrical Wiener processes $W_j$, $j=1,2$, in $H = L^{2}(\oO)$: }} 
Let us define the Wiener processes by introducing  the eigenfunction 
of the Laplace operator $-\Delta$ on $H = L^2(\CO)$ being denoted by 
$\{ \varphi_k:k\in\NN\}$ and by $\{\lambda_k:k\in\NN\}$ the associated eigenvalues.
We refer to \cite{mechtild} for a more detailed introduction.
In case $\CO$ is an arbitrary bounded set with a smooth boundary, 
the eigenfunctions are constructed by a tensor product (see \cite{BDPR2016}) 
and the enumeration is chosen in increasing order, counting the multiplicity. 
In this case, the following estimate of the asymptotic behaviour of the eigenvalues is known:
there exist two numbers $0 < c <C$ such that
\DEQSZ\label{EVsup}
c \, k^\frac 2d \le \lambda_k\le C\,  k^\frac 2d, \quad  \mbox{ for all }k\in\NN.
\EEQSZ
In addition, there exists some constant $c_0>0$ such that
\DEQSZ\label{EFsup}
\sup_{x\in\CO} |\varphi_k(x)|\le c_0\,\lambda _k^\frac {d-1}2,\quad k\in\NN.
\EEQSZ
See \cite{BDPR2016} p.7 for details. 
Since the driving Wiener processes $W_1$ and $W_2$ are cylindrical Wiener processes 
in $H$ by hypothesis, the spectral decomposition theorem yields
\DEQSZ\label{noise_sr}
W_j(t) &=&\sum_{k=1}^\infty \varphi_k\,\mathbf{w}^j_k(t),\quad t\ge 0,\,\, j=1,2,
\EEQSZ
where $\{\mathbf{w}^j_k:k\in\NN\}$, $j=1,2$, are two mutually independent 
families of i.i.d. standard scalar Wiener processes. Note that all stochastic processes 
are considered on an appropriate filtered probability space $\MA = (\Omega, \CF, \mathbb{P}, \mathbb{F})$ 
satisfying the usual conditions in the sense of \cite{Protter}.

\paragraph{\textbf{The multiplication operator: }}
We apply a so-called linear noise to our system and compare \cite[Example 2.1.2]{BDPR2016}.
To be more precise,  for $u, h\in H$ and $\gamma>0$ being sufficiently large 
let us define the linear mapping $g$ given by
\DEQSZ\label{g_sr}
\qquad g_\gamma (u) [h] &:=&  u \cdot (-\Delta)^{-\gamma/2} h=\sum_{k=1}^\infty \lambda _k^{-\gamma/2}  \la h,\varphi_k\ra (u\cdot \varphi_k) = \sum_{k=1}^\infty \lambda _k^{-\gamma/2}(u\cdot \varphi_k) \, \varphi_k\otimes\varphi_k [h],
\EEQSZ
\noindent where $\cdot$ denotes the scalar multiplication of two real-valued functions.
\paragraph{\textbf{The stochastic integral: }}
For a separable Banach space $E$ let $\Upsilon(H,E)$ be the space of all 
$\Upsilon$--radonifying operators from $H$ to $E$. We refer to \cite{Jan1} for details. 
Recall that this space is the closure 
of the space of finite rank operators from $H$ to $E$ with respect to the norm
$$
\Big|\sum_{k=1}^N \varphi_k\otimes x_k\Big|^2_{\Upsilon(H,E)}:= \EE \Big|\sum_{k=1}^N \gamma_k x_k\Big|^2_E,
$$
where $\{\varphi_k:k=1,\ldots ,N\}$ is a family of  orthonormal functions in $H$ and  $\{x_k:k=1,\ldots,N\}$ is a subset of $ E$. Finally,  $\{\gamma_k:k=1,\ldots,N\}$ is a family of independent standard Gaussian random variables. 

\noindent Let $E$ be a Banach space of $M$--type $2$. 
Let us denote by $\CM^2_\MA(0,T;E)$ 
the space of all $\mathbb{F}$-adapted stochastic processes $\xi:[0,T]\times \Omega\to E$ 
with
$$
\mathbb{E}\lk( \int_0^T|\xi(s)|^2_{E} {}ds\rk)< \infty.
$$
{For any $\xi \in  \CM^2_\MA(0,T;E)$, and $j=1,2$ let us define the process $M_j=\{M_j(t):t\in [0,T]\}$ given by
\DEQS
M_j(t)=\int_0^tg_\gamma(\xi(s))dW_j(s) 
:= \sum_{k=1}^{\infty} \lambda_k^{-\gamma/2}\varphi_k\,
\int_0^t \xi (s) d\mathbf{w}^j_k(s)
, \qquad t\in [0,T] .
\EEQS
If $E$ is a Hilbert space, the It\^{o} isometry reads for some constants $C(\gamma)>0$  
\DEQSZ\label{ito_isometry}
\mathbb{E}\lk(\lk|\int_0^tg_\gamma(\xi(s)){} d W_j(s) \rk|_E^2 \rk) = C(\gamma) \mathbb{E}\lk( \int_0^t|g(\xi(s))|^2_{\Upsilon(H,E)} {}d s\rk) , \qquad t\in[0,T].
\EEQSZ
If $E$ is a Banach space of $M$--type $2$, then we get the following inequality \cite{brzezniak} 
\DEQSZ\label{mtypep}
\mathbb{E}\lk(\lk|\int_0^tg_\gamma(\xi(s)){} d W_j(s) \rk|_E^2 \rk) \le  C(\gamma) \mathbb{E}\lk( \int_0^t|g(\xi(s))|^2_{\Upsilon(H,E)} {}d s\rk) , \qquad t\in[0,T].
\EEQSZ

%
\paragraph{\textbf{Maximal inequalities and the reduction to the $\Upsilon$-radonifying norm: }} 
Note that for $\varsigma\ge 1$  we have the Burkholder--Gundy--Davis inequality (see \cite[Thm 4.7]{vanneerven}) 
\DEQSZ \label{burkholder}
\mathbb{E} \lk(\sup_{0\le s\le t}\lk| \int_0^sg_\gamma(\xi (r)){}d W(r)\rk|_E^\varsigma \rk)
\leq C_{\varsigma}\, \mathbb{E}\lk( \int_0^t |g_\gamma(\xi(s))|_{\Upsilon(H, E)}^2 ds \rk)^\frac{\varsigma}{2}.
\EEQSZ
Consequently, to obtain estimates of the stochastic integral, it is sufficient to calculate 
the $\Upsilon$--radonifying norm (in short $\Upsilon$-norm) of the process
$[0,T]\ni s\mapsto g_\gamma(\xi(s))$, i.e.\ to calculate $|g_\gamma(\xi(s))|_{\Upsilon(H, E)}$ 
for some $M$--type $2$ Banach space $E$ and $s\in [0, T]$. 

\noindent Note that by \cite[Theorem 3.20]{vanneerven2} 
we have the following representation of the $\Upsilon$-norm 
\[
|g_{\gamma_1}(u)|_{\Upsilon(H,E)}^2 =  \EE \Big|\sum_{k} \gamma_k g_{\gamma_1}(u)[\varphi_k]\Big|_E^2,
\]
as long as $g_{\gamma_1}(u)[\varphi_k]\in E$.

\medskip

\paragraph{\textbf{The $\Upsilon$-radonifying norm for the inhibitor term~$u$: }}
To show existence and boundedness for the process $u$, we have to estimate the $\Upsilon$-norm 
on some $L^p$-space $E$, $p\ge2$, depending on the choice of $\gamma=\gamma_1$. 

\noindent \textbf{In the special case of $E=H = L^2(\CO)$} the $\Upsilon$--norm coincides with the respective Hilbert-Schmidt norm.
Due to  H\"older's inequality and the Sobolev embedding or alternatively \cite[Theorem~1~(iii), p.\ 190/191]{runst} (see also Subsection~\ref{ss:Besov}) we have a positive constant $C_R$ such that for all $u, \varphi\in H$ we have 
\[|u\varphi|_{H}\le C_R |u|_{H^{\delta_1}_2}|\varphi|_{H^{\delta_2}_2}\qquad \mbox{ for any }\quad \delta_1,\delta_2\ge 0$, $\delta_1+\delta_2\ge \frac d2.\]
Now, we can write
	\DEQSZ \label{eq: Hil-Schi}
|g_{\gamma_1}(u)|_{\Upsilon(H,H)}^2 &=&\sum_{l}^\infty | g_{\gamma_1}(u)\varphi_
l|^2_{H}
=  \sum_{l,k}^\infty \lambda _k^{-\gamma_1} |u \varphi_k \la \varphi_l, \varphi_k\ra |^2_{H}
\\
&=&
\sum_{k}^\infty \lambda _k^{-\gamma_1} | u \varphi_k |_{H}^2 \nonumber
\le
C_R^2 |u|_{H^{\delta_1}_2} ^2 \, \sum_{k}^\infty |\varphi_k |_{H^{\delta_2}_2}^2 \lambda _k^{ -2\frac {\gamma_1}2}
\\
&\le&
C_R^2 |u|_{H^{\delta_1}_2} ^2 \, \sum_{k}^\infty  \lambda _k^{ \delta_2-\gamma_1}
\le
C_R^2\,|u|_{H^{\delta_1}_2} ^2 \, \sum_{k}^\infty k^{\frac 2d(\delta_2-\gamma_1)}
,\nonumber
\EEQSZ
which is finite for $\gamma_1>\delta_2+\frac d2$. Using interpolation and 
the Young inequality, we know that for any $\delta_1\in(0,1)$
and any $\ep>0$ there exists a positive constant $C(\ep, \delta_1)$ such that
\begin{equation}\label{e:interp0}
|u|_{H^{\delta_1}_2} ^2 \le \ep |u|_{H^1_2} ^2 + C(\ep, \delta_1)|u|_{H} ^2 .
\end{equation}
Combining \eqref{eq: Hil-Schi} with \eqref{e:interp0}, we obtain 
\DEQSZ \label{eq: Hil-Schi2}
|g_{\gamma_1}(u)|_{\Upsilon(H,H)}^2&\le&  \ep |u|_{H^1_2} ^2 + C(\ep,\gamma_1,\delta_1)|u|_{H} ^2 .
\EEQSZ 
Note, that $\delta_1<1$ implies that $\delta_2>\frac d2-1$ and $\gamma_1>d-1$.
Our aim is to give an estimate of
$$
\sup_{0\le s\le t}\lk|\int_0^sg_{\gamma_1}(\xi (r)){}d W_1(r)\rk|_{H}
.$$
Fix $\varsigma\ge 1$.
By the Burkholder-Davis-Gundy inequality \eqref{burkholder} combined with \eqref{eq: Hil-Schi2} and Theorem~4.7 from \cite{vanneerven}
it follows that  for any $\eps>0$ there exists $C( \eps, \gamma_1)>0$ such that we have for all $ t\in[0,T]$
\begin{align}
&\mathbb{E} \lk(\sup_{0\le s\le t}\lk| \int_0^sg_{\gamma_1}(\xi (r)){}d W_1(r)\rk|_{H}^\varsigma  \rk)
\leq C(\varsigma, \gamma_1)  \EE\bigg[\Big(\eps\int_0^t |\xi(s)|_{{H^{1}_2}}^{2}ds  + C(\eps) \int_0^t |\xi(s)|_{{H}}^2 ds \Big)^\frac{\varsigma }{2}\bigg]\nonumber ,
\end{align}
and by straightforward calculations, i.e.\ interpolation and again the Young inequality,  that  for any $	\eps>0$ there exists $C(\eps, \varsigma, \gamma_1)>0$ such that we have for all $ t\in[0,T]$
\begin{align}
&\mathbb{E} \lk(\sup_{0\le s\le t}\lk| \int_0^sg_{\gamma_1}(\xi (r)){}d W_1(r)\rk|_{H}^\varsigma  \rk) \nonumber\\
&
\leq \eps \, \EE\bigg[\Big(\int_0^t |\xi(s)|_{H_2^{1}}^{2}ds\Big)^{\frac{\varsigma }{2}}\bigg]   +  
C(\eps, \varsigma, \gamma_1) \EE\bigg[\Big( \int_0^t |\xi(s)|_{H}^2 ds \Big)^\frac{\varsigma }{2}\bigg]\label{Hrhoburkholderl2}
.
\end{align}

\noindent \textbf{In case of $E=L^p(\CO)$}, observe first, that we have by H\"older's inequality and interpolation for all $\delta_1,\delta_2\ge 0$ with $\delta_1+\delta_2\ge \tfrac d p$, the inequality $|u\cdot \varphi|_{L^p}\le C|u|_{H^{\delta_1}_p}|\varphi|_{H^{\delta_2}_p}$.
Applying the Rosenthal inequality for exponent $t= 2$ as in \cite{Ibragimova}, \eqref{EVsup}, and \eqref{EFsup}, we get positive constants $C$ (which may vary from line to line) 
\DEQSZ \label{eq:gamma1}
| g_{\gamma_1}(u)|_{\Upsilon(H,L^p)} ^2 
&=& \EE |\sum_{k}^\infty \gamma_k  \lambda_k^{-\gamma_1}  (u \cdot \varphi_k) |_{L^p}^2 \nonumber
\leq C \sum_{k}^\infty \EE |\gamma_k|^2  |\lambda_k^{-\gamma_1}  (u \cdot \varphi_k) |_{L^p}^2\nonumber\\
&=&C \sum_{k}^\infty   \lambda_k^{-\gamma_1} |u \cdot \varphi_k |_{L^p}^2 
\le C |u|_{H^{\delta_1}_p} ^2 \, \sum_{k}^\infty|\varphi_k |_{H^{\delta_2}_p} \lambda _k^{ -\gamma_1}
\le
C\,  |u|_{H^{\delta_1}_p} ^2 \, \sum_{k}^\infty k^{\frac 2d(\delta_2-\gamma_1)} 
.
\EEQSZ
Observe, if $\gamma_1 > \delta_2 + \frac{d}{2}$,
then $\sum_{k}^\infty k^{\frac 2d(\delta_2-\gamma_1)}<\infty$.
We continue by interpolation and the Young inequality, which yield that for any 
$\delta_1<\theta{<\frac 2p<1}$ and any $\ep>0$, there exists a positive constant $C(\ep)>0$ such that
\DEQSZ \label{eq:gamma2}
|g_{\gamma_1}(u)|_{\Upsilon(H,L^p)} ^2 &\le&
\ep |u|_{H^{\theta}_p} ^2+C(\ep)|u|^2_{L^p}.
\EEQSZ
Note that the assumptions are satisfied for the choice $\gamma_1 >d -\frac{2}{p}$. 

\medskip
\noindent In the sequel, we use estimate \eqref {eq:gamma2} to bound the following integral appropriately 
$$
\sum_{k\in\NN}
\int_0^t \la \xi^{p-1}(s), g_{\gamma_1}(\xi (s))[\varphi_k]\ra \, d\mathbf{w}^1_k(s)
.$$
\del{
	Note that by \eqref{g_sr} $g_{\gamma_1}(\xi (s))[\varphi_k] = \lambda_k^{-\gamma_1/2} \xi(s) \varphi_k$.  
Applying the Burkholder-Davis-Gundy inequality gives
\DEQSZ\label{hierbeginnts}
\lqq{
\EE \sup_{0\le s\le t}\lk|
\sum_{k\in\NN}
\int_0^s \la \xi^{p-1}(r),g_{\gamma_1}(\xi (r))[\varphi_k]\ra \, d\mathbf{w}^1_k(r)\rk|^\varsigma}&&
\\
\nonumber &
\le& C\EE
\sum_{k\in\NN}
\lk(\int_0^t \lk|\la \xi^{p-1}(s),g_{\gamma_1}(\xi (s))[\varphi_k]\ra \rk|^2\, ds\rk)^\frac \varsigma2
.
\EEQSZ
First note that we have by H\"older's inequality 
$$\lk|\la \xi^{p-1}(s),g_{\gamma_1}(\xi(s))[\varphi_k] \ra\rk|\le \lk|\xi^{p-1}(s)\rk|_{L^{\frac p{p-1}}}\lk|g_{\gamma_1}(\xi(s))[\varphi_k]\rk|_{L^p}
=\lk|\xi(s)\rk|^{p-1}_{L^{ p}}\lk|g_{\gamma_1}(\xi(s))[\varphi_k]\rk|_{L^p}.
$$

$$
\sum_{k\in\NN}
\int_0^t \la \xi^{p-1}(s), g_{\gamma_1}(\xi (s))\varphi_k\ra \, d\mathbf{w}^1_k(s)
.$$
} 
Note that by \eqref{g_sr} $g_{\gamma_1}(\xi (s))\varphi_k = \lambda_k^{-\gamma_1/2} \xi(s) \varphi_k$.  
Applying the Burkholder-Davis-Gundy inequality gives
\DEQSZ\label{hierbeginnts}
\lqq{
\EE \sup_{0\le s\le t}\lk|
\sum_{k\in\NN}
\int_0^s \la \xi^{p-1}(r),g_{\gamma_1}(\xi (r))\varphi_k\ra \, d\mathbf{w}^1_k(r)\rk|^\varsigma}&&
\\
\nonumber &
\le& C\EE
\sum_{k\in\NN}
\lk(\int_0^t \lk|\la \xi^{p-1}(s),g_{\gamma_1}(\xi (s))\varphi_k\ra \rk|^2\, ds\rk)^\frac \varsigma2
.
\EEQSZ
First note that we have by H\"older's inequality 
$$\lk|\la \xi^{p-1}(s),g_{\gamma_1}(\xi(s))\varphi_k \ra\rk|\le \lk|\xi^{p-1}(s)\rk|_{L^{\frac p{p-1}}}\lk|g_{\gamma_1}(\xi(s))\varphi_k\rk|_{L^p}
=\lk|\xi(s)\rk|^{p-1}_{L^{ p}}\lk|g_{\gamma_1}(\xi(s))\varphi_k\rk|_{L^p}.
$$
The Young inequality implies that for any $\epsilon > 0$, there exists a positive constant $C(\epsilon) > 0$ such that the following holds:
$$
\lk|\xi(s)\rk|^{p-1}_{L^{ p}}\lk|g_{\gamma_1}(\xi(s))\varphi_k\rk|_{L^p}
\le C(\ep)\lk|\xi(s)\rk|^p_{L^{ p}}+\ep \lk|g_{\gamma_1}(\xi(s))\varphi_k\rk|^p_{L^p}.
$$
Returning to \eqref{hierbeginnts}, the preceding estimate yields
$$
\lk(\int_0^t \lk|\la \xi^{p-1}(s)g_{\gamma_1}(\xi (s))\varphi_k\ra \rk|^2\, ds\rk)^{\frac{\varsigma}{2}}
\le
C(\ep)\lk(\int_0^t \lk|\xi(s)\rk|_{L^p}^{2p}\, ds\rk)^\frac{\varsigma}{2} +
\ep\lk(\int_0^t \lk|g_{\gamma_1}(\xi (s))\varphi_k \rk|_{L^p}^{2p}\, ds\rk)^\frac{\varsigma}{2}
.
$$
By applying H\"older’s inequality in time, we obtain the following for the first term in the sum
\begin{align*}
\EE \lk(\int_0^t \lk|\xi(s)\rk|_{L^p}^{2p}\, ds\rk)^{\frac{\varsigma}{2}} \le \EE\lk(\sup_{0\le s\le t} \lk|\xi(s)\rk|_{L^p}^p \int_0^t \lk|\xi(s)\rk|_{L^p}^{p}\, ds\rk)^{\frac{\varsigma}{2}}.
\end{align*}
Once again, by applying Young’s inequality, we know that for any $\epsilon > 0$, there exists a positive constant $C(\epsilon) > 0$ such that the following holds
\begin{align*}
\EE \lk(\int_0^t \lk|\xi(s)\rk|_{L^p}^{2p}\, ds\rk)^{\frac{\varsigma}{2}}
&\le
\ep \EE \sup_{0\le s\le t} \lk|\xi(s)\rk|_{L^p}^{p\varsigma}
+C(\ep, \varsigma)\EE \Big(\int_0^t \lk|\xi(s)\rk|_{L^p}^{p}\, ds\Big)^\varsigma.
\end{align*}
To handle the second term in the sum, as in \eqref{eq:gamma2}, interpolation and Young’s inequality yield that for all $\delta_1, \delta_2 \geq 0$ with $\delta_1 + \delta_2 \geq \tfrac{d}{p}$, $\delta_1 < \theta < \tfrac{2}{p}$, and $\delta_1 \in (0, \tfrac{d}{p})$, there exists a positive constant $C$ such that we have
\DEQS
\EE	\lk(\sum_{k}^\infty\int_0^t \lk|g_{\gamma_1}(\xi (s))\varphi_k \rk|_{L^p}^{2p}\, ds\rk)^{\frac{\varsigma}{2}}
&\le&
C \EE\lk(\sum_{k}^\infty \int_0^t |\xi(s)|_{H^{\delta_1}_p} ^{2p}|\varphi_k|^{2p}_{H^{\delta_2}_p } \,ds\,\rk)^{\frac{\varsigma}{2}}
\\
&\le &C\lk(\sum_{k=1}^\infty  k^{\frac 2d(\delta_2-\gamma_1)}\rk)^{\frac{\varsigma}{2}}
\EE\lk(
\int_0^t |\xi(s)|_{H^{\delta_1}_p} ^{2p}\,ds  \,\rk)^{\frac{\varsigma}{2}}
.
\EEQS
The infinite sum on the right-hand side is finite for $\gamma_1 > \frac{d}{2} + \delta_2$. Since $\delta_1 < \frac{2}{p}$, it follows from the interpolation result given in Proposition~\ref{runst1} that
$$
\EE\lk(
\int_0^t |\xi(s)|_{H^{\delta_1}_p} ^{2p} \,ds\,\rk)^{\frac{\varsigma}{2}}
\le C_1(\varsigma) \EE\lk(\sup_{0\le s\le t}|\xi(s)|_{L^p} ^{p}\rk)^{\varsigma}+
C_2(\varsigma) \EE\lk(
\int_0^t |\xi(s)|_{H^{\theta}_p} ^{p} \,ds\rk)^{\varsigma},
$$
and therefore, for any $\epsilon_1, \epsilon_2 > 0$ and $\varsigma \geq 1$, there exists a positive constant $C(\epsilon_1, \epsilon_2, \varsigma)$ such that
\DEQSZ\label{hierendets} 
\lqq{
\EE \sup_{0\le s\le t}\lk|
\sum_{k\in\NN}
\int_0^s \la \xi^{p-1}(r)g_{\gamma_1}(\xi (r))[\varphi_k]\ra \, d \mathbf{w}^1_k(r)\rk|^{\varsigma}}\nonumber&&
\\
\nonumber &
\le&  C(\ep_1,\ep_2, \varsigma)\,\EE \Big(\int_0^t \lk|\xi(s)\rk|_{L^p}^{p}\, ds\Big)^\varsigma
+\ep_1 \EE \sup_{0\le s\le t} \lk|\xi(s)\rk|_{L^p}^{p\varsigma}
\\&&{}+\ep_2\lk[ \EE\lk(\sup_{0\le s\le t}|\xi(s)|_{L^p} ^{p\varsigma}\rk)+
\EE\lk(
\int_0^t |\xi(s)|_{H^{\theta}_p} ^{p} \,ds\rk)^\varsigma\rk]
.
\EEQSZ
Note, we need that   $\frac d2+\frac dp -\min(\frac 2p,\frac dp)<\gamma_1$.
By Proposition \ref{runst1}, the second term on the RHS
can be cancelled in the proof of Proposition~\ref{reg_uk}.
\medskip

\paragraph{\textbf{The $\Upsilon$-radonifying norm estimate for the stochastic integral of the activator term~$v$: }}
Fix $\varsigma\ge 1$.  We keep the space $H = L^2(\CO)$ and set $E = H^\rho_2(\CO)$, for some well-chosen parameter $\rho\in \RR$. To show existence and boundedness for the process $v$ in Section~\ref{s:uniform}, we have to estimate bounds on the $\Upsilon$-norm with values in $E$. 
To be more precise, in Proposition~\ref{gammaradonv} 
we show that for all $\ep>0$ there exists a positive constant $C(\ep)$
such that the operator $$g_{\gamma_2}: H_2^{\rho+\frac{\aleph}{2}}(\CO) \to \Upsilon(H, E)$$ satisfies
\DEQSZ \label{eq: gYoung0}
|g_{\gamma_2}(v)|_{\Upsilon(H, E)}^2
\leq\ep  |v|_{H_2^{\rho+\frac{\aleph}{2}}}^2 +  C(\ep)    |v|^2_{E}, \quad \forall\, v \in 
H_2^{\rho+\frac{\aleph}{2}}(\CO).
\EEQSZ
By the main result of \cite[p. 788]{ErikaSeidler2}, we have 
\begin{align}
&\mathbb{E} \lk(\sup_{0\le s\le t}\lk| \int_0^s e^{(t-s) (r_2A+ a_2 I)} g_{\gamma_2}(\xi (s)){}d W(s)\rk|_E^\varsigma  \rk)\\
& \leq C_T \mathbb{E} \Big(\int_0^t |g_{\gamma_2}(\xi (s))|^2_{\Upsilon(H, E)} ds\Big)^{\varsigma /2}\nonumber\\
& \leq \EE\bigg[\Big(\eps\int_0^t |\xi(s)|_{H_2^{\rho+\aleph/2}}^{2}ds  + C(\eps) \int_0^t |\xi(s)|_{H_2^{\rho}}^2 ds \Big)^\frac{\varsigma }{2}\bigg]\nonumber.
\end{align}
Combining \eqref{burkholder} with \eqref{eq: gYoung0}, it follows that for any $\epsilon > 0$, there exists a constant $C(\epsilon) > 0$ such that for all
$t\in[0,T]$ we have
\begin{align}
&\mathbb{E} \lk(\sup_{0\le s\le t}\lk| \int_0^sg(\xi (s)){}d W(s)\rk|_E^\varsigma  \rk)
\leq \EE\bigg[\Big(\eps\int_0^t |\xi(s)|_{H_2^{\rho+\aleph/2}}^{2}ds  + C(\eps) \int_0^t |\xi(s)|_{H_2^{\rho}}^2 ds \Big)^\frac{\varsigma }{2}\bigg]\nonumber.
\end{align}
Rearranging yields that for all $\epsilon > 0$, there exists a positive constant $C(\epsilon)$ such that for all $t \in [0, T]$, we have
\begin{align}\label{Hrhoburkholderqqq}
&\mathbb{E} \lk(\sup_{0\le s\le t}\lk| \int_0^sg(\xi (s)){}d W(s)\rk|_E^\varsigma  \rk)
\leq  \eps \, \EE\bigg[\Big(\int_0^t |\xi(s)|_{H_2^{\rho+\aleph/2}}^{2}ds\Big)^{\frac{\varsigma }{2}}\bigg]   + C(\eps) \EE\bigg[\Big( \int_0^t |\xi(s)|_{H_2^{\rho}}^2 ds \Big)^\frac{\varsigma }{2}\bigg]
\end{align}
\del{\tesfalem{\sout{The estimate in \eqref{Hrhoburkholderqqq} will play a crucial role in Subsection~\ref{ss:RegAndUniform} for estimating $I_3$ in the proof of Proposition~\ref{reg_uk}.}} \todo[inline]{I have not changed anything, but the subsection is not clear, but I think it is referring either to Appendix B or Appendix D; plus we don't have $I_3$ in the proof of Proposition~\ref{reg_uk}}}.

\begin{proposition}[Estimate of the Gamma-radonifying norm]\label{gammaradonv}
Consider $\rho\in \RR$ and $\gamma_2>0$, such that
$$
\gamma_2> d-\frac \aleph2+ \tilde\eps \qquad \mbox{ and }\quad 
\begin{cases} 
\rho >\frac{d}{2} -\frac{\aleph}{2} & \mbox{ for } d=1, 
\\
-\frac{d}{2} < \rho \leq \frac{d}{2}-\frac{\aleph}{2} & \mbox{ for } d=2,
\end{cases}
$$
where $\tilde\eps$ is a sufficiently small positive number.
Then, there exists a positive constant $C>0$ 
such that the operator $g_{\gamma_2}: H_2^{\rho+ \frac{\aleph}{2}}(\CO) \to \Upsilon(H, H^{\rho}_2(\CO))$ satisfies 
\eqref{eq: gYoung0}. 
\begin{equation}\label{eq: gYoung0Wh}	
|g_{\gamma_2}(v)|_{\Upsilon(H, H^\rho_2)}^2
\leq C(\ep)  |v|^2_{H_2^{\rho}} + \ep  |v|_{H_2^{\rho+\frac{\aleph}{2}}}^{2}, \quad \forall\, v \in H_2^{\rho+\frac{\aleph}{2}}(\CO) .
\end{equation}
In addition, there exist some $\ep>0$ and  a positive constant $C>0$ such that
\begin{equation}\label{eq: gYoung2}
|g_{\gamma_2}(v)|_{\Upsilon(H, H^{\rho+\ep}_2)}^2
\leq C  |v|_{H_2^{\rho+\frac{\aleph}{2}}}^{2}, \quad \forall\, v \in  H_2^{\rho+\frac{\aleph}{2}}(\CO) .
\end{equation}
\end{proposition}

\begin{proof}
\textbf{We start with the case $d=1$ and $\rho >\frac{1}{2} -\frac{\aleph}{2}$.} 
Then, we have by Theorem 1-(i),\cite[p.\ 190]{runst} in combination with Subsection~\ref{ss:Besov} that 
\DEQSZ\label{hierweiter}
|v \varphi_k|_{H_2^\rho} &\leq&  |\varphi_k|_{H_2^{s_1}}\,|v|_{H_2^{s_2}},
\EEQSZ
where $p=2$, $q=q_1= q_2 = 2$, $s_1 = \rho$ and $s_2 = \rho + \frac{\aleph}{2}-\tilde \eps$ for some $\tilde \eps <1- \frac{\aleph}{2}$. 
Note that \eqref{hierweiter} is valid under condition (2) and (3) on the parameters $p, q, s_1, s_2$ 
of Theorem \ref{RS2a} (compare Theorem 1-(i),\cite[p.\ 190]{runst}).
Condition (3) of Theorem 1-(i),\cite[p.\ 190]{runst} (see also Subsection~\ref{ss:Besov}) is satisfied since $s_2 < \frac{1}{2}$   
iff $\rho > \tfrac{1}{2}- \tfrac{\aleph}{2}$ for some sufficiently small $\tilde \eps>0$ which is given by assumption. 
Since $\aleph \in (1,2]$ we have $\rho> \frac{1}{2} - \frac{\aleph}{2} = \frac{1}{2} - \frac{\aleph}{4} - \frac{\aleph}{4}\geq -\frac{\aleph}{4}$. Hence 
$s_1 + s_2 = 2 \rho + \frac{\aleph}{2} -\tilde \eps>0$, which implies the conditions of Theorem \ref{RS2a} (compare Condition (2) of Theorem 1-(i),\cite[p.\ 190]{runst}).
This implies 
\DEQSZ
|g_{\gamma_1}(u)|_{\Upsilon(H,H^\rho_2)} ^2 
&=& \EE |\sum_{k}^\infty \gamma_k  \lambda_k^{-\gamma_1}  (u \varphi_k) |_{H^\rho_2}^2 \nonumber
\leq C  \sum_{k}^\infty  |\gamma_k|^2  |\lambda_k^{-\gamma_1}  (u \varphi_k) |_{H^\rho_2}^2\nonumber\\
&=&C  \sum_{k}^\infty   \lambda_k^{-\gamma_2}  |u  \varphi_k |_{H^\rho_2}^2  
\leq  C |v|_{H_2^{\rho+\frac{\aleph}{2}-\tilde \eps}}^2 \sum_{k=1}^\infty |\varphi_k|_{H_2^{\rho}}^2 \lambda_k^{-\gamma_2} \nonumber\\
&\leq& C(\oO)  |v|_{H_2^{\rho+\frac{\aleph}{2}-\tilde \eps}}^2\sum_{k=1}^\infty  k^{-2(\gamma_2-\rho)} .\label{e:GRd1}
\EEQSZ
The right-hand side is finite iff  $\gamma_2> \rho+\frac 12 > \frac{1}{2} - \frac{\aleph}{2} + \frac 12 = 1 - \frac{\aleph}{2}$. 
In addition, by  interpolation of $H^{s_2}_2(\CO)$ between $H^{\rho}_2(\CO)$ and $H^{\rho+\frac{\aleph}{2}}_2(\CO)$ and by the (weighted) Young inequality, we know, that for all positive $\ep$ there exists a positive constant $C(\ep)$ such that
\DEQSZ\label{eq: gYoung}
|g(u)|_{\Upsilon(H, H_2^\rho)}^2 &\leq& \eps  |u|_{H_2^{\rho+\aleph/2}}^{2} + C(\eps)  |u|_{H_2^{\rho}}^2.
\EEQSZ
The last inequality gives the assertion.

\del{Interpolation gives $|u|_{H_2^{\beta+ \aleph /4}} \le |u|^\frac 12 _{H_2^{\beta+ \aleph /2}} |u|_{H_2^{\beta}} ^\frac 12 $.
Next, by the Young inequality, it follows that  for any $\eps>0$, there exists a constant $C = C(\eps)>0$ such that
\DEQSZ\label{eq: gYoung}
|g(u)|_{\Upsilon(H, H_2^\beta)}^2 &\leq& \eps |u|_{H_2^{\beta+\aleph/2}}^{2} + C(\eps) |u|_{H_2^{\beta}}^2.
\EEQSZ}

\noindent \textbf{We treat the case $d\in \{2, 3\}$: } Recall $1 < \aleph \leq 2$. 
Let $-\frac{d}{2} < \rho \leq \frac{d}{2} - \frac{\aleph}{2}$. 
 Such a $\rho$ exists iff 
\[
-\frac{d}{2} < \rho \leq \frac{d}{2} -\aleph\qquad \Leftrightarrow \qquad \aleph \leq d \qquad \Leftrightarrow \qquad d\geq 2. 
\]
To apply Theorem 1 (iii) \cite[p.\ 190/191]{runst} 
we set parameters $p=2$, $n=d$, 
$s_1 := \rho + \frac{\aleph}{2}-\tilde \eps$ and
$s_2 = \frac{d}{2} - \frac{\aleph}{2} + \tilde \eps$.  
Note that $s_1 + s_2 = \rho + \frac{d}{2}>0$ iff $\rho> - \frac{d}{2}$, which guarantees 
condition (2) in \cite[p.\ 190]{runst}. 
In addition,   
\begin{align*}
s_2 - s_1  
&= \frac{d}{2} - \frac{\aleph}{2} + \tilde \eps - (\rho+ \frac{\aleph}{2} -\tilde \eps) 
= \frac{d}{2} - \aleph + 2\tilde \eps - \rho > \frac{d}{2} - \aleph - \rho.
\end{align*}
The right-hand side of the preceding expression is nonnegative (satisfying \cite[p.\ 190 (1)]{runst}) iff 
\[
\frac{d}{2} - \aleph - \rho\geq 0, \qquad \Leftrightarrow \qquad \rho \leq \frac{d}{2} -\aleph. 
\]
Finally, due to $\aleph >1$ for all $\tilde \eps \in (0,\frac{\aleph}{2})$ 
condition (11) in \cite[p.\ 191]{runst} is satisfied (combined with Subsection~\ref{ss:Besov}),   
\[
s_2 = \frac{d}{2} - \frac{\aleph}{2}  + \tilde \eps < \frac{d}{2}. 
\]
Consequently, Theorem 1 (iii) \cite[p.\ 190/191]{runst} combined with Subsection~\ref{ss:Besov} 
yields a constant $C>0$ such that 
\DEQSZ\label{hieroben1}
|u \varphi_k|_{H_2^{\rho}} & \leq & C \,|u|_{H_2^{s_1}} |\varphi_k|_{H_2^{s_2}}. 
\EEQSZ
Hence,
\begin{align}
|g_{\gamma_2}(v)|_{\Upsilon(H, H^\rho_2)}^2
&
\le C   \sum_{k=1}^\infty |v \varphi_k|_{H^\rho_2}^2\, k^{-\frac{2}{d}\gamma_2} \nonumber\\
&\leq C(\oO)  |v|_{H_2^{s_1}}^2 \sum_{k=1}^\infty  |\varphi_k|^2_{H^{s_2}_2}\, k^{-\frac{2\gamma_2}{d}} 
\leq C(\oO)  |v|_{H_2^{\rho + \frac{\aleph}{2}-\tilde \eps}}^2\sum_{k=1}^\infty  k^{-2\frac{(\gamma_2-s_2)}{d}} .
\label{e:GRd23}
\end{align}
Since $\gamma_2>d-\frac \aleph 2  +\tilde \eps =\frac d2-\frac \aleph 2+\frac{d}{2} +\tilde \eps= s_2+\frac{d}{2}$, we know the sum is finite. 
Again, interpolation with $\theta = \frac{2\tilde \eps}{\aleph}\in (0,1)$ for some $\tilde \eps>0$ sufficiently small between $H^\rho_2$ and $H^{\rho+\frac{\aleph}{2}}_2$ yields 
$$
|u|_{H_2^{s_1}} \leq |u|_{H_2^{\rho}}^\theta ~|u|^{1-\theta}_{H_2^{\beta+\frac{\aleph}{2}}}.
$$
Young's inequality finally yields \eqref{eq: gYoung0Wh}.
To show the second part, we refer to \eqref{e:GRd1} and \eqref{e:GRd23}. 
A shift by the arbitrarily small parameter $\tilde \eps$ shows the result. 
This finishes the proof. 
\end{proof}

\paragraph{\textbf{Trace estimates: }}
Let 
\begin{equation}\label{e:gamma1Spur}
\gamma_1>\frac dp+\frac d2-\min(\frac d2,1).
\end{equation}
Since we apply later the It\^ o formula for $\psi(u)=|u|_{L^p}^p$, $p\geq 2$,
let us note that  
\DEQS
\mbox{Tr}\Big[  D^2\psi(u)[g_{\gamma_1}^\ast (u),g_{\gamma_1}(u)]\big]=
p(p-1)\sum_{k=1}^\infty \lambda_k^{-\gamma_1}\int_\CO |u(x)|^{p-2}u(x)^2 \varphi^2_k(x)\, d x
.
\EEQS
Note that  
H\"older's inequality gives
\DEQS
\lqq{ \sum_{k=1}^\infty \lambda_k^{-{\gamma_1}}\lk|\int_\CO |u(x)|^{p-2}u(x)^2 \varphi^2_k(x)\, d x\rk|}
&&\\
&\le&  |u^{p-2}|_{L^\frac {p}{p-2}}\lk(  \sum_{k=1}^\infty \lambda_k^{-{\gamma_1}}|u^2\varphi^2_k|_{L^\frac p2}\rk)
\le |u|^{p-2}_{L^p}\lk(  \sum_{k=1}^\infty \lambda_k^{-{\gamma_1}}|u\varphi_k|^2_{L^p}\rk)
\EEQS
H\"older's inequality and Sobolev embedding Theorems gives for $\delta_1\in(0,\min(\frac dp,\frac 2p))$, 
and $\delta_1+\delta_2\ge \frac dp$,
$$ |u\varphi_k|\le C|u|_{H^{\delta_1}_p}|\varphi_k|_{H^{\delta_2}_p}\le C|u|_{H^{\delta_1}_p}\lambda _k^{\frac{\delta_2}2}.
$$
The Young inequality gives further
\DEQS
\sum_{k=1}^\infty \lambda_k^{-{\gamma_1}}\lk|\int_\CO |u(x)|^{p-2}u(x)^2 \varphi^2_k(x)\, d x\rk|
\le \ep |u|^p_{L^p}+C(\ep) |u|_{H^{\delta_1}_p}^p \lk(  \sum_{k=1}^\infty \lambda_k^{-\gamma_1}\lambda_k ^{\delta_2}\rk)^\frac p2.
\EEQS
Now, \eqref{EVsup} yields to
\DEQS
\sum_{k=1}^\infty \lambda_k^{-{\gamma_1}}\lk|\int_\CO |u(x)|^{p-2}u(x)^2 \varphi^2_k(x)\, d x\rk|
\le \ep |u|^p_{L^p}+C(\ep) |u|_{H^{\delta_1}_p}^p \lk(  \sum_{k=1}^\infty k^{\frac 2d(\delta_2-\gamma_1)}\rk)^\frac p2.
\EEQS
Due to the condition on $\gamma_1$, the right-hand side is finite, and we have by applying interpolation and the Young inequality
we know for all $\theta<\frac 2p$
\DEQSZ\label{itotrace}
\sum_{k=1}^\infty \mbox{Tr}\Big[  D^2\psi(u)[g_{\gamma_1} ^*(u)\varphi_k,g_{\gamma_1}(u)\varphi_k]\big]
\le
\ep |u|_{H^\theta_p}^p+ C(\ep) |u|_{L^p}^p.
\EEQSZ
The preceding inequality is used in Section~\ref{s:uniform} . 

For $\Psi(v) := |(-\Delta^{-\rho/2}) v|_{L^2}^2$ we can use Proposition \ref{gammaradonv}. 
In particular,
\begin{align}\label{eq:itotracerho}
D^2 \Psi(v)
&= 2\langle \Delta^{\rho/2} (v  \varphi_k), \Delta^{\rho/2} (v  \varphi_k)\rangle  = 2|v\varphi_k|_{H_2^\rho}^2 .
\end{align}
Continuing as \eqref{hierweiter}, we can show that under the assumption of Proposition \eqref{gammaradonv},
we can show that for  all $\ep>0$ there exists a positive constant $C(\ep)$
such that
\del{\begin{equation}\label{e:dieSpurDieWirBrauchen}
|\sum_{k=1}^\infty \mbox{Tr}\big[  D^2\Psi(u)[g ^*(u)\varphi_k,g(u)\varphi_k]\big]|
\leq C(\ep) |v|^2_{H_2^{\beta}} + \ep |v|_{H_2^{\beta+\frac{\aleph}{2}}}, \quad \forall\, v \in H_2^{\beta}(\CO)\cap H_2^{\beta+\frac{\aleph}{2}}(\CO) .
\end{equation}}
\begin{equation}\label{e:dieSpurDieWirBrauchen}
		|\sum_{k=1}^\infty \mbox{Tr}\big[  D^2\Psi(u)[g ^*(u)\varphi_k,g(u)\varphi_k]\big]|
		\leq C(\ep) |v|^2_{H_2^{\rho}} + \ep |v|_{H_2^{\rho+\frac{\aleph}{2}}}, \quad \forall\, v \in H_2^{\rho}(\CO)\cap H_2^{\rho+\frac{\aleph}{2}}(\CO).
\end{equation}

\bigskip 
\section{\textbf{Existence of a unique joint solution $(u_\kappa, v_\kappa)$ to the system \eqref{sysu}-\eqref{sysw}}}
\label{s:appendix B}

\noindent {The aim of this section is to establish all the requisites of the application of the stochastic Schauder theorem, in order to have local solutions. For this sake we first establish the wellposedness of the fixed point operator, the invariance on a particular intersection of balls in different topologies with the cone of nonnegative functions, compactness and continuity in the respective subsections. }

\bigskip 
\subsection{\textbf{Wellposedness of the fixed point operator}}\label{ss:wellposedness}\hfill\\

\noindent In this section we show that for any $\n\in\NN$ a unique pair $(\uk,\vk)$ of solution
to the system \eqref{sysu0} and \eqref{sysw0}
exists, lies in a the precompact set\footnote{The constants $K_1, K_2$, and $ K_3(\kappa)$ are defined in \eqref{def_constants}.} $\mathcal{K}_{\mathfrak{A}}^\kappa:=\mathcal{K}_{\mathfrak{A}}(K_1, K_2, K_3(\kappa))$ 
and is non-negative in $\oO$. 
Let us start with a technical proposition.

\begin{tproposition}\label{tetaxi}
	Let us assume that the parameters $d$, $\rho$, $\aleph$, $q$, and $p^\ast\ge 2$ satisfy
	\DEQSZ\label{bed000}
	\frac{d}{2} - \rho \le \frac{\FE}2\frac 2{qm} + \frac{d}{q} (\frac{1}{r}- \frac{1}{p^\ast}).
	\EEQSZ
	Let  $r_2\ge 1$ such that $\frac 1{p^\ast}+\frac 1{r_2 }\le \frac{1}{r}$. Then, for all  
	parameters $K_1, K_2, K_3>0$ and any 
	any $\kappa\in\NN$, there exists a constant $C=C(\kappa)>0$ such that for all $\tilde r\in \NN$ and $(\eta, \xi) \in \mathcal{K}_{\mathfrak{A}}(K_1, K_2, K_3)$ we have
	\DEQSZ\label{eee}
	\EE  \|\eta\, \phi_\kappa(\cdot,\xi)\, \xi^q\|_{L^m(0,T;{L^r})}^{\tilde r}
	\le C(\kappa)\, \EE \sup_{0\le s\le T}|\eta(s)|_{L^{p^\ast}}^{\tilde r}.
	\EEQSZ
\end{tproposition}
\begin{remark} 
Technical Proposition \ref{tetaxi} is similar to Technical Proposition \ref{dasauch}. However, the difference is that here we can use the cut-off function, whereas in Proposition \ref{dasauch} we have to tackle the term $\|v\|_{L^{qm}(0,T;L^{qr_2})}^{\tilde r q}$ without any cut-off. Hence, we apply the Cauchy–Schwarz inequality to separate the two terms in \eqref{twotwremssep} and need an exponent on  $$\lk( \EE \|v\|_{L^{qm}(0,T;L^{qr_2})}^{\tilde r q}\rk)$$
smaller than one. 
Consequently, the parameter conditions in Technical Proposition \ref{tetaxi} are less restrictive than those in Technical Proposition \ref{dasauch}. 
\end{remark} 
\begin{proof}
The proof is an application of H\"older's inequality for the exponents $p^\ast$ and $r_2$, using the definition of $\phi_\kappa$. In particular, we have
\DEQSZ\label{twotwremssep}\notag 
\EE  \| \eta\,\phi_\kappa(\cdot,\xi)\, \xi^q\|_{L^m(0,T;{L^r})}^{\tilde r}
&\le& \EE \lk(\int_0^T |\eta(s)|_{L^{p_*}}^{m} |\phi_\kappa(\cdot, \xi) \xi^q(s)|_{L^{r_2}}^m ds \rk)^{\frac{\tilde r}{m}}\\\notag 
&\le&\EE \sup_{s\in [0, T]} |\eta(s)|_{L^{p_*}}^{\tilde r} \lk(\int_0^T  |\phi_\kappa^\frac{1}{q}(\cdot, \xi) \xi(s)|_{L^{qr_2}}^{mq} ds \rk)^{\frac{\tilde r }{m}}\\
&= &  \EE \lk[ \sup_{0\le s\le T}|\eta(s)|_{L^{p^\ast}}^{\tilde r }
\cdot \|\phi_\kappa^\frac{1}{q}(\cdot,\xi)\,\xi\|_{L^{qm}(0,T;L^{qr_2})}^{\tilde r q}\rk].
\EEQSZ
Using Proposition \ref{interp_rho} and by hypothesis \eqref{bed000} that
$$
\frac{d}{2} - \rho \le \big(\frac{\aleph}{2}\big) \frac{2}{qm} + \frac{d}{r_2q}
$$
yields positive constants $C$ and $C(\kappa)$ such that 
$$ 
\|\phi_\kappa^\frac{1}{q}(\cdot,\xi)\,\xi\|_{L^{qm}(0,T;L^{r_2q})} \le C \| \phi_\kappa^\frac{1}{q}(\cdot,\xi)\,\xi\|_{\mathbb{H}_{\rho, \aleph}} \leq C(\kappa),
$$ 
due to the definition of the cut-off $\phi_\kappa$. This shows the desired inequality \eqref{eee}. 
\end{proof}

\del{\begin{proof}
This part we will need for the limiting case $q=2$, $d=2$, and $\aleph=2$. Let us assume $1\le r<2$.
The proof is an application of the generalised H\"older's inequality given in Theorem \ref{RSb} (compare also to  p. 171, \cite{runst})
for the exponents $\tilde p$ and $r_2$ and $s>0$ such that
$$
\frac 1{r_2}>\frac 12 , \quad \frac 1{\tilde p}<\frac {s_2-s_1}d,
\quad 
\mbox{and}\quad \frac {(s_1+s_2)}d>\frac 1r-1.
$$
In this way we get
\DEQSZ\label{twotwremssep}\notag 
\EE  \| \eta\,\phi_\kappa(\cdot,\xi)\,
\xi^q\|_{L^m(0,T;H^{s_1}_r})^{\tilde r}
\le   \EE \lk(\int_0^T |\eta(s)|_{{H^{s_1}_{\tilde p}}}^{m}
|\phi_\kappa(\cdot, \xi) \xi^q(s)|_{H^{s_2}_{r_2}}^m ds \rk)^{\frac{\tilde r}{m}}  .
\EEQSZ
An application of  H\"older's inequality with $\frac 1{l'}+\frac 1l\le \frac 1m$, gives
\DEQSZ 
\EE  \| \eta\,\phi_\kappa(\cdot,\xi)\,
\xi^q\|_{L^m(0,T;H^{s_1}_2})^{\tilde r}	
&\le&
\EE \lk( \int_0^ T |\eta(s)|_{H^{s_1}_{\tilde p}} ^{l'} \, ds\rk)^\frac {\tilde r}{l'} \lk(\int_0^T  |\phi_\kappa(\cdot, \xi) \xi^q(s)|_{H^{s_2}_{r_2}}^{l} ds \rk)^{\frac{\tilde r }{l}}.
\EEQSZ
Theorem 1 in \cite[p. 363-364]{runst} gives for
$$
p>\frac {dqr_2}{d+(q-1){s_2}t }
$$
(here, we have to find the $p^\ast$ large enough)
\DEQSZ
\EE  \| \eta\,\phi_\kappa(\cdot,\xi)\,
\xi^q\|_{L^{m}(0,T;H^{s_1}_2})^{\tilde r}	
&\le&  \EE \lk[ \|\eta\|_{L^{l'}(0,T;H^{s_1}_{\tilde p} )}^{\tilde r }
\cdot \|\phi_\kappa(\cdot,\xi)\,\xi\|_{L^{ql}(0,T;H^{s_2}_p)}^{\tilde r q}\rk].
\EEQSZ
Now, by the embedding $H^{\tilde s}_2(\CO)\hookrightarrow H^{{s_2}}_p(\CO)$ for 
$$
s+\lk(\frac d2-\frac dp\rk)=\tilde s,
$$
we get
\DEQSZ
\EE  \| \eta\,\phi_\kappa(\cdot,\xi)\,
\xi^q\|_{L^{l'}(0,T;H^{s_1}_2})^{\tilde r}	
&\le&  
\EE \lk[ \|\eta\|_{L^{l'}(0,T;H^\gamma_{\tilde p})}^{\tilde r }
\cdot \|\phi_\kappa(\cdot,\xi)\,\xi\|_{L^{ql}(0,T;H^{\tilde s}_2)}^{\tilde r q}\rk].
\EEQSZ
For interpolation, where 
$$
\tilde s=\rho \theta+(\theta+\aleph)(1-\theta)\quad \mbox{and}\quad \frac {1}{ql}=\theta\frac 1\infty+(1-\theta)\frac 12
$$
or
$$
\theta =\frac {\aleph+\rho-\tilde s}{\aleph}\quad 
\mbox{and}\quad \frac 1{lq}=\frac {1-\theta}2,
$$
we get by interpolation
\DEQSZ
\EE  \| \eta\,\phi_\kappa(\cdot,\xi)\,
\xi^q\|_{L^m(0,T;H^{s_2}_2})^{\tilde r}	
&\le&  
\EE \lk[ \|\eta\|_{L^{l'}(0,T;H^{s_2}_{\tilde p})}^{\tilde r }
\cdot \|\phi_\kappa(\cdot,\xi)\,\xi\|_{\mathbb{H}_{\rho,\aleph}}
^{\tilde r q}\rk].
\EEQSZ	%
By the definition of the cut-off, there  exists  a constant $C$ and $C(\kappa)$ such that 
$$ 
\|\phi_\kappa^\frac{1}{q}(\cdot,\xi)\,\xi\|_{L^{qm}(0,T;L^{r_2q})} \le C \| \phi_\kappa^\frac{1}{q}(\cdot,\xi)\,\xi\|_{\mathbb{H}_{\rho, \aleph}} \leq C(\kappa),
$$ 
due to the definition of the cut-off $\phi_\kappa$. 
Now, we have 
\DEQSZ
\EE  \| \eta\,\phi_\kappa(\cdot,\xi)\,
\xi^q\|_{L^m(0,T;H^{s_2}_2})^{\tilde r}	
&\le&  
C(\kappa)\, \EE \lk[ \|\eta\|_{L^{l'}(0,T;H^{s_2}_{\tilde p})}^{\tilde r }
\rk].
\EEQSZ	%
Observe, one can show by interpolation and the Stroock-Varopoulos inequalitiy that if
for any $\gamma<\frac 2{p^\ast}$
$$
\frac 1{\tilde p} = \frac {\theta }{p^\ast} + \frac {1-\theta}{2},\quad s_1=\gamma(1-\theta)<(1-\theta)\frac 2{p^\ast},
$$
then 
$$ \|\eta\|_{L^{l'}(0,T;H^{s_2}_{{\tilde p}} ) }
\le  \sup_{0\le s\le T} |\eta(s)|^{p^\ast}_{L^{p^\ast}}+\int_0^ T |\eta ^{p^\ast-2}(\nabla \eta(s))^2|_{L^1}\, ds.
$$
Due to Propsition \ref{reg_uk}, 
$$
\EE  \|\eta\|_{L^{l'}(0,T;H^{s_2}_{\tilde p})}^{\tilde r}<\infty.
$$
for $\theta=\frac 12 (2 - p^\ast s_1)$ and  $\tilde p=\frac{4 p^\ast}{4 + (p^\ast)^2 s_1}$.	%
\end{proof}
}

\bigskip 
\subsection{\textbf{Invariance and well-posedness of the operator $\mathcal{V}_\MA^\kappa$}}\label{s:invariance}\hfill\\

\noindent In a sequence of propositions,  we analyze several properties of the operator 
$\mathcal{V}_{\MA,W}^\kappa$. 
In Proposition~\ref{reg_uk} we show well-posedness of the operator. Additionally,  in  Proposition ~\ref{reg_uk} and Proposition  \ref{limituk}, we establish some estimates that allows us to demonstrate the invariance of the set
$\mathcal{K}_\MA(K_1,K_2,K_3,k_4)$ for some parameters $K_1$, $K_2$, $K_3$, and $K_4$ under $\mathcal{V}_{\MA,W}^\kappa$.

\bigskip

\begin{proposition}\label{reg_uk}
Let $(q, \aleph, d)$ be given. 
Then, 
	for any $\kappa\in\mathbb{N}$  for any $K_i>0$ for $i=1,2,3,$ the following statements are satisfied.
\begin{enumerate}

\item  
Let us assume that $\rho\in\mathbb{R}$ satisfies 
$$
d\lk(\frac 12 -\frac 1q\rk)\le \rho+\frac {\FE}{2q},\quad \rho<\frac \FE2-\frac d2,
$$
and $p^\ast\in[2,\infty)$ be a number such that 
$$ 
p^\ast > \frac {2d}{ \FE +2d-dq+2q\rho}.
$$ %
\casedrei{ In case $d=3$ we assume 
$$
\frac d2-\rho\le \frac \FE {q}+\frac 56 \frac dq
\quad \mbox{and}
\quad p^\ast > \frac {6d}{ 3\FE +5d-3dq+6q\rho}.
$$}
Then, 	for any $u_0\in L^2\big(\Omega;L^2(\CO)\big)\cap L^{p^\ast}\big(\Omega;L^{p^\ast}(\CO)\big)$, and any $v_0\in L^2\big(\Omega;H^\rho_2(\CO)\big)$
for any  $\n\in\NN$ and  $(\xi,\eta)\in \mathcal{K}_\MA(K_1,K_2,K_3)$, there exists a unique solution $\uk:\Omega\times[0,\TTend\times \CO  \to \RR$
to the equation \eqref{sysu0},
such that
\DEQSZ
\EE\sup_{0\le t\le T} |\uk(t)|_{L^2}^2+2\EE\int_0^T |\uk(t)|_{H^{1}_2}^2\, dt\le \EE|u_0|^2_{L^2}+C(\kappa)K_2<\infty\quad . \label{e:trouble} 
\EEQSZ
\item  
Let us assume that $\rho\in\mathbb{R}$ satisfies 
$$
\frac{d}{2}(1-\frac 1q) - \frac{\FE}{2q} <  \rho
 \quad \mbox{and}
\quad \rho<\frac \FE2-\frac d2
,
$$
and $p^\ast\in[2,\infty)$ be a number such that 
$$
p^\ast>\frac {2d}{2\FE +d-dq-2\rho+2q\rho}.
$$
\casedrei{In case $d=3$ we assume 
$$
\frac{d}{2} - \frac{\FE}{q-1} <\rho ,  \quad \mbox{and}
\quad p^\ast > \frac {6d}{ 3\FE +5d-3dq+6q\rho}.
$$}
Then, 	for any $u_0\in L^2\big(\Omega;L^2(\CO)\big)\cap L^{p^\ast}\big(\Omega;L^{p^\ast}(\CO)\big)$, and any $v_0\in L^2\big(\Omega;H^\rho_2(\CO)\big)$
for any  $\n\in\NN$ and  $(\xi,\eta)\in \mathcal{K}_\MA(K_1,K_2,K_3)$, there exists a unique solution $\uk:\Omega\times[0,\TTend\times \CO  \to \RR$
to the equation 
\eqref{sysw0},
such that
\DEQSZ
\EE\sup_{0\le t\le T} |\vk(t)|_{H^{\rho}_2}^2+2\EE\int_0^T |\vk(t)|_{H^{  \FE /2+\rho}_2}^2\, dt\le \EE|v_0|^2_{H^\rho_2} +C(\kappa) K_2<\infty. \label{e:vtrouble}
\EEQSZ
\end{enumerate}
\end{proposition}
\medskip
\begin{remark}
	Since
	$q<\frac {\FE+2d}{2d-\FE}$, the condition in a is possible. 
Since
$q<\frac {\FE+d}{2d-\FE}$, the condition in b is possible. 
\end{remark}

\begin{proof}[\textbf{Proof of Proposition \ref{reg_uk}}] ~~ 
Let us recall the system   \eqref{sysu0}-\eqref{sysw0}, which is given by the initial condition $\uk (0)=u_0$, $\vk  (0)=v_0$, and 
%
\DEQSZ
\label{sysu}
\\\nonumber 
d{\uk }(t)&= [r_1 \,\Delta \uk (t)+ a_1 \uk (t)+ b_1 - c_1 \phi_\n(\xi,t)\cdot \eta(t) \cdot \,|\xi|^q(t)]\, dt +\sigma_1 \, g_{\gamma_1}(\uk (t))\, dW_1,\quad \nonumber\\
&\hfill
\label{sysw} \\\nonumber 
d{\vk  }(t)&= [r_2 \,A \vk  (t)+ a_2 \vk  (t)+ b_2+ c_2\phi_\n(\xi, t)\cdot \eta(t)\cdot |\xi|^q(t)]\, dt + \sigma_2 \,g_{\gamma_2}(\vk  (t)) \, dW_2(t).\nonumber
\EEQSZ
\noindent Let  $(\eta,\xi)$ be a pair of two processes, belonging to $\spK$ and
define 
\DEQSZ\label{fdeff}
f(t) & = &\phi_\kappa(\xi, t) \cdot \eta(t)\cdot |\xi|^q(t).
\EEQSZ 
\newcommand{\mm}{m}
Let $\mm=2$ and $\gamma=1/2$ for $d=1$ and 
let $\mm=2$ and $\gamma=1$ if $d=2$.
First, we show the existence of a unique solution to equation \eqref{sysu}. This, we are doing by an application of Theorem  4.5 of \cite{maxjan}. However, before doing it, we have to verify the assumptions. Let us start with the 
term given by $f$. In case, $d=1,2$, it follows by the embedding    $L^1(\CO)\hookrightarrow H^{-\gamma}_2(\CO)$ 
$$
\EE  \|f\|_{L^\mm (0,T;{H^{-\gamma }_2})}^{\mm }\le C\EE  \|\phi_\n(\xi,\cdot)\cdot \eta(\cdot) \cdot \,\xi^q(\cdot)\|_{L^\mm (0,T;{L^1})}^{\mm }.
$$
The H\"older inequality gives for $\frac 1p+\frac 1{p'}\le 1$
\DEQS 
\EE  \|f\|_{L^\mm (0,T;{H^{-\gamma}_2})}^{\mm } &\le &
C\EE \sup_{0\le t\le T}|\eta(t)|_{L^p}^\mm   \|\phi_\n(\xi,\cdot) \cdot \,|\xi|^q(\cdot )\|_{L^\mm (0,T;{L^{p'}})}^{\mm }.
\\
&\le &
C\EE \sup_{0\le t\le T}|\eta(t)|_{L^p}^\mm   \|\phi^{\frac1q}_\n(\xi,\cdot) \cdot \,\xi(\cdot )\|_{L^{\mm q}(0,T;{L^{p'q}})}^{\mm q}.
\EEQS 
Since $-\frac \FE 2\le \rho$, we know by the definition of the cut-off function $\phi_\n(\xi,\cdot )$ and Proposition \ref{interp_rho} that for 
$$
\frac d2-\rho \le \frac \FE 2\frac 2{q\mm}+\frac d{{p'}q}= \frac \FE {q\mm}+\frac dq\lk(1-\frac 1p\rk)
$$
we have $  \|\phi_\n(\xi,\cdot ) \cdot \,\xi^q(\cdot)\|_{L^{q\mm }(0,T;{L^{p'q}})}\le \kappa$.
In particular, we obtain
$$
\EE  \|f\|_{L^\mm (0,T;{H^{-\gamma}_2})}^{\mm }\le C\EE \sup_{0\le t\le T}|\eta(t)|_{L^p}^\mm   C(\kappa).
$$
Since $(\eta,\xi)\in \spK$, $\EE \sup_{0\le t\le T}|\eta(t)|_{L^{p^\ast}}^\mm \le K_1$ and we have 
$$
\EE  \|f\|_{L^\mm (0,T;{H^{-\gamma}_2})}^{\mm }\le CK_1  C(\kappa).
$$
The conditions on the parameters is $m=2$ and 
\DEQSZ\label{ucond01}
\frac d2-\rho \le \frac \FE {2q}+\frac dq
\EEQSZ 
and $p^\ast $ has to be chosen 
\DEQSZ\label{condpast01}
p^\ast > \max(\frac {2d}{ \FE +2d-dq+2q\rho},4).
\EEQSZ

\noindent Next we use Theorem 4.5 of \cite{maxjan} to establish inequality \eqref{e:trouble}.
To use this result, we need to set spaces $X_0$ and $X_1$ such that $X_1 \embed X_0$. 
Let  $X_0= H^{-1}_2$ and $X_1=H^{1}_2$ so that $L^2$ is our interpolation 
space between $X_0$ and $X_1$. Thus the conditions of Theorem 4.5 of \cite{maxjan} is satisfied, since $u_0\in L^2$. 
Note that since $\frac12-\frac1d<1 $ for all $d\in\mathbb{N}$, the embedding $X_1\embed X_0$ holds for dimensions $d=1,2$.

\noindent With this setting of our spaces, it follows by Theorem 4.5 of \cite{maxjan} 
that  a unique solution $\uk$ exists 
that satisfies
\DEQS
\EE\sup_{0\le s\le T}|u_\kappa|^2_{L^2} + r_1 \EE \int_0^T |u_\kappa|^2_{H^{1}_2} ds 
\le \EE|u_0|_{L^2}^2+C(\kappa)K_1<\infty.  
\EEQS
\noindent

\medskip
\noindent Next we prove statement b.) 

\medskip
\noindent The solution $\vk$ is given by 
\del{\DEQS
\lqq{ \vk(t)= e^{(r_2\Delta-\alpha_2I) t}v_0+\int_0^ t   e^{(r_2\Delta-\alpha_2I) (t-s)} f(s)\, ds}
&&\\
&&{}
+ \int_0^te^{(r_2\Delta-\alpha_2I)(t-s)}g(\vk(s))\,dW_2(s) + \alpha_1\int_0^ t   e^{(r_2\Delta-\alpha_2I) (t-s)}  \beta_2\, ds,
\EEQS}
\DEQS
\lqq{ \vk(t)= e^{(r_2A-a_2I) t}v_0+\int_0^ t   e^{(r_2A-a_2I) (t-s)} f(s)\, ds}
&&\\
&&{}
+ \int_0^te^{(r_2A-a_2I)(t-s)}g(\vk(s))\,dW_2(s) + a_1\int_0^ t   e^{(r_2A-a_2I) (t-s)}  b_2\, ds,
\EEQS
where $f$ is defined in \eqref{fdeff}. By Theorem \ref{Th:compact}, it follows that
$$
\|\mathfrak{F}_{r_2\Delta-\alpha_2I}f\| _{\mathbb{H}_{\rho,\aleph}}
\le C \|f\|_{L^2(0,T;H^{\rho-\frac \aleph 2}_2)}.
$$
To show the existence of a unique solution to $v_\n$, 
first, note, that
if under the assumption of the proposition we have the bound 
\DEQSZ
\EE \|{f}\|_{L^2(0,T;H^{\rho-\frac \aleph 2}_2) }^2\leq C(\kappa) K_2 <\infty. \label{e:fmaxforv}
\EEQSZ
Now,
we  can apply \cite[Therorem 4.5]{maxjan} with a choice of spaces $X_1=H^{\rho+\frac{\aleph}{2}}_2$, $X_0 = H^{\rho -\frac{\aleph}{2}}_2$ and the interpolation space $H^\rho_2$   to obtain a unique mild solution  
$$v_\kappa\in  L^2(  \Omega,C([0,T],H^\rho_2))\cap  L^2(  \Omega,L^2([0,T],H^{\rho+\frac{\aleph}{2}}_2)),$$ 
satisfying \eqref{e:vtrouble}.

\noindent Let us start to verify if the assumptions of Technical Proposition \ref{tetaxi} are verified. 
If $\frac \FE 2 -\rho > \frac d2$, then we have $L^1(\CO)\hookrightarrow H^{\rho-\frac \aleph 2}_2(\CO)$
and it has to be satified 
	\DEQSZ\label{musserfuelltsein}
\frac{d}{2} - \rho &\le & \frac{\FE}{2q} + \frac{d}{q} (1-\frac1{p^\ast}).
\EEQSZ
In particular, if $p^\ast\to\infty$ the right hand side will increase, and the existence of $p^\ast$ such that \eqref{musserfuelltsein} is satisfied, is given by 
	\DEQS
\frac{d}{2} - \rho &\le & \frac{\FE}{2q} + \frac{d}{q} ,
\EEQS
or
	\DEQS
\frac{d}{2}(1-\frac 1q) - \frac{\FE}{2q} &<&  \rho.
\EEQS
Then, any $p^\ast$ satisfying the assumption given in the Proposition satisfies \eqref{musserfuelltsein}.

\del{\color{blue}
In the sequel we show that the bound  \eqref{e:fmaxforv} for dimension $d=1,2$ can be shown by  Technical Proposition~\ref{tetaxi}.
However, to use \eqref{e:fmaxforv} we use an embedding of type $L^r\embed H^{-1}_2$.
By Sobolev-Rellich embedding theorem, we use the condition 
\DEQSZ
\frac{1}{r} < \frac{1}{2}+\frac{1}{d}. \label{e:sobolev}
\EEQSZ

When $d=3$,  \eqref{e:sobolev} is satisfied if $r>\frac65 $. 
Next we apply Technical proposition \ref{exu}, where we choose the parameters $\tilde{r} = 2, r= \frac65+\epsilon$, $m=2$, $r_2=2$. The parameter   $p^\ast$ is chosen  accordingly to  
\DEQS
\frac{1}{p^\ast} +\frac{1}{r_2} \leq \frac{1}{r}.
\EEQS
Consequently, 
$p^\ast \geq \lk(  \frac13+\epsilon\rk)^{-1}  $, where $\epsilon>0$ is arbitrarily small. Due to the  Sobolev embedding we have  
\DEQSZ
\EE \Vert f\Vert_{L^2(0,T;H^{-1}_2)}^{2}	\leq C_r\EE \Vert f\Vert_{L^2(0,T;L^r)}^{\tilde{r}} \leq C(\kappa)K_2,
\EEQSZ
for some $C_r$ positive.

{\bf Remark:} By the  conditions of Technical Proposition \ref{tetaxi} one can show that  
$$q< \frac{\frac{\aleph}{2}+\frac dr-\frac{d}{p^\ast}}{\frac{d-2\rho}{2}}= \lk(\frac{\aleph}{2}+\frac3r-\frac{3}{p^\ast} \rk)\lk( \frac{2}{3-2\rho} \rk)$$

Choosing $p^*$ arbitrarily large, $\epsilon>0$ arbitrarily small and $\rho<0$ arbitrarily close to -1, then  we get the trade-off between the non-linearity $q$ and non-locality exponent $\aleph$ 
$$q<\frac{\aleph}{5}+1\in[\frac65,\frac75].$$
If we consider an upper bound on $\rho$ to be $\frac{\aleph}{2}-\frac d2$ as in \eqref{e:cond28}, then we will have an 
$$q<\frac{\aleph +5}{6-\aleph}\in[\frac65,\frac74].$$
Thus, in both cases, we have non-linearity $q>1$ for any choice of $\aleph\in[1,2].$

When $d=2$,  \eqref{e:sobolev} is satisfied if $r>1 $.  Next, we apply Technical proposition \ref{tetaxi} with parameter choices $\tilde{r} = 2, r= 1+\epsilon$, $m=2$ where $\epsilon>0$ is arbitrarily small, $r_2=2$ and $p^\ast>1$ such that $\frac12+\frac{1}{p^\ast}\leq \frac{1}{1+\epsilon}$, then the Sobolev embedding theorem gives
\DEQSZ
\EE \Vert f\Vert_{L^2(0,T;H^{-1})}^{2}	\leq C_r\EE \Vert f\Vert_{L^2(0,T;L^r)}^{\tilde{r}} \leq C(\kappa)K_2,
\EEQSZ
for some $C_r>0$.

{\bf Remark: } As we pointed out in the above remark, by   the conditions of Technical Proposition \ref{exu}, we have 
\DEQS
q<\lk(\frac{\aleph}{2}+\frac2r-\frac{2}{p^\ast} \rk)\lk( \frac{2}{2-2\rho} \rk).
\EEQS
Choosing $p^*$ arbitrarily large, $\epsilon>0$ arbitrarily small and $\rho$ arbitrarily close to $\frac{\aleph}{2}-1$ (as it is in \eqref{e:cond28}) yields the trade-off between the non-linearity $q$ and non-locality exponent $\aleph$ 
\DEQS
q<\frac{\aleph+4}{4-\aleph} \in [\frac53,3].
\EEQS

When $d=1$,  \eqref{e:sobolev} is satisfied if $r>\frac12$.  Next, we apply Technical proposition \ref{tetaxi} with parameter choices $\tilde{r} = 2, r= \frac12+\epsilon$, $m=2$ where $\epsilon>0$ is arbitrarily small, $r_2=2$ and $p^\ast>1$ such that $\frac12+\frac{1}{p^\ast}\leq \frac{1}{\frac12+\epsilon}$, then the Sobolev embedding theorem gives
\DEQSZ
\EE \Vert f\Vert_{L^2(0,T;H^{-1})}^{2}	\leq C_r\EE \Vert f\Vert_{L^2(0,T;L^r)}^{\tilde{r}} \leq C(\kappa)K_2,
\EEQSZ
for some $C_r>0$.

{\bf Remark: } As we pointed out in the above remark, by the  conditions of technical proposition \ref{exu}, we have 
\DEQS
q<\lk(\frac{\aleph}{2}+\frac2r-\frac{2}{p^\ast} \rk)\lk( \frac{2}{2-2\rho} \rk).
\EEQS
Choosing $p^*$ arbitrarily large, $\epsilon>0$ arbitrarily small and $\rho$ arbitrarily close to $\frac{\aleph}{2}-1$ (as it is in \eqref{e:cond28}) yields the trade-off between the non-linearity $q$ and non-locality exponent $\aleph$ 
\DEQS
q<\frac{\aleph+4}{2-\aleph} \in [5,\infty).
\EEQS
}

\end{proof}

\begin{proposition}\label{limituk}
Let us assume the triplet $(q,\FE,\rho)$ satisfies the assumptions in Proposition~\ref{reg_uk},
\eqref{cond00pstar} and $q\ge 1$
\DEQSZ\label{cond00q}
 q<\frac {\FE+d}{d-2\rho}.
\EEQSZ 
	Under the conditions of Proposition \ref{reg_uk}, we know for all $\kappa\in\NN$ and  $p\ge 4$ with 
\DEQSZ\label{condp000}
p> \frac {4d}{ \FE +d-dq+2q\rho},
\EEQSZ
 there exists a constant $\lambda=\lambda(\kappa)>0$ such that we have
	\del{
		\DEQS
\frac 1{ 2p}	\EE e ^{-\lambda s} |\uk(t)|_{L^p}^p 
	+\frac p2 	\int_0^ t\EE  e ^{-\lambda s}\la u_{\tesfalem{\kappa}}^{p-2}(s)\nabla \uk(s) ,\nabla u_{\tesfalem{\kappa}}(s)\ra ds
%
%
&\le & |u_0|_{L^p}^p+\tfrac 12\EE \sup_{0\le s\le t}e^{-\lambda s} |\eta (s)|_{L^p}^p.
\EEQS 
}
 
\del{\DEQS
 	\lqq{\EE \sup_{0\leq s\leq t} e ^{-\lambda s} |\uk(s)|_{L^p}^p 
+(p-\epsilon)	\int_0^ t\EE  e ^{-\lambda s}\la u_{\kappa}^{p-2}(s)\nabla \uk(s) ,\nabla u_{\kappa}(s)\ra ds } &&\\
 && {}+  (\lambda-C(\kappa,p,\epsilon_1,\epsilon_2))\int_0^te^{-\lambda s}|u_\kappa(s)|_{L^p}^p\,ds\\
&\le & |u_0|_{L^p}^p+ \epsilon_2\EE \sup_{0\le s\le t}e^{-\lambda s} |\eta (s)|_{L^p}^p.
\EEQS}

\DEQS
\lqq{\EE \sup_{0\leq s\leq t} e ^{-\lambda s} |\uk(s)|_{L^p}^p 
	+C(p)	\int_0^ t\EE  e ^{-\lambda s}\la u_{\kappa}^{p-2}(s)\nabla \uk(s) ,\nabla u_{\kappa}(s)\ra ds } &&\\
&& {}+  \tilde{C}(\lambda,p)\int_0^te^{-\lambda s}|u_\kappa(s)|_{L^p}^p\,ds \leq |u_0|_{L^p}^p+ \tilde{\tilde{C}}\EE \sup_{0\le s\le t}e^{-\lambda s} |\eta (s)|_{L^p}^p,
\EEQS
 where $C(p),\,\tilde{C}(\lambda,p)$ and $\tilde{\tilde{C}}$ are positive constants. 
\end{proposition}
\begin{proof}
First, note that direct calculations gives 
$$ \la u_{\kappa}^{p-2}(s)\nabla \uk(s) ,\nabla \uk(s)\ra =| \uk^{\frac p2-1}(s)\nabla \uk(s)|_{L^2}^2.
$$ 	
	The It\^o formula applied to 
	$\Psi(t)=e^{-\lambda t} |\uk(t)|_{L^p}^p$ 
	gives 
\del{for $\delta=\frac p2$}
	\DEQS
\lqq{ e^{-\lambda t}|\uk(t)|_{L^p}^p +r_1p (p-1)	\int_0^ t e^{-\lambda s} \la u_{\kappa}^{p-2}(s)\nabla \uk(s) ,\nabla u_{\kappa}(s)\ra ds+\lambda\int_0^ t  e^{-\lambda s}|\uk(s)|_{L^p}^p \, ds\qquad\qquad
}
\\
 &=& 
    |u_0|_{L^p}^p   -c_1p	\int_0^ t e^{-\lambda s}\la \uk^{p-1}(s),\phi_\kappa(\eta,\xi,s)\eta(s)|\xi|^q(s)\ra ds\qquad\qquad\qquad\qquad
\\
&&{} + p{\sigma_1}\sum_{k=1}^\infty \int_0^t\int_\CO e^{-\lambda s}|\uk (s, x)|^{p}  \lambda_k^{-\gamma_1/2}\varphi_k (x)\,dx\,d \mathbf{w}^1_k (s) \\
&&{} +  p \int_0^te^{-\lambda s}(a_1 |\uk (s)|_{L^p}^p+ b_1\la \uk ^{p-1}(s),1\ra)\,ds \\
&&{} + \frac{\sigma_1^2}{2}    
	\sum_{k=1}^\infty \int_0^t e^{-\lambda s}\mbox{Tr} D^2\psi(\uk (s, x))[g(\uk ) \varphi_k,g(\uk )\varphi_k]\,ds.
%
\EEQS
 For convenience, we estimate the terms at the right hand side first.  \eqref{itotrace} can be used to to estimate the trace term;
 \DEQS
 \lqq{\frac{\sigma_1^2}{2} \sum_{k=1}^\infty \int_0^t e^{-\lambda s}\mbox{Tr} D^2\psi(\uk (s, x))[g(\uk ) \varphi_k,g(\uk )\varphi_k]\,ds }  && \\
 &\leq& \epsilon\frac{\sigma_1^2}{2} \int_0^t|e^{-\lambda s}u_\kappa(s)|_{H^\theta_p}^p+C(\epsilon)\frac{\sigma_1^2}{2}\int_0^te^{-\lambda s}|u_\kappa(s)|_{L^p}^p .
 \EEQS
 Also, we estimate the term involving $\eta|\xi|^q$ for $\delta=\frac p2$ as,
 \DEQS
 	\lqq{c_1p\int_0^ t e^{-\lambda s}\la \uk^{p-1}(s),\phi_\kappa(\eta,\xi,s)\eta(s)|\xi|^q(s)\ra ds }&&\\
 	&\le & c_1p	\int_0^ t e^{-\lambda s}\la \uk^{p-1}(s),\eta(s)|\xi|^q(s)\ra ds
 	\\
 	&\le & c_1p	\int_0^ t e^{-\lambda s} |\nabla \uk^{\delta }(s)|_{L^2}
 	|\nabla^{-1}(\uk^{p-1-\delta}\eta(s)|\xi|^q(s)|_{L^2} ds
 	\\
 	&\le & \ep 	\int_0^ t e^{-\lambda s}  |u^{\delta-1 }(s)\nabla u(s)|_{L^2}^2
 	\, ds+ 
 	C_pC_\ep 	\int_0^ t e^{-\lambda s}|u^{p-1-\delta }(s)\eta(s)|\xi|^q(s)|_{L^1}^2ds.
 \EEQS

 We then have
\DEQS
\lqq{ e^{-\lambda t}|\uk(t)|_{L^p}^p +r_1p (p-1)	\int_0^ te^{-\lambda s} \la u_{\kappa}^{p-2}(s)\nabla \uk(s) ,\nabla u_{\kappa}(s)\ra ds+\lambda\int_0^ t  e^{-\lambda s}|\uk(s)|_{L^p}^p \, ds
}&&\\
&\le & |u_0|_{L^p}^p-c_1 p	\int_0^ t e^{-\lambda s} |\nabla \uk^{\delta }(s)|_{L^2}
|\nabla^{-1}(\uk^{p-1-\delta}\eta(s)|\xi|^q(s)|_{L^2} ds  \\
&&{}  + p{\sigma_1}\sum_{k=1}^\infty \int_0^t\int_\CO e^{-\lambda s}|\uk (s, x)|^{p}  \lambda_k^{-\gamma_1/2}\varphi_k (x)\,dx\,d \mathbf{w}^1_k (s) \\
&&{}+ (pa_1+1)\int_0^t e^{-\lambda s}|\uk (s)|_{L^p}^{p}\,ds-\frac{b_1^pe^{-\lambda t}}{\lambda} \\
&&{}+ \epsilon \int_0^te^{-\lambda s}|u_\kappa(s)|_{H^\theta_p}^p+C(\epsilon)\int_0^te^{-\lambda s}|u_\kappa(s)|_{L^p}^p 
\\
&\le &  |u_0|_{L^p}^p+\ep 	\int_0^ t e^{-\lambda s}  |u^{\delta-1 }(s)\nabla u(s)|_{L^2}^2
\, ds+ 
	C_pC_\ep 	\int_0^ t e^{-\lambda s}|u^{p-1-\delta }(s)\eta(s)|\xi|^q(s)|_{L^1}^2ds\\
	&&{}  + p{\sigma_1}\sum_{k=1}^\infty \int_0^t\int_\CO e^{-\lambda s}|\uk (s, x)|^{p}  \lambda_k^{-\gamma_1/2}\varphi_k (x)\,dx\,d \mathbf{w}^1_k (s)  \\
	&&{}+ (pa_1+1)\int_0^t e^{-\lambda s}|\uk (s)|_{L^p}^{p}\,ds + \epsilon \int_0^t|e^{-\lambda s}u_\kappa(s)|_{H^\theta_p}^p+C(\epsilon)\int_0^te^{-\lambda s}|u_\kappa(s)|_{L^p}^p .
\EEQS
The H\"older inequality with $m=\frac {2p}{p-2}$ and $m'=\frac {2p}{p+2}$ gives
	\DEQS
\lqq{ |\uk(t)|_{L^p}^p +p	\int_0^ t \la u_{\kappa}^{p-2}(s)\nabla \uk(s) ,\nabla u_{\kappa}(s)\ra ds+\lambda\int_0^ t  e^{-\lambda s}|\uk(s)|_{L^p}^p \, ds
}
\\
&\le &  |u_0|_{L^p}^p+\ep 	\int_0^ t e^{-\lambda s}| u^{\delta-1 }(s)\nabla u(s)|_{L^2}^2\, ds+ 
C_pC_\ep 	\int_0^ t e^{-\lambda s} |u^{\frac p2-1}(s)|_{L^\frac{2p}{p-2}}^2 |\eta(s)|\xi|^q(s)|_{L^\frac {2p}{p+2}}^2ds  \\
&&{}  + p{\sigma_1}\sum_{k=1}^\infty \int_0^t\int_\CO e^{-\lambda s}|\uk (s, x)|^{p}  \lambda_k^{-\gamma_1/2}\varphi_k (x)\,dx\,d \mathbf{w}^1_k (s)  \\
&&{}+ (pa_1+1)\int_0^t e^{-\lambda s}|\uk (s)|_{L^p}^{p}\,ds  {+ \epsilon \int_0^te^{-\lambda s}|u_\kappa(s)|_{H^\theta_p}^p+C(\epsilon)\int_0^te^{-\lambda s}|u_\kappa(s)|_{L^p}^p }.
\EEQS
Let $\delta>0$ but small and $m_1$, $m_2$, and $m_3$ chosen such that
$$
(p-2-\delta)m_1=p,\quad 2m_2=p,\quad \delta m_3=p.
$$
Additionally, let us put $\gamma_1=\frac 1{m_1}$, $\gamma_2=\frac 1{m_2}$,  and $\gamma_3=\frac 1{m_3}$.
	In the next step we will argue similarly we have done in Technical Proposition \ref{tetaxi}.
Applying the H\"older inequality gives
\DEQS 
\lqq{ e^{-\lambda t}|\uk(t)|_{L^p}^p +p	\int_0^ te ^{-\lambda s}\la u_{{\kappa}}^{p-2}(s)\nabla \uk(s) ,\nabla u_{\kappa}(s)\ra ds+{\lambda}\int_0^ t  e^{-\lambda s}|\uk(s)|_{L^p}^p \, ds
}
\\
&\le &  {|u_0|_{L^p}^p}+ \ep 	\int_0^ te^{-\lambda s} | u^{\delta-1 }(s)\nabla u(s)|_{L^2}^2\, ds
\\
&&{} + 
C_pC_\ep \sup_{0\le s\le t}e^{-\lambda\gamma_1 s}|u(s)|_{L^p}^{p-2-\delta }
 \sup_{0\le s\le t}e^{-\lambda\gamma_3 s} |\eta (s)|_{L^p}
	\int_0^ t e^{-\lambda\gamma_2 s}|u(s)|_{L^p}^\delta  ||\xi|^q(s)|_{L^\frac{2p}{p-4}}^2ds \\
	&&{}  {+ p{\sigma_1}\sum_{k=1}^\infty \int_0^t\int_\CO e^{-\lambda s}|\uk (s, x)|^{p}  \lambda_k^{-\gamma_1/2}\varphi_k (x)\,dx\,d \mathbf{w}^1_k (s)}  \\
	&&{}{+ (pa_1+1)\int_0^t e^{-\lambda s}|\uk (s)|_{L^p}^{p}\,ds } {+ \epsilon \int_0^te^{-\lambda s}|u_\kappa(s)|_{H^\theta_p}^p+C(\epsilon)\int_0^te^{-\lambda s}|u_\kappa(s)|_{L^p}^p }
\\
&\le &  {|u_0|_{L^p}^p}+ \ep 	\int_0^ te^{-\lambda s} | u^{\delta-1 }(s)\nabla u(s)|_{L^2}^2\, ds
\\
&&{} + 
C_pC_\ep \sup_{0\le s\le t}e^{-\lambda\gamma_1 s}|u(s)|_{L^p}^{p-2-\delta }
\sup_{0\le s\le t}e^{-\lambda\gamma_3 s} |\eta (s)|_{L^p}
\\
&&\qquad \times {}\lk( \int_0^ t e^{-\lambda s}|u(s)|_{L^p}^p\, ds\rk)^\frac \delta p 
\lk( \int_0^ t |\xi(s)|_{L^\frac{2pq}{p-4}}^{2q\frac p{p-\delta} }\, ds\rk)^\frac {p-\delta}p \\
&&{}  {+ p{\sigma_1}\sum_{k=1}^\infty \int_0^t\int_\CO e^{-\lambda s}|\uk (s, x)|^{p}  \lambda_k^{-\gamma_1/2}\varphi_k (x)\,dx\,d \mathbf{w}^1_k (s)}  \\
&&{}{+ (pa_1+1)\int_0^t e^{-\lambda s}|\uk (s)|_{L^p}^{p}\,ds } {+ \epsilon \int_0^te^{-\lambda s}|u_\kappa(s)|_{H^\theta_p}^p+C(\epsilon)\int_0^te^{-\lambda s}|u_\kappa(s)|_{L^p}^p  }.
	\EEQS
	The Young inequality gives
\DEQSZ
\lqq{e^{-\lambda t} |\uk(t)|_{L^p}^p +p	\int_0^ te ^{-\lambda s}\la u_{{\kappa}}^{p-2}(s)\nabla \uk(s) ,\nabla u_{{\kappa}}(s)\ra ds   { + \lambda \int_0^ t  e^{-\lambda s}|\uk(s)|_{L^p}^p \, ds} } && \nonumber
\\
&\le & {|u_0|_{L^p}^p} + \ep 	\int_0^ te^{-\lambda s} | u^{\delta-1 }(s)\nabla u(s)|_{L^2}^2\, ds \nonumber
\\
&&{} + 
\ep_1  \sup_{0\le s\le t}e^{-\lambda s}|u(s)|_{L^p}^{p}+\ep_2 
\sup_{0\le s\le t}e^{-\lambda s} |\eta (s)|_{L^p}^p \nonumber
\\
&&{} +C(\ep_1,\ep_2) \lk( \int_0^ t e^{-\lambda s}|u(s)|_{L^p}^p\, ds\rk)
\lk( \int_0^ t |\xi(s)|_{L^\frac{2pq}{p-4}}^{2q\frac p{p-\delta} }\, ds\rk)^{\frac {p-\delta}\delta } \nonumber\\
&&{}  {+ p{\sigma_1}\sum_{k=1}^\infty \int_0^t\int_\CO e^{-\lambda s}|\uk (s, x)|^{p}  \lambda_k^{-\gamma_1/2}\varphi_k (x)\,dx\,d \mathbf{w}^1_k (s)}  \nonumber \\
&&{}{+ (pa_1+1)\int_0^t e^{-\lambda s}|\uk (s)|_{L^p}^{p}\,ds } {+ \epsilon \int_0^te^{-\lambda s}|u_\kappa(s)|_{H^\theta_p}^p+C(\epsilon)\int_0^te^{-\lambda s}|u_\kappa(s)|_{L^p}^p }  \label{e:test} .
\EEQSZ
The last term on the right hand side can be estimated by Proposition \ref{interp_rho}. In particular, we need
\DEQSZ\label{wenned001}
\frac d2-\rho\le \frac \FE 2\frac {2(p-\delta)}{2qp}
+\frac {d(p-4)}{2pq}.
\EEQSZ 
If $p>4$ is large enough and \eqref{cond00q} is satisfied,
then there exists a number $p$ and $\delta$ such that \eqref{wenned001} is satisfied.
Since $\frac {2(p-\delta)}{2qp}\uparrow \frac 1q$ for $\delta\to 0$ and $\frac {d(p-4)}{2pq}\uparrow \frac d{2q}$ for $p\to\infty$, $\delta$ has to be chosen small enough and $p$ large enough to get the inequality satisfied for the given $\rho,\FE,$ and $q$.

\noindent {Next, applying the expectation and estimating the stochastic integral as in  \eqref{hierendets}, we have for every $\epsilon_1,\,\epsilon_2>0$, $d,\,p,\,\gamma_1$ satisfying $\frac{d}{2}+\frac{d}{p}-\min(\frac{2}{p},\frac{d}{p})<\gamma_1$ and $\theta\in(0,\frac{2}{p})$,  
		\DEQS
		\lqq{\EE p{\sigma_1}\sum_{k=1}^\infty \int_0^t\int_\CO e^{-\lambda s}|\uk (s, x)|^{p}  \lambda_k^{-\gamma_1/2}\varphi_k (x)\,dx\,d \mathbf{w}^1_k (s)}&&\\
		&\le&   C(\ep_1,\ep_2 )\,\EE \Big(\int_0^te^{-\lambda s} \lk|u_\kappa(s)\rk|_{L^p}^{p}\, ds\Big) 
	\\
	&&{} 	+(\ep_1+\ep_2) \EE \sup_{0\le s\le t} e^{-\lambda s}\lk|u_\kappa(s)\rk|_{L^p}^{p }  
		+\ep_2 \EE\lk(
		\int_0^t e^{-\lambda s}|u_\kappa(s)|_{H^{\theta}_p} ^{p} \,ds\rk) 
		\EEQS
	}

\noindent We now apply $\EE$ to \eqref{e:test}, collect similar terms terms and use the embedding $H^1_p\embed H^\theta_p$ for $\theta\in(0,\frac{2}{p})$. Hence we get 
	\DEQS
\lqq{
(1-\ep_1)	\EE e ^{-\lambda s} |\uk(t)|_{L^p}^p +(p-\ep) 	\int_0^ t\EE  e ^{-\lambda s}\la u_{ {\kappa}}^{p-2}(s)\nabla \uk(s) ,\nabla u_{ {\kappa}}(s)\ra ds}
\\
&&{}+(\lambda - C(\kappa,\ep_1,\ep_2))\int_0^ t  e^{-\lambda s}|\uk(s)|_{L^p}^p \, ds
\\
&\le & |\uk(0)|_{L^p}^p+\ep_2\EE \sup_{0\le s\le t}e^{-\lambda s} |\eta (s)|_{L^p}^p.
\EEQS 
\del{Taking $\ep_1,\ep_2<1$, $\lambda >$, and $K_2$ so large that
$$
K_2\le \frac 1{1-\ep_1}( )\ep_2 K_2,
$$
then, 	}
\end{proof}

\bigskip 
\subsection{\textbf{Compactness of the fixed point operator}}\label{ss:compactness}\hfill\\

\noindent In the following, we apply the version of the Aubin-Lions Lemma provided in Theorem~\ref{th-gutman} to establish the compactness of the mapping $\mathcal{V}_{\mathfrak{A}}(\eta,\xi)$ restricted to the set  $\mathcal{K}_{\mathfrak{A}}(K_1, K_2, K_3) \hookrightarrow L^2(0, T; L^2(\CO)) \times \mathbb{H}_{\rho, \aleph}$ for any choice of $K_1, K_2, K_3 > 0$. 
From this result, it follows
 the compactness of the fixed point operator. 
Recall that $\mathcal{V}_{\mathfrak{A}}(\eta,\xi)=(u_{\kappa},v_{\kappa})$. 


\begin{proposition}\label{semigroup_compact}
Under the conditions of Proposition \ref{reg_uk}, the following holds.
	\begin{enumerate}
		\item[(i)]
		For all initial conditions $u_0$ and $v_0$, for all $\kappa\in\NN$, and all constants $K_1,K_2>0$, and $K_3>0$, there exist constants $C_1,C_2>0$ and $ \delta_0,\alpha_0>0$, 
		such that for any  pair of two processes $(\eta,\xi)$  belonging to $\spK$ the solutions $\uk$ and $\vk$ of equations
		\eqref{sysu}  and \eqref{sysw}, respectively,
		satisfy
		\begin{enumerate}
			\item[(a)] $$%
			\EE \lk\| \uk\rk\|_{ \WW^ {\alpha_0}_2(0,T ;H^{-2}_2)}+ \EE \lk\| \uk\rk\|_{L^2(0,T ;H^{\delta_0}_2)}<C_1,
			$$
			and
			\item[(b)]
			$$
			\EE\lk\| \vk\rk\|_{ \WW^ {\alpha_0}_2(0,T ;H^{-2}_{2})}+\EE \lk\| \vk\rk\|_{L^2(0,T;H^{\rho+\delta_0+\aleph/2}_{2})}  \le C_2.
			$$
		\end{enumerate}
		\item[(ii)]
		For all initial conditions $u_0$ and $v_0$, for all $\kappa\in\NN$, and all constants $K_1,K_2>0$, and $K_3>0$, there exist a constant $C_3>0$ and $ \rho_0,\alpha_0>0$, 
		such that for any  pair of two processes $(\eta,\xi)$  belonging to $\spK$ the solutions $\uk$ and $\vk$ of equations
		\eqref{sysu}  and \eqref{sysw}, respectively,
		satisfy
		$$
		\EE \lk\| \vk
		\rk\|_{C^{(\alpha_0)}_b(0,T;H^{-2}_{2})} +
		\sup_{0\le s\le T} \EE\lk| \vk(t)-e^{t\Delta^{\aleph/2}}v_0\rk|_{H^{\rho+\rho_0}_2}  \le C_3.
		$$
	\end{enumerate}
	\footnote{The space $\mathbb{W}^\alpha_2$ as in \cite[Chapter 5]{runst}}
\end{proposition}

\begin{proof}
	We will focus here only on the difficult  terms and omit the linear terms.
	In particular, we assume that the pair $(\uk,\vk)$ solves the following system
	\DEQSZ
	\\
	\notag
	\lk\{\barray d{\uk }(t)&=& [r_1 \,\Delta \uk (t)+a_1 \uk (t)+b_1 - c_1 \phi_\n(\xi,t)\cdot \eta(t)\cdot \,|\xi|^q(t)]\, dt +\sigma_1 \, g(\uk (t))\, dW_1,
	\\
	\uk (0)&=&u_0.
	\earray\rk. \label{sysu1}
	\EEQSZ
	and 
	\DEQSZ
	\\
	\notag
	\lk\{\barray  d{\vk  }(t)&=& [r_2 \,\tesfalem{A} \vk  (t)+a _2 \vk  (t)+b_2+ c_2\phi_\n(\xi, t)\cdot \eta(t)\cdot |\xi|^q(t)]\, dt + \sigma_2 \,g(\vk  (t)) \, dW_2(t), \\
	\vk  (0)&=&v_0\earray\rk.
	\label{sysw1}
	\EEQSZ
	\del{We can write this now as mild solution in the following form
		$$
		\uk(t)=e^{t\Delta}u_0-\int_0^ t e^{\Delta(t-s)} \phi_\n(\xi,s) \eta(s) \,|\xi|^q(s)\, ds + \int_0^ t e^{\Delta(t-s)}\sigma_1 \, g(\uk (t))\, dW_1(s),
		$$
		and
		$$
		\vk(t)=e^{t\Delta}v_0+\int_0^ t e^{\Delta(t-s)} \phi_\n(\xi,s) \eta(s) \,|\xi|^q(s)\, ds + \int_0^ t e^{\Delta(t-s)}\sigma_1 \, g(\vk (t))\, dW_2(s).
		$$}
Finally,  let us remind the definition of $\CK_\MA(K_1,K_2,K_3)$ is given by 
\begin{align*}
	\CK_\MA(K_1,K_2,K_3) &:= \Big\{(\eta,\xi) \in\CMM_\MA (0,T)\mid 
	\\ &
\mathbb{E}\,\|\eta\|^{2}_{{\mathbb{H}_{0,2}}}\leq K_1, \quad {\mathbb{E}\sup_{0< s\le T}e^{-\lambda s}|\eta(s)|_{L^{p^\ast}}^{p^\ast}\le K_2}, \\
	&	\mbox{ and }\mathbb{E}\|\xi\|^{\tesfalem{2}}_{\BH_{\rho, \aleph}}\leq K_3,
	\quad  
	\Big\}.
\end{align*}

and, remember the definition $\phi_\kappa(t,\tesfalem{(\eta,\xi)}):=\psi_\n(h(\tesfalem{(\eta,\xi)},t))$, where $h(\tesfalem{(\eta,\xi)},t )= \|\xi\|_{\mathbb{H}_{\tesfalem{\rho,\aleph}}}$
	%
	%
	\del{\bf why do we need the proposition above:}
	
	\medskip
	Let us start with (i)-(a).
	Let us define the convolution of a process $w$ by
	$$
	\mathfrak{F}(w)(t):= \int_0^ t e^{(t-s)\tesfalem{(r_1\Delta+a_1I)}}w(s)\, ds
	$$
	and
	$$
	\mathfrak{S}_\gamma(w)(t):= \int_0^ t e^{(t-s)\tesfalem{(r_1\Delta+a_1I)}}g_\gamma(w(s))\,dW(s),\quad t\in[0,T].
	$$
	Note that given  $(\eta,\xi)\in\CK(K_1,K_2,K_3)$,
	we have by the variation of constant formula
	\DEQS
	\uk(t)
	&= & e^{-\tesfalem{(r_1\Delta+a_1I)} t}u_0
	-\mathfrak{F}(  \phi_\kappa(\xi,\cdot) \eta\, |\xi|^q)(t)
	+  \sigma_1\mathfrak{S}(  \uk)(t)
	\\
	&=:& I_0(t)+I_1(t)+I_2(t)
	.
	\EEQS
	\begin{remark}
		In Appendix \ref{appb}, we summarize several estimates related to convolution processes. These estimates should be compared with Remark \ref{Prop:2.1-remark-F} and the definition given in \eqref{def_st_con}.
	\end{remark} %

	\noindent Fix some  $\delta_0>0$ with $0<\delta_0<\frac{1}{2}$ 
	and set  
	$$
	B_0=H^{\delta_0}_2(\CO)\qquad \mbox{ and }\qquad B_1=L^2	(\CO).
	$$ 
	Now we will show that there exists some $C>0$ such that for all
	$(\eta,\xi)\in\CK_\MA(K_1,K_2,K_3)$, 
	$$\EE\|\uk\|_{ L^2(0,T;B_0)\cap \WW ^{\alpha_0}_2 (0,T;B_1)}\le C,
	$$
	where $\uk$ solves \ref{nonlinearcutoffu}.
	Let us start with $I_0$. 
	According to estimate \eqref{fracine01} (see \cite[Proposition 3.1]{maxjan}), if 
	$u_0\in L^2(\oO)$ then $I_0 \in H^1_2(0, T; L^2(\oO)) \cap L^2(0, T; H^{\delta_0}_2(\oO))$. 
	Hence 
	\DEQSZ\label{e:comI0estimate1}
	\|I_0(\cdot)\|_{L^2(0,T;L^2)}\le C |u_0|_{L^2}
	\EEQSZ
	and
	\DEQSZ\label{e:comI0estimate2}
	\lk\|\frac d{dt}I_0(\cdot)\rk\|_{L^2(0,T;L^2)}
	\le C \lk\|I_0(\cdot)\rk\|_{H^{1}_2(0,T;L^2)}
	\leq  C |u_0|_{L^2}.
	\EEQSZ
	Furthermore, 
	since we have an analytic semigroup, $\delta_0 < \frac{1}{2}$ yields 
	by \cite[Proposition A.12]{DaPrZa} that 
	\begin{align}\label{e:comI0estimate3}
		|I_0|_{L^2(0,T;H^{\delta_0}_2)}^2
		&= |I_0|_{L^2(0,T;L^2)}^2 + |(-\Delta )^{2\delta_0} I_0|_{L^2(0,T;L^2)}^2
		\\
		&\le C |u_0|_{L^2}^2 + \int_0^T |(-\Delta )^{2\delta_0} I_0(t)|_{L^2}^2 dt \nonumber
		\\
		&\le C |u_0|_{L^2} + C_T \int_0^T t^{-2\delta_0} dt |u_0|_{L^2}^2\le C |u_0|_{L^2}^2.\nonumber
	\end{align}
	By interpolation we know that for $\alpha_0 \in(0,\frac{1}{2})$ we have
	\begin{equation}\label{e:interpol}
		\qquad\qquad  W^{\alpha_0}_2(0,T;B)=[L^2(0,T;B_0),W^1	_2(0,T;B_1)]_{\alpha_0 ,2},
	\end{equation}
	where $B=[B_0,B_1]_{\alpha_0, 2} = H_2^{\delta_0(1-\alpha_0)}$, $B_0=H^{\delta_0}_2(\CO)$, $B_1=L^2(\CO)$, $\alpha_0 \in(0,\frac{1}{2})$.
	Note that when the parameters in the definition of the Besov spaces and Lizorkin-Triebel spaces are the same, the Besov spaces and Lizorkin-Triebel spaces coincide. Specifically, $F_{p,p}^\alpha =B_{p,p}^\alpha $.
	In this way, we get
	$$[ W^{1}_2(0,T;H^{-2}_2(\CO)),L^2(0,T;L^2(\CO))]_{2,\alpha }=\mathbb{W}^{\alpha }_2(0,T;H^{-2\alpha }_2(\CO)).
	$$
	Therefore we know by \eqref{e:comI0estimate1}, \eqref{e:comI0estimate2} and \eqref{e:comI0estimate3}
	that for all $\alpha_0\in(0,\frac{1}{2})$ there exists a constant $C>0$ such that
	\begin{equation}\label{e:I0}
		\qquad\qquad\|I_0\|_{L^2(0,T;H^{\delta_0}_2)}+\|I_0\|_{W^{\alpha_0}_2(0,T;H_2^{\delta_0 (1-\alpha_0)})}\le C |u_0|_{L^2}. 
	\end{equation}
	Due to the condition of the initial condition, the required estimates for $I_0$
	are proven.
	Next, we tackle $I_1$.
	Since the Laplacian satisfies the maximal regularity assumption in $L^2(\CO)$, there exists a constant $C>0$ such that (compare estimate \eqref{maxjan0}) 
	\begin{align*}
		\qquad\qquad\| \mathfrak{F}_A( \eta\, \phi_\kappa(\xi,\cdot) |\xi|^q)\|_{L^2(0,T;H^{\delta} _2)}\le
		C\, \|\eta\, \phi_\kappa(\xi,\cdot)  |\xi|^q\|_{L^2(0,T;H^{\delta-2}_2)} .
	\end{align*}
	The embedding $L^1(\CO)\hookrightarrow  H^{-2+\delta }_2(\CO)$
	for $0<\delta<\frac12 $ gives
	\begin{align*}
		\qquad\qquad	\| \mathfrak{F}_A( \eta\, \phi_\kappa(\xi,\cdot) |\xi|^q)\|_{L^2(0,T;H^{\delta }_2)}\le
		C\, \|\eta\, \phi_\kappa(\xi,\cdot)  |\xi|^q\|_{L^2(0,T;L^1)} .
	\end{align*}
	\noindent Due to Technical Proposition~\ref{tetaxi}, there exists a constant $C(\kappa)>0$ such that for $p^\ast$ satisfying the condition (ii)
	we can write
	\DEQS
	\| \mathfrak{F}_A(\eta \phi_\kappa(\xi,\cdot) |\xi|^q)\|_{L^{\tesfalem{2}}(0,T;H^{\delta}_2)}\le
	C(\kappa) \EE \sup_{0\le s\le T}\,|\eta(s)|_{L^{p^\ast}}^2.
	\EEQS
	Since $(\xi,\eta) \in\CK(K_1,K_2,K_3)$, the RHS is finite.
	Let us turn our attion to the regularity in time.
Taking again $B_1'=H^{-2}_2(\CO)$.
Again, since the Laplacian satisfies the maximal regularity assumption in $H^{-2}_2(\CO)$, there exists a constant $C>0$ such that
\DEQS
\lk\|  \mathfrak{F}\lk(  \phi_\kappa(\xi,\cdot) \eta\, |\xi|^q\rk)\rk\|_{L^2(0,T;B_1')}+
\lk\| \frac d{dt} \mathfrak{F}\lk(  \phi_\kappa(\xi,\cdot) \eta\, |\xi|^q\rk)\rk\|_{L^2(0,T;{B_1'})}\le
C\, \lk\| \phi_\kappa(\xi,\cdot) \eta\, |\xi|^q\rk\|_{L^2(0,T;H^{-2-\frac \aleph 2}_2)}.
\EEQS
Again, we can show by interpolation that 
for any $\alpha\in(0,1)$ we have
\DEQS
\| \mathfrak{F}(  \phi_\kappa(\xi,\cdot) \eta\, |\xi|^q)\|_{\mathbb{W}^{\alpha}(0,T;B_1')}
\le C(\kappa) \EE \sup_{0\le s\le T}\,|\eta(s)|_{L^{p^\ast}}^2\le  C\, C(\kappa)  \, K_2.
\EEQS

\del{By the same argument as before, 
	the right-hand side is finite.
	To be more precise, again, due to Technical Propostion~\ref{tetaxi} we have for the same value of $p^*>2$ satisfying (ii) as before 
	\DEQS
	\lk\| \frac d{dt} \mathfrak{F}_A\lk( \uk \phi_\kappa(\xi,\cdot) |\xi|^q\rk)\rk\|_{L^2(0,T;H^{-\frac d2}_2)}
	\le
	C(\kappa) \EE \sup_{0\le s\le T}\,|\uk(s)|_{L^{p^\ast}}^2.
	\EEQS
	By the same interpolation arguments as before, we have that
	for any $\alpha_0\in[0,1]$, there exists some $C>0$ such that
	\begin{equation}\label{e:I1}
		\qquad\qquad \| \mathfrak{F}( \uk \phi_\kappa(\xi,\cdot) |\xi|^q)\|_{\mathbb{W}^{\alpha_0}(0,T;H^{-\frac d2}_2)}\le C\, C(\kappa) \EE \sup_{0\le s\le T}\,|\uk (s)|_{L^{p^\ast}}^2\le  C\, C(\kappa)  \, K_2.
	\end{equation}}
\noindent The right hand side is finite by the definition of $	\mathcal{K}_\MA(K_1,K_2,K_3)$ and since $(\xi,\eta)\in\mathcal{K}_\MA(K_1,K_2,K_3)$.
	
	\noindent It remains to tackle $I_2$. First, let us define the multiplication operator $g_{\gamma_1}(u):\phi\mapsto u\phi$.
	By Theorem 3.5 in \cite{maxjan}, inequality (3.11),
	we know that \tesfalem{there} exists a $C>0$ such that the process
	$$
	\EE \| \mathfrak{S}_{\gamma_1}(\uk)\|_{\mathbb{W}^\frac 14_2(0,T;H^{-2}_2)}^2\le C\EE \|g_{\gamma_1}(\uk)\|^2_{L^2(0,T;\Upsilon(H,H^{-\frac 32}_2))}
	$$
	and
	$$
	\EE \| \mathfrak{S}_{\gamma_1}(\uk)\|_{L^2(0,T;H^{\frac 14}_2)}^2\le C\EE \|g_{\gamma_1}(\uk)\|^2_{L^2(0,T;\Upsilon(H,H^{-\frac 32}_2))}.
	$$
	Due to Assumption~\ref{wiener}, we know that the for the multiplication operator $g_{\gamma_1}(u)$
	there exists a constant $C>0$ such that
	$$
	|g_{\gamma_1}(u)|_{\Upsilon(H,H^{-\frac 32}_2)}\le C|u|_{L^2},\quad u\in L^2(\CO).
	$$
	
	\medskip
	
	\del{\noindent Let us start with (b). Note, again,  that
	we have by the variation of constant formula
	\DEQSZ 
	\vk(t)\notag
	&= & e^{-r_v\Delta t}v_0
	+\mathfrak{F}(  \phi_\kappa(\xi,\cdot) \eta\, |\xi|^q)(t)
	+  \sigma_2\mathfrak{S}(  \vk)(t)
	\\
	&=:& I_0(t)+I_1(t)+I_2(t)
	.
	\EEQSZ}
	\noindent Take some $\delta_0$ such that  $\frac \aleph 2-\rho-\frac d2 > \delta_0>0$ (observe, in $d\ge 2$, $\rho<0$) and put $B'_0=H^{\delta_0+\rho+\frac \aleph 2}_2(\CO)$.
	Then, we have by the maximal regularity of the Laplacian
	\DEQS
	\| \mathfrak{F}(  \phi_\kappa(\xi,\cdot) \eta\, |\xi|^q)\|_{L^2(0,T;B_0')}\le
	C\, \| \phi_\kappa(\xi,\cdot) \eta\, |\xi|^q\|_{L^2(0,T;H^{\delta_0+\rho-\frac \aleph 2}_2)}
	\EEQS
	As before, due to the embedding $L^1(\CO)\hookrightarrow H^{\delta_0+\rho-\frac \aleph 2}_2(\CO)$ and the technical Lemma \ref{tetaxi}, we have
	\DEQS
	\| \mathfrak{F}(  \phi_\kappa(\xi,\cdot) \eta\, |\xi|^q)\|_{L^2(0,T;B'_0)}\le
	C\, \| \phi_\kappa(\xi,\cdot) \eta\, |\xi|^q\|_{L^2(0,T;L^1)}
	\le C(\kappa) \EE \sup_{0\le s\le T}\,|\eta(s)|_{L^{p^\ast}}^2.
	\EEQS
	Since $(\xi,\eta) \in\CK(K_1,K_2,K_3)$, the RHS is finite.
	Taking again $B_1'=H^{-2}_2(\CO)$.
	Again, since the Laplacian satisfies the maximal regularity assumption in $H^{-2}_2(\CO)$, there exists a constant $C>0$ such that
	\DEQS
	\lk\|  \mathfrak{F}\lk(  \phi_\kappa(\xi,\cdot) \eta\, |\xi|^q\rk)\rk\|_{L^2(0,T;B_1')}+
	\lk\| \frac d{dt} \mathfrak{F}\lk(  \phi_\kappa(\xi,\cdot) \eta\, |\xi|^q\rk)\rk\|_{L^2(0,T;{B_1'})}\le
	C\, \lk\| \phi_\kappa(\xi,\cdot) \eta\, |\xi|^q\rk\|_{L^2(0,T;H^{-2-\frac \aleph 2}_2)}.
	\EEQS
	Again, we can show by interpolation that 
	for any $\alpha\in(0,1)$ we have
	\DEQS
	\| \mathfrak{F}(  \phi_\kappa(\xi,\cdot) \eta\, |\xi|^q)\|_{\mathbb{W}^{\alpha}(0,T;B_1')}
	\le C(\kappa) \EE \sup_{0\le s\le T}\,|\eta(s)|_{L^{p^\ast}}^2.
	\EEQS
Next, we tackle 
$I_2$.
By estimate \eqref{maxjanineq} (compare Theorem 3.5 in \cite{maxjan}, inequality (3.11)),
we know that exists a $C>0$ such that 
the process
we know that exists a $C>0$ such that
the process satisfies for the time regularity
$$
\EE \| \mathfrak{S}_{\gamma_1}(\uk)\|_{\mathbb{W}^\frac 14_2(0,T;H^{-2}_2)}^2\le C\EE \|g_{\gamma_1}(\uk)\|^2_{L^2(0,T;\Upsilon(H,H^{-\frac 32}_2))}
$$
and
$$
\EE \| \mathfrak{S}_{\gamma_1}(\uk)\|_{L^2(0,T;H^{\frac 14}_2)}^2\le C\EE \|g_{\gamma_1}(\uk)\|^2_{L^2(0,T;\Upsilon(H,H^{-\frac 32}_2))}.
$$
Due to Assumption~\ref{wiener}  and estimate \eqref {eq: Hil-Schi2}, 
we know that the for the multiplication operator $g_{\gamma_1}(u)$ for any $\ep>0$, 
there exists a constant $C=C(\ep)>0$ such that
$$
|g_{\gamma_1}(u)|^2_{\Upsilon(H,H^{-\frac 32}_2)}\le C(\ep)|u|^2_{L^2}+\ep |u|^2_{H^1_2},\quad u\in H^1_2(\CO).
$$
It follows that for any $\ep>0$, 
there exists a constant $C=C(\ep)>0$ such that
\begin{align*}
	&\EE \| \mathfrak{S}_{\gamma_1}(\uk)\|_{\mathbb{W}^\frac 14_2(0,T;H^{-2}_2)}^2
	+\EE \| \mathfrak{S}_{\gamma_1}(\uk)\|_{L^2(0,T;H^{\frac 14}_2)}^2	
	\\
	&\leq C(\eps) \EE\sup_{s\in [0, T]} |\uk(s)|_{L^{2}}^{2}  + \eps \EE|\uk|_{L^{2}(0, T; L^{2})}. 
\end{align*}
Collecting everything together, we get (a).

\medskip

\noindent Let us start with (b). Note, again, that by the variation of constants formula, we essentially have an equation of the following type
\DEQSZ\label{definei1i2i3}
\vk(t)\notag
&= & e^{(r_2 A + a_2 I) t}v_0
+c_2\tilde{\mathfrak{F}}( \uk\, \phi_\kappa(\xi,\cdot) |\xi|^q)(t)
+  \sigma_2\tilde{\mathfrak{S}}(\vk)(t)
\\
&=:& J_0(t)+J_1(t)+J_2(t),
\EEQSZ
where 
$$
\tilde{\mathfrak{F}}(w)(t):= \int_0^ t e^{(t-s) (r_2 A + a_2 I)}w(s)\, ds
$$
and
$$
\tilde{\mathfrak{S}}_{\gamma_2}(w)(t):= \int_0^ t e^{(t-s)(r_2 A + a_2 I)}g_{\gamma_2}(w(s))\,dW_2(s),\quad t\in[0,T].
$$
\tesfalem{Assume $\rho\in\mathbb{R}$} and take some $\delta_0 >0$ such that $\delta_0 + \rho\leq 0$ 
and set \[\tilde B_0:=H^{\rho+\frac \aleph 2 +\delta_0}_2(\CO).\]
Then, we have by the maximal regularity of the fractional Laplacian $A$ for a positive constant $C>0$ (compare estimate \eqref{maxjan0})
\DEQS
\| \tilde{\mathfrak{F}}( \tesfalem{\eta} \phi_\kappa(\xi,\cdot) |\xi|^q)\|_{L^2(0,T;\tilde B_0)}\le
C\, \|\tesfalem{\eta} \phi_\kappa(\xi,\cdot) |\xi|^q\|_{L^2(0,T;H^{\rho+\delta_0}_2)}.
\EEQS
As before, due to the embedding $L^2(\CO)\hookrightarrow H^{\delta_0+\rho}_2(\CO)$, we have 
\DEQS
\| \tilde{\mathfrak{F}}( \tesfalem{\eta} \phi_\kappa(\xi,\cdot) |\xi|^q)\|_{L^2(0,T;\tilde B_0)}\le
C\, \|\tesfalem{\eta} \phi_\kappa(\xi,\cdot) |\xi|^q\|_{L^2(0,T;L^2)}.
\EEQS
Due to the  Technical Proposition~\ref{tetaxi}, we have
\DEQS
\EE\| \tilde{\mathfrak{F}}( \tesfalem{\eta} \phi_\kappa(\xi,\cdot) |\xi|^q)\|_{L^2(0,T;\tilde B_0)}^2 
\le C(\kappa) \EE \sup_{0\le s\le T}\,|\tesfalem{\eta}|_{L^{p^\ast}}^2.
\EEQS
By Proposition \ref{tetaxi} 
we know that
$$	\mathbb{E}\sup_{0< s\le T}|\tesfalem{\eta}(s)|_{L^{p^\ast}}^{p^\ast}\le K_2.
$$
It follows that 
the right-hand side is finite.

\noindent Let us now investigate the time regularity. Taking again $\tilde B_1=L^2(\CO)$.
Again, since $A$ satisfies the maximal regularity property in $L^2(\CO)$, there exists a constant $C>0$ such that
\DEQS
\lk\|  \tilde{\mathfrak{F}}\lk( \tesfalem{\eta} \phi_\kappa(\xi,\cdot) |\xi|^q\rk)\rk\|_{L^2(0,T;\tilde B_1)}+
\lk\| \frac d{dt} \tilde{\mathfrak{F}}\lk( \tesfalem{\eta} \phi_\kappa(\xi,\cdot) |\xi|^q\rk)\rk\|_{L^2(0,T;{\tilde B_1})}\le
C\, \lk\|\tesfalem{\eta} \phi_\kappa(\xi,\cdot) |\xi|^q\rk\|_{L^2(0,T;\tilde B_1)}.
\EEQS
Again, we can show as before by interpolation  and the technical Proposition~\ref{tetaxi} that  
for any $\alpha\in(0,1)$ we have
\DEQS
\| \tilde{\mathfrak{F}}(\tesfalem{\eta} \phi_\kappa(\xi,\cdot) |\xi|^q)\|_{\mathbb{W}^{\alpha}(0,T;\tilde B_1)}
\le C(\kappa) \EE \sup_{0\le s\le T}\,|\tesfalem{\eta}(s)|_{L^{p^\ast}}^2.
\EEQS

\medskip
\noindent To  tackle $J_2$, we apply Theorem 3.5 in \cite{maxjan} again. We know that exists a $C>0$ such that
the process satisfies for the time regularity
$$\EE \| \mathfrak{S}_{\gamma_2}(\vk)\|_{\mathbb{W}^\frac 14_2(0,T;H^{\rho-2+\frac \aleph2}_2)}^2\le C\EE \|g_{\gamma_2}(\vk)\|^2_{L^2(0,T;\Upsilon(H,H^{\rho-\frac 52+\frac \aleph2}_2))}
$$
and for the compact containment condition, we know for some $\tilde \rho>\rho$
$$\EE \| \mathfrak{S}_{\gamma_2}(\vk)\|_{L^2(0,T;H^{\tilde \delta}_2)}^2\le C\EE \|g_{\gamma_2}(\vk)\|^2_{L^2(0,T;\Upsilon(H,H^{-\frac d2}_2))}.
$$
Observe, that  $\rho<\frac d2-\aleph/2$. Hence we have to estimate $ \|g_{\gamma_2}(\vk)\|^2_{L^2(0,T;\Upsilon(H,H^{-1}_2))}$.
Due to the Assumptions \ref{wiener}, and Proposition \ref{gammaradonv}, we know
$$
|g_{\gamma_2}(v_\kappa)|_{\Upsilon(H,H^{-1}_2)}\le C|\vk|_{H^{\rho+\aleph/2}_2}.
$$
Again, since $\rho>-\frac d2$, due to the Assumptions \ref{wiener}, and Proposition \ref{gammaradonv}, we know that the for the multiplication operator $g_{\gamma_2}$
there exists a constant $C>0$ such that the assumption in
$$|g_{\gamma_2}(v_\kappa)|_{\Upsilon(H,H^{-\frac d2}_2)}\le C|\vk|_{H^{\rho+\aleph/2}_2}.
$$

\medskip
\noindent It remains to show (ii).
Let 
\tesfalem{$J_0,\,J_1$ and $J_2$} be defined by \eqref{definei1i2i3}.
In particular, we have to show that
for some Banach spaces $B_0''$ and $B_1''$ and some $\alpha_0>0$ we have
$B''_0\hookrightarrow H^{\rho}_2(\CO)\hookrightarrow B''_1$, $B''_0\hookrightarrow H^{\rho}_2(\CO)$ compactly,
and for all $(\xi,\eta) \in\CK(K_1,K_2,K_3)$
\del{\DEQSZ\label{s2}
\sup_{0\le t\le T}\EE| \vk(t)-I_0(t)|_{B''_0}\le C  \quad  \mbox{and} \quad \EE\lk\| \vk\rk\|_{C^{(\alpha_0)}_b(0,T;B''_{1})}\le C.
\EEQSZ}

\DEQSZ\label{s2}
\sup_{0\le t\le T}\EE| \vk(t)-\tesfalem{J_0}(t)|_{B''_0}\le C  \quad  \mbox{and} \quad \EE\lk\| \vk\rk\|_{C^{(\alpha_0)}_b(0,T;B''_{1})}\le C.
\EEQSZ

	
\end{proof}

\bigskip 
\subsection{\textbf{Continuity of the fixed point operator}}\label{ss:continuity}\hfill\\

\noindent In this subsection we establish the continuity of the operator $$\mathcal{V}_\kappa:\CK(K_1,K_2,K_3)\to \CM_{\MA,(\rho,\aleph)}^2(0,T).
$$
In fact, we demonstrate an even stronger result. Although we cannot establish Lipschitz continuity, we can prove Hölder continuity. Before presenting the continuity result, we introduce a technical proposition that is essential for our proof.

\begin{tproposition}\label{l1estimate}
Under the condition of Proposition \ref{reg_uk},
\del{Let us assume that $(\rho,\aleph,q,p\ast)$ satisfy
$$
\left(\frac d2-\rho\right) \lk( 1-\frac 1q\rk) \le \frac \aleph  q-\frac d{p^\ast}
$$ 
and
{\color{red}
$$
\frac \aleph 2 +\frac d2>\rho.
$$}
\todo[inline]{Carefully check how and where this assumption is used!}
}
 there exist constants $\delta_1,\delta_2>0$ such that
for any $K_1,K_2,K_3>0$ and $\kappa\in\NN$	and there is a constant $C(K_1,K_2,K_3,\kappa)>0$ such that for any $(\eta_i,\xi_i)\in \CK(K_1,K_2,K_3)$, $i=1,2$, we have 
\item
\DEQSZ\label{e:techlemB1}	
\lqq{ 	\EE \lk\| \mathfrak{F}_{\mathcal{A}}(\phi_\n(\xi_1,\cdot) \eta_1  \xi_1^{q})- \mathfrak{F}_{\mathcal{A}}(\phi_\n(\xi_2,\cdot) \eta_2  \xi_2^ {q})\rk\|_{\mathbb{H}_{\rho,\aleph}}
}\notag
&&
\\
& \le & C(\kappa,K_1,K_2,K_3) \, \lk\{ \lk( \EE \|\eta_1-\eta_2\|_{L^2(0,T;L^{2})}^2\rk)^{\delta_1}+  \lk( \EE \|\xi_1-\xi_2\|^2_{\mathbb{H}_{\rho,\aleph}}
\rk)^{\delta_2}\rk\}.
\EEQSZ.
\end{tproposition}
\begin{remark}\label{shorter}
Analysing the proof of the technical lemmata, we see that under the assumptions of the technical lemma, there exists a constant $ C > 0$  such that:
\DEQS\lqq{ C \|\phi_\n(\xi_1,\cdot) \eta_1  \xi_1^ {q}- \phi_\n(\xi_2,\cdot) \eta_2  \xi_2^ {q}\|_{L^2(0,T;H^{\rho-\aleph/2}_2)} 	}
\\
&\le &
C(\kappa,K_1,K_2,K_3) \, \lk\{ \lk( \EE \|\eta_1-\eta_2\|_{L^2(0,T;L^{2})}^2\rk)^{\delta_1}+  \lk( \EE \|\xi_1-\xi_2\|^2_{\mathbb{H}_{\rho,\aleph}}
\rk)^{\delta_2}\rk\}.
\EEQS.	
\end{remark}
\begin{proof}[\textbf{Proof of Technical \tesfalem{Proposition} \ref{l1estimate}:}]
Observe that for any $n>2$ and $a\geq b\geq 0$ we have by the subadditivity of the $x\mapsto x^\frac{1}{n}$ for $x>0$, and the elementary Young inequality for exponents 
$\frac{n-1}{k-1}$ and $\frac{n-1}{n-k}$, the following inequality  
\DEQSZ\label{inq11}
&&a-b  =  (a^\frac{1}{n} - b^\frac{1}{n}) \sum_{k=1}^n a^ \frac{k-1}{n} b^\frac{n-k}{n} 
\le (a-b)^ \frac{1}{n} \sum_{k=1}^n a^\frac{k-1}{n} b^\frac {n-k}{n}
\le \frac{n}{2} (a-b)^\frac{1}{n}  \lk( a^ \frac {n-1} n +  b^ \frac {n-1} n \rk).
\EEQSZ
Before we continue, we define the following stopping times: 
$$\tau_j(\kappa):=\inf\{s>0~|~\|\xi_j(s)\|_{H^\rho_{2}}\ge 2\kappa\}, \qquad j=1,2.$$ 
\noindent We set $\Omega_1:=\{ \tau_2(\kappa)\le \tau_1(\kappa)\}$ and 
$\Omega_2:= \Omega_1^c$,  
and start with the preliminary observation that since 
$\phi_\n((\eta, \xi),t) = \psi(\|\xi\|_{\mathbb{H}_{\rho, \aleph}^t} / \kappa)$ 
only depends on $\xi$, we write in an abuse of notation $\phi_\n(\xi,t) = \phi_\n((\eta, \xi),t)$ 
and obtain 
\begin{align*}
&\phi_\kappa(\xi_1, \cdot)\eta_1|\xi_1|^q-\phi_\kappa(\xi_2, \cdot)\eta_2|\xi_2|^q\\
&\qquad =
(\phi_\kappa(\xi_1, \cdot)-\phi_\kappa(\xi_2, \cdot))\eta_1|\xi_1|^q + (\eta_1-\eta_2)|\xi_1|^q\phi_\n(\cdot ,\xi_2)+ \eta_2 \phi_\n(\xi_2,\cdot )\big(|\xi_1|^q-|\xi_2|^q\big)\\
\end{align*}
\noindent Since the fractal Laplacian $A$ has the maximal $H_2^\rho(\CO)$ and $H_2^{\rho+\aleph/2}(\CO)$--regularity, see \eqref{maxjan0},  
there exists a constant $C>0$ such that
\begin{align}\label{L^1}
\lqq{ \| \mathfrak{F}_{A}(\phi_\n(\xi_1,\cdot) \eta_1  \xi_1^ {q})- \mathfrak{F}_{A}(\phi_\n(\xi_2,\cdot) \eta_2  \xi_2^ {q})\|_{\mathbb{H}_{\rho,\aleph}} }
\\
&\qquad\qquad \notag \le  C \|\phi_\n(\xi_1,\cdot) \eta_1  |\xi_1|^ {q}- \phi_\n(\xi_2,\cdot) \eta_2  |\xi_2|^ {q}\|_{L^2(0,T;H^{\rho-\aleph/2}_2)}.
\end{align}
Due to embedding $L^r(\oO) \hookrightarrow H^{\rho-\frac{\aleph}{2}}(\oO)$ 
we obtain
\begin{align}\label{LL^1}
\lqq{ \| \mathfrak{F}_{A}(\phi_\n(\xi_1,\cdot) \eta_1  \xi_1^ {q})- \mathfrak{F}_{A}(\phi_\n(\xi_2,\cdot) \eta_2  |\xi_2|^ {q})\|_{\mathbb{H}_{\rho,\aleph}} }
\\
&\qquad\qquad\notag  \le  C \|\phi_\n(\xi_1,\cdot) \eta_1  |\xi_1|^ {q}- \phi_\n(\xi_2,\cdot) \eta_2  |\xi_2|^ {q}\|_{L^2(0,T;L^r)}.
\end{align}

\noindent We start with the $\omega$-wise estimates on $\Omega_1$. By the triangle inequality, we get
\DEQS
\lqq{  \mathds{1}_{\Omega_1}\lk\| \mathfrak{F}(\phi_\n(\xi_1,\cdot) \eta_1  \xi_1^ {q})- \mathfrak{F}(\phi_\n(\xi_2,\cdot) \eta_2  \xi_2^ {q})\rk\|_{\mathbb{H}_{\rho, \aleph}^{T\wedge \tau_1(\kappa)}}}
&&	
\\ &\le&  \mathds{1}_{\Omega_1}C\lk(
\| \lk( \phi_\kappa(\xi_1,\cdot) -\phi_\kappa(\xi_2,\cdot)\rk)\eta_1\, |\xi_1|^q \|_{L^2(0,T\wedge \tau_1(\kappa);L^{r})}\rk.
\\
&&\lk.{}+
\|  \phi_\kappa(\xi_2,\cdot)  \lk(\eta_1-\eta_2\rk) |\xi_1|^q \|_{L^2(0,T\wedge \tau_1(\kappa);L^{r})}\rk.\\
&&\lk.{}+  \| \phi_\kappa(\xi_2,\cdot) \lk( |\xi_1|^q -|\xi_2|^q)\eta_2\rk)\|_{L^2(0,T\wedge \tau_1(\kappa);L^{r})}\rk)
\\
&=:& I_1(T)+I_2(T)+I_3(T).
\EEQS
We estimate the terms one by one by one. 

\noindent \textbf{$\mathbf{I_1(T)}$:} 
Recall the definition of the scalar cutoff function $\phi_\kappa$ given in \eqref{h12.def}, which does not depend on the spatial variable. 
By the simple algebraic identity
$a^2-b^2=(a-b)(a+b)$, we have 
\DEQS
I_1(T)
&\le & \mathds{1}_{\Omega_1}C
\| \lk( \sqrt{\phi_\kappa(\xi_1,\cdot)} -\sqrt{\phi_\kappa(\xi_2,\cdot)}\rk) \lk( \sqrt{\phi_\kappa(\xi_1,\cdot)} +\sqrt{\phi_\kappa(\xi_2,\cdot)}\rk)\eta_1\, |\xi_1|^q \|_{L^2(0,T\wedge \tau_1(\kappa);L^{r})}
\\
&\le & C
\mathds{1}_{\Omega_1}C
\| \lk( \sqrt{|\phi_\kappa(\xi_1,\cdot)- \phi_\kappa(\xi_2,\cdot)|}\rk) \lk( \sqrt{\phi_\kappa(\xi_1,\cdot)} +\sqrt{\phi_\kappa(\xi_2,\cdot)}\rk)\eta_1\, |\xi_1|^q \|_{L^2(0,T\wedge \tau_1(\kappa);L^{r})}.
\EEQS
Since $\phi_\n((\eta, \xi),t) = \psi(\|\xi\|_{\mathbb{H}_{\rho, \aleph}^t} / k)$ and $\psi$ is a Lipschitz continuous function, 
\DEQS \sqrt{\phi_\kappa(\xi_1,\cdot)- \phi_\kappa(\xi_2,\cdot)} &\le & |\phi_\kappa(\xi_1,\cdot)- \phi_\kappa(\xi_2,\cdot)|^\frac 12
\\
&\le &C\, \|\xi_1-\xi_2\| _{\mathbb{H}_{\rho,\aleph}}^\frac 12 
.
\EEQS

\noindent Since on $\Omega_1$ we have $\tau_2(\kappa)\leq \tau_2(\kappa)$ 
we have that $\mathds{1}_{\Omega_1}\phi_\n(\xi_2, \cdot)\leq \mathds{1}_{\Omega_1}\phi_\n(\xi_1, \cdot)$, 
which remains true for the square root by monotonicity. Hence 
\DEQS
I_1(T)
&\le &2 C\|\xi_1-\xi_2\| _{\mathbb{H}^{T\wedge \tau_1(\kappa)}_{\rho,\aleph}}^\frac 12 \|\sqrt{\phi_\n(\xi_1, \cdot)}\eta_1\, |\xi_1|^q \|_{{L^2(0,T\wedge \tau_1(\kappa);L^{r})}}.
\EEQS
Let $m=2$ and let $r_2>1$ such that
$$
\frac d2-\rho\le \frac \aleph 2+\frac d{qr_2},\quad \frac{1}{r_2}+\frac{1}{p^*} = \frac{1}{r}
$$
Due to the assumptions on $(\rho,\aleph,q,p\ast)$, it is possible to find $r_2$ for the give $r$.
Now we can apply the Technical Proposition \ref{tetaxi} and obtain
\DEQS
\EE  \mathds{1}_{\Omega_1} |I_1(T)|^ 2 
&\le &2 C(\kappa) C\EE\left(  \|\xi_1-\xi_2\| _{\mathbb{H}^{T\wedge \tau_1(\kappa)}_{\rho,\aleph}} 
\sup_{0\le s\le T}|\eta_1(s)|^{p^\ast}_{L^{p^\ast}} \rk).
\EEQS
H\"older's inequality gives 
\DEQS
\EE  \mathds{1}_{\Omega_1} |I_1(T)|^ 2 
&\le &2 C(\kappa) C\EE\left(  \|\xi_1-\xi_2\| _{\mathbb{H}^{T}_{\rho,\aleph}} ^2\rk)^\frac 12 \lk(\EE 
\sup_{0\le s\le T}|\eta_1(s)|_{L^{p^\ast}}^{2p^\ast} \rk)^\frac 12.
\EEQS
Due to the definition of the set $\CK_\MA^\kappa$, we can write 
\DEQS
\EE  \mathds{1}_{\Omega_1} |I_1(T)|^ 2 
&\le &2 C(\kappa) C\EE\left(  \|\xi_1-\xi_2\| _{\mathbb{H}^{T}_{\rho,\aleph}} ^2\rk)^\frac 12K_2^\frac 12.
\EEQS

\noindent \textbf{$\mathbf{I_2(T)}$: }
To tackle $I_2$ we fix some number $n\in \mathbb{N}$ sufficiently large such that
\begin{equation}\label{e:hilfsparametern}
\frac {n-1}{p^\ast n}+\frac{1}{(p^\ast)'}\le \frac {n-1}{n}, \mbox{ where } (p^\ast)' = \frac{r^\ast}{q}. 
\end{equation}
This is possible since $\frac{1}{p^\ast}+\frac{1}{(p^\ast)'}<1$.
Applying H\"older's inequality in space and in time,
we get 
and by \eqref{inq11}
\DEQS
I_2(T)&\le &
2 \| \lk(\eta_1-\eta_2\rk) \phi_\kappa(\xi_2,\cdot)  |\xi_1|^q \|_{L^2(0,T\wedge {\tau}_1(\kappa);L^{r})}
\\
&\le & n
\| \lk(\eta_1-\eta_2\rk)^\frac 1n  \cdot \lk(\eta^\frac {n-1}n_1+\eta^\frac {n-1}n_2\rk)\phi_\kappa(\xi_2,\cdot)  |\xi_1|^q \|_{L^2(0,T\wedge {\tau}_1(\kappa);L^r)}
\\
&\le &n
\| \lk(\eta_1-\eta_2\rk)^\frac 1n \|_{L^{2n}(0,T\wedge {\tau}_1(\kappa);L^{n})}\| \lk(\eta^\frac {n-1}n_1+\eta^\frac {n-1}n_2\rk)\phi_\kappa(\xi_2,\cdot)  |\xi_1|^q \|_{L^{\frac{2n}{n-1}}(0,T\wedge \tau_1(\kappa);L^{r\frac{n}{n-1} })}
\\
&\le &n
\| \eta_1-\eta_2 \|_{L^{2}(0,T;L^{1})}^\frac 1n \|\lk(\eta_1^\frac {n-1}n+\eta^\frac {n-1}n_2\rk)\phi_\kappa(\xi_2,\cdot)  |\xi_1|^q \|_{L^{\frac{2n}{n-1}}(0,T\wedge \tau_1(\kappa);L^{r\frac{n}{n-1}})}.
\EEQS
Setting in Technical Proposition \ref{tetaxi} with $m=\frac {2n}{n-1}$ and $\eta:=
\eta_1^\frac n{n-1}+ \eta_2^\frac n{n-1}$ instead of $\eta$, we can write
\DEQS
I_2(T)
&\le &C(\kappa)
n
\EE \| \eta_1-\eta_2 \|_{L^{2}(0,T;L^{1})}^\frac 1n \lk(\|\eta_1^\frac n{n-1} \|^{p^\ast}  _{L^\infty(0,T;L^{p^\ast})}+
\|\eta^\frac n{n-1}_2\|_{L^\infty(0,T;L^{p^\ast})}\rk) 
\\
&\le &C(\kappa)n
\EE\| \eta_1-\eta_2 \|_{L^{2}(0,T;L^{2})}^\frac 1n \lk(\|\eta_1\|^{\frac {n-1}n p^\ast}_{L^\infty(0,T;L^{\frac {n-1}n p^\ast})}+\|\eta_2\|_{L^\infty(0,T;L^{\frac {n-1}np^\ast})}^{\frac {n-1}n p^\ast}\rk)
.
\EEQS
By Cauchy Schwarz's inequality we get
\DEQS
\lqq{ \EE[| I_2(T)|^2] \le 
C(\kappa) n
\bigg(\EE\bigg\{ \| \eta_1-\eta_2 \|_{L^{2}(0,T;L^{2})}^2 \bigg\}^\frac{1}{2n} }
\\
&&{}\times  
\bigg\{ \EE \lk(\|\eta_1\|^{p^\ast \frac {2n}{2n-1}\frac n{n-1}}   _{L^\infty(0,T;L^{p^\ast \frac {n-1}n})}+\|\eta_2\|_{L^\infty(0,T;L^{p^\ast \frac {n-1}n})}^{p^\ast\frac {2n}{2n-1}\frac n{n-1}}\rk)\bigg\}^\frac{2n-1}{4n} 
.\EEQS
Due to the definition of the set $\CK_\MA^\kappa$, we can write 
\DEQS
\EE[  \mathds{1}_{\Omega_1}| I_2(T)|^2] \le 
C(\kappa) n K_2^\frac{2n-1}{4n} \,
\bigg(    \EE\bigg\{   \| \eta_1-\eta_2 \|_{L^{2}(0,T;L^{2})}^2 \bigg\}^\frac{1}{2n} 
.\EEQS


\noindent \textbf{$\mathbf{I_3(T)}$:} Again, we use the embedding $L^r(\CO) \hookrightarrow H^{-1}_2(\CO)$ and apply H\"older's inequality such that we have 
\DEQS
I_3(T)
&\le & \| \phi_\kappa(\xi_2,\cdot)  \lk( |\xi_1|^q -|\xi_2|^q\rk)\eta_2\|_{L^2(0,T\wedge \tau_1(\kappa);L^r)}.
\EEQS
The simple algebraic identity $a^q-b^q=c(q)(a-b)(a^{q-1}+b^{q-1})$ gives
\DEQS
I_3(T)
&\le & C(q) \lk\| \phi_\kappa(\xi_2,\cdot)\lk(  \xi_1 -\xi_2\rk)\lk(|\xi|^{q-1}_1+|\xi|^{q-1}_2\rk)\eta_2\rk\|_{L^2(0,T\wedge\tau_1(\kappa);L^r)}.
\EEQS
Using identity \eqref{inq11}, we obtain
\DEQS
I_3(T)
&\le & C(q,n)
\lk\| \phi_\kappa(\xi_2,\cdot)\lk|  \xi_1 -\xi_2\rk|^\frac 1n
\lk(|\xi_1|^{q-1+\frac {n-1}n}+|\xi_2|^{q-1+\frac {n-1}n}\rk)
\eta_2\rk\|_{L^2(0,T\wedge\tau_1(\kappa);L^r)}.
\EEQS
By H\"older's inequality in space 
we get
\DEQS
\lqq{ I_3^2(T)
\le  C(q,n) }
\\
&&\int_0^{T\wedge\tau_1(\kappa)}   \lk|  \xi_1(s) -\xi_2(s)\rk|^\frac 2n_{L^{2} }
\lk| \phi_\kappa(\xi_2,s)
\lk(|\xi_1|^{q-1+\frac {n-1}n}(s)+|\xi_2|^{q-1+\frac {n-1}n}(s)\rk)
\eta_2(s)\rk|_{L^{\frac {2rn} {2n-r}} }   ^2 \, ds.
\EEQS
Applying Cauchy Schwarz's inequality in time gives
\DEQS
\lqq{ \EE  \mathds{1}_{\Omega_1}| I_3^2(T)|^2
\le  C(q,n)\lk( \EE  \int_0^{T\wedge\tau_1(\kappa)}   \lk|  \xi_1(s) -\xi_2(s)\rk|^2_{L^{2} }\, ds\rk)^\frac 1n}
\\
&&{}\times 
\lk( \EE  \int_0^{T\wedge\tau_1(\kappa)} \lk| \phi_\kappa(\xi_2,s)
\lk(|\xi_1|^{q-1+\frac {n-1}n}(s)+|\xi_2|^{q-1+\frac {n-1}n}(s)\rk)
\eta_2(s)\rk|_{L^{\frac {2rn} {2n-r}}   }   ^{2\frac n{n-1}}  \, ds\rk)^\frac {n-1} n.
\EEQS
Due to the assumptions on $(\rho,\aleph,q,p\ast)$, it is possible to find $r_2$ satisfying the Assumptions \eqref{bed0000} and \eqref{bed000} in Technical Proposition \ref{tetaxi}, where we replaced $r$ by $r\frac {2n-1}{2n}$.
Also, note that we have to take $n$ sufficiently large.
Applying now Proposition \ref{tetaxi} with $\lk(\xi^{q-1+\frac {n-1}n}_1(s)+\xi^{q-1+\frac {n-1}n}_2(s)\rk)$ instead of $\xi$ and $\frac {2n}{n-1}$ for $m$  noting that 
$q-1+\frac {n-1}n\le q$,
we get
\DEQS
\EE  \mathds{1}_{\Omega_1} | I_3^2(T)|^2
&\le & C(q,n,\kappa )\lk( \EE  \int_0^{T\wedge\tau_1(\kappa)}   \lk|  \xi_1(s) -\xi_2(s)\rk|^2_{L^{2} }\, ds\rk)^\frac 1n
\\
&&{}\times 
\lk( \EE  \sup_{0\le s\le T}  \lk| 
\eta_2(s)\rk|_{L^{p^\ast}}   ^{2p^\ast }\rk)^\frac {n-1} n.
\EEQS
Due to the definition of the set $\CK_\MA^\kappa$, we can write 
\DEQS
\EE[  \mathds{1}_{\Omega_1}| I_3(T)|^2] \le 
C(\kappa,q,n) K_2^\frac{n-1}{n} \,
\bigg(    \EE\bigg\{   \| \xi_1-\xi_2 \|_{\mathbb{H}_{\rho,\aleph}}^2
\bigg\}^\frac{1}{2n} 
.\EEQS
\noindent Now, we have to go through the estimates on $\Omega_2:=\Omega\setminus \Omega_1$.
Here, we can use the same idea, only the order has to be changed. In particular, we have to decompose $ \mathfrak{F}(\phi_\n(\xi_1,\cdot) \eta_1  \xi_1^ {q})- \mathfrak{F}(\phi_\n(\xi_2,\cdot) \eta_2  \xi_2^ {q})$ in the following way
\DEQS
\lqq{  \mathds{1}_{\Omega_2}\lk\| \mathfrak{F}(\phi_\n(\xi_1,\cdot) \eta_1  |\xi_1|^ {q})- \mathfrak{F}(\phi_\n(\xi_2,\cdot) \eta_2  |\xi_2|^ {q})\rk\|_{\mathbb{H}_{\rho, \aleph}^{T\wedge \tau_1(\kappa)}}}
&&	
\\ &\le&  \mathds{1}_{\Omega_1}C\lk(
\|  \phi_\kappa(\xi_2,\cdot)  \lk(\eta_1-\eta_2\rk) \xi_1^q \|_{L^2(0,T\wedge \tau_1(\kappa);L^{r})}\rk.\\
&&\lk.{}+  \| \phi_\kappa(\xi_2,\cdot) \lk( |\xi_1|^q -|\xi_2|^q)\eta_2\rk)\|_{L^2(0,T\wedge \tau_1(\kappa);L^{r})}\rk.
\\
&& \lk.{}+ \lk( \phi_\kappa(\xi_1,\cdot) -\phi_\kappa(\xi_2,\cdot)\rk)\eta_1\, |\xi_1|^q \|_{L^2(0,T\wedge \tau_1(\kappa);L^{r})}  \rk)
\\
&=:& I_1(T)+I_2(T)+I_3(T).
\EEQS

Now, we use that we are working on $\Omega_2:=\{ \tau_2(\kappa)\le \tau_1(\kappa)\}$. 
Recall due to the definition of  $\CK_\MA^\kappa$ we know that 
$\mathbb{E}\sup_{0< s\le T}|\eta_2(s)|_{L^{p^\ast}}^{p^\ast}\le K_2$. 
\medskip
\noindent In the sequel we mimick the calculation on $\Omega_2$ up to the stopping time ${\tau}_2(\kappa)$. The only difference is that one has to change the order between $\tau_1(\kappa)$ and $\tau_2(\kappa)$. 
The remaining calculations are analogous to the calculations carried out before. 
Collecting altogether yields the assertion \eqref{e:techlemB1}.
\end{proof}	

\bigskip

\noindent With the help of Technical Lemma \ref{l1estimate} we show the following 
proposition.  

\begin{proposition}\label{contu}
Assume \eqref{e:gamma1Spur} for $p=2$.
Then, under the Assumption of Proposition \ref{reg_uk}, for any $K_1,K_2,K_3>0$ and $\kappa\in\NN$ there exist constants $\delta_1,\delta_2>0$, and $C(K_1,K_2,K_3,\kappa)>0$ such that for any $(\eta_1,\xi_1),(\eta_2,\xi_2)\in\CK_\MA(K_1,K_2,K_3)$ and $(\uk^1,\vk^1)$ and $(\uk^2,\vk^2)$ being solution to the system \eqref{sysu} and \eqref{sysw} for some $\kappa \in \NN$ we have
\begin{enumerate}
\item  
\DEQS
\EE  \lk\|\uk^1-\uk^2 \rk\|^2_{L^2(0,T;L^2)}
&	\le& C(K_1,K_2,K_3,\kappa) \,\lk\{ \EE \|\xi_1-\xi_2\|_{\BH_{\rho,\aleph}}^2 \rk\}^{\delta_1},
\EEQS
\item
\DEQS
\EE  \lk\|\vk^1-\vk^2 \rk\|^2_{\BH_{\rho,\aleph}}
&	\le& C(K_1,K_2,K_3,\kappa) \,\lk[ \lk\{ \EE \|\xi_1-\xi_2\|_{\BH_{\rho,\aleph}}^2 \rk\}^{\delta_1}
+
 \lk\{  \EE \|\eta_1-\eta_2\|^2_{L^{2}(0,T;L^2)}\rk\}^{\delta_2}\rk].
\EEQS
\end{enumerate}
\end{proposition}	
\begin{proof} We start with item \textit{a.)}: Let $(\eta_1,\xi_1),(\eta_2,\xi_2)\in\CK_\MA(K_1,K_2,K_3)$ be given. As in the proof of Technical Lemma~\ref{l1estimate} we fix for given $\kappa\in \NN$ the stopping times 
$${\tau_j(\kappa):=\inf\{s>0~|~\|\xi_j(s)\|_{\mathbb{H}_{\rho, \aleph}}\ge 2\kappa\}, \qquad j=1,2,}$$
and $\Omega_1:=\{ \tau_2(\kappa)\le \tau_1(\kappa)\}$ and 
$\Omega_2:=\Omega_1^c$. 
We set $\gamma=1$ and assume that $q<2$. The proof for $q=2$ is the limit case and will be treated separately below. Let $\theta\in(0,1)$ such that for $d=2$ we have $\theta<q-2$. 
The It\^o formula for $\psi(w)=|w|^2_{L^2}$ yields (bearing in mind that $\gamma =1$) 
\DEQS
\lqq{ \sup_{0\le s\le t} 
|\uk^1(s)-\uk^2(s)|_{H^{\gamma-1}_2}^2
+\int_0^t |\uk^1(s)-\uk^2(s)|_{H^{\gamma}_2}^2\, ds
}
&&
\\
&\le &\int_0^{t} \la \uk^1(s)-\uk^2(s),\phi_\kappa(\xi^1,t)\,\eta_1(s)|\xi_1|^q(s)-\phi_\kappa(\xi^2,t)\,\eta_2(s)|\xi_2|^q(s)\ra\, ds 
\\
&&{}+2\sum_{k=1}^\infty
\int_0^ {t}\la \uk^1(s)-\uk^2(s),g_{\gamma_1}(\uk^1(s)-\uk^2(s))\varphi_k \ra d{\bf w}_k(s)
\\
&&{}+2\sum_{k=1}^\infty
\int_0^ {t}\mbox{Tr}\lk[ D^2\psi(\uk^1(s)-\uk^2(s))\lk[g^\ast_{\gamma_1}(\uk^1(s)-\uk^2(s))\varphi_k \rk]\rk]ds
\\ &=:& D_1(t)+ D_2(t)+ D_3(t).
\EEQS
$\mathbf{D_1(t)}: $ 
By duality and Young's inequality we know that for any $\ep>0$ there exists a constant $C(\ep)>0$ such that 
\DEQS
\lqq{  
|D_1(t)|}
&&
\\&\le &
\int_0^{t}| \uk^1(s)-\uk^2(s)|_{H^{1}_2} |\phi_\kappa(\xi_1,t)\,\eta_1(s)|\xi_1|^q(s)-\phi_\kappa(\xi_2,t)\,\eta_1(s)|\xi_1|^q(s)|_{H^{-1}_2} \, ds 
\\
&\le &
\ep \int_0^{t}| \uk^1(s)-\uk^2(s)|_{H^{1}_2}^2\, ds
+C(\ep) \int_0^{t}|\phi_\kappa(\xi_1,t)\,\eta_1(s)|\xi_1|^q(s)-\phi_\kappa(\xi_2,t)\,\eta_1(s)|\xi_1|^q(s)|_{H^{-1}_2}^2 \, ds 
.
\EEQS
The embedding $L^1(\CO)\hookrightarrow H^{-1}_2(\CO)$, 
we get 
\DEQS
\lqq{  
|D_1(t)|)}
&&
\\&\le &
\int_0^{t}| \uk^1(s)-\uk^2(s)|_{H^{1}_2} |\phi_\kappa(\xi_1,t)\,\eta_1(s)|\xi_1|^q(s)-\phi_\kappa(\xi_2,t)\,\eta_1(s)|\xi_1|^q(s)|_{H^{-1}_2} \, ds 
\\
&\le &
\ep \int_0^{t}| \uk^1(s)-\uk^2(s)|_{H^{1}_2}^2\, ds
+C(\ep) \int_0^{t}|\phi_\kappa(\xi_1,t)\,\eta_1(s)|\xi_1|^q(s)-\phi_\kappa(\xi_2,t)\,\eta_1(s)|\xi_1|^q(s)|_{L^1}^2 \, ds 
.
\EEQS
Remark \ref{shorter} with $m=2$, $\rho=0$, $\aleph=2$, gives that
\DEQS
\lqq{  
\EE |D_1(t)|\le \ep \EE \int_0^{t}| \uk^1(s)-\uk^2(s)|_{H^{1}_2}^2\, ds
}
\\&& {}
+
C(\ep,\kappa,K_1,K_2,K_3) \, \lk\{ \lk( \EE \|\eta_1-\eta_2\|_{L^2(0,T;L^{2})}^2\rk)^{\delta_1}+  \lk( \EE \|\xi_1-\xi_2\|^2_{\mathbb{H}_{\rho,\aleph}}
\rk)^{\delta_2}\rk\}.
\EEQS	
\noindent $\mathbf{D_2(t)}: $ Using Burkholder \eqref{Hrhoburkholderl2} for $\rho=0$ gives that for any $\ep>0$, there exists some constant $C(\ep)>0$ such that 
\DEQS
\lqq{  
\EE |D_2(t)|\le \ep \EE \int_0^{t}| \uk^1(s)-\uk^2(s)|_{H^{1}_2}^2\, ds
}
\\&& {}
+
C(\ep)\EE \int_0^{t}| \uk^1(s)-\uk^2(s)|_{L^2}^2\, ds.
\EEQS

\noindent $\mathbf{D_3(t)}: $ Using Burkholder \eqref{itotrace} for $\rho=0$ gives that for any $\ep>0$, there exists some constant $C(\ep)>0$ such that 
\DEQS
\lqq{  
\EE |D_2(t)|\le \ep \EE \int_0^{t}| \uk^1(s)-\uk^2(s)|_{H^{1}_2}^2\, ds
}
\\&& {}
+
C(\ep)\EE \int_0^{t}| \uk^1(s)-\uk^2(s)|_{L^2}^2\, ds.
\EEQS	

Now, let us tackle b:
\del{ The It\^o formula for $\psi(w)=|w|^2_{H^\rho_2}$ yields 
\DEQS
\lqq{ \sup_{0\le s\le t} 
|\vk^1(s)-\vk^2(s)|_{H^{\rho}_2}^2
+\int_0^t |\vk^1(s)-\vk^2(s)|_{H^{\rho+1}_2}^2\, ds
}
&&
\\
&\le &\int_0^{t} \la (-\Delta)^{\frac 12(\rho+\FE)}(\vk^1(s)-\vk^2(s)),(-\Delta)^{\frac 12(\rho+\FE)}(\phi_\kappa(\xi^1,t)\,\eta_1(s)|\xi_1|^q(s)-\phi_\kappa(\xi^2,t)\,\eta_2(s)|\xi_2|^q(s))\ra\, ds 
\\
&&{}+2\sum_{k=1}^\infty
\int_0^ {t}\la(-\Delta)^{\frac 12(\rho+\FE)}( \vk^1(s)-\vk^2(s)),(-\Delta)^{\frac 12(\rho+\FE)}g_{\gamma_2}(\vk^1(s)-\vk^2(s))\varphi_k \ra d{\bf w}_k(s)
\\
&&{}+2\sum_{k=1}^\infty
\int_0^ {t}\mbox{Tr}\lk[ D^2\psi((-\Delta)^{\frac 12(\rho+\FE)}(\vk^1(s)-\vk^2(s)))\lk[g^\ast_{\gamma_2}((-\Delta)^{\frac 12(\rho+\FE)}(\vk^1(s)-\vk^2(s)))\varphi_k \rk]\rk]ds
\\ &=:& D_1(t)+ D_2(t)+ D_3(t).
\EEQS}
Let us start again as in the Prof of Proposition \ref{reg_uk}-(b)
The solution $\vk$ is given by 
\del{\DEQS
	\lqq{ \vk(t)= e^{(r_2\Delta-\alpha_2I) t}v_0+\int_0^ t   e^{(r_2\Delta-\alpha_2I) (t-s)} f(s)\, ds}
	&&\\
	&&{}
	+ \int_0^te^{(r_2\Delta-\alpha_2I)(t-s)}g(\vk(s))\,dW_2(s) + \alpha_1\int_0^ t   e^{(r_2\Delta-\alpha_2I) (t-s)}  \beta_2\, ds,
	\EEQS}
\DEQS
\lqq{ \vk(t)= e^{(r_2A-a_2I) t}v_0+\int_0^ t   e^{(r_2A-a_2I) (t-s)} f(s)\, ds}
&&\\
&&{}
+ \int_0^te^{(r_2A-a_2I)(t-s)}g(\vk(s))\,dW_2(s) ,
\EEQS
where $f$ is defined by
\DEQS
f(t) & = &\phi_\kappa(\xi_1, t) \cdot \eta_1(t)\cdot |\xi_1|^q(t)-\phi_\kappa(\xi_2, t) \cdot \eta_2(t)\cdot |\xi_2|^q(t).
\EEQS
The term 
$$
\int_0^te^{(r_2A-a_2I)(t-s)}g(\vk(s))\,dW_2(s)
$$
can be treated  by standard assumptions, see \ref{appb}.
By Theorem \ref{Th:compact}, it follows that
$$
\|\mathfrak{F}_{r_2\Delta-\alpha_2I}f\| _{\mathbb{H}_{\rho,\aleph}}
\le C \|f\|_{L^2(0,T;H^{\rho-\frac \aleph 2}_2)}.
$$
and Technical Proposition \ref{l1estimate} gives
\DEQS
\EE \|{f}\|_{L^2(0,T;H^{\rho-\frac \aleph 2}_2) }^2
&\le &
C(\kappa,K_1,K_2,K_3) \, \lk\{ \lk( \EE \|\eta_1-\eta_2\|_{L^2(0,T;L^{2})}^2\rk)^{\delta_1}+  \lk( \EE \|\xi_1-\xi_2\|^2_{\mathbb{H}_{\rho,\aleph}}
\rk)^{\delta_2}\rk\}.
\EEQS
\end{proof}

\bigskip 
\section{\textbf{Non-negativity of the solution}}
\newcommand{\uv}{\vk}  
\newcommand{\leb}{\mbox{Leb}}  %

\noindent In this section we establish that the nonnegativity of the initial conditions are maintained.  
As mentioned in the introduction, we prove in the next proposition that, if $\Sigma(u)$ has a certain structure, the solution belongs to $[0,1]$, i.e., it models indeed a concentration. 
\begin{proposition}\label{non_negativity}
Let$(\uk,\vk)$ be a solution to the the system \eqref{nonlinearcutoffu}-\eqref{nonlinearcutoffv}.
Then, 
$$
\PP\lk( \Leb\{ x\in\CO:\uk(x,t)\ge 0\}=0\rk) =1 
\quad \mbox{and}\quad \PP\lk( \Leb\{ x\in\CO:\vk(x,t)\ge 0\}=0\rk) =1 .
$$ 

\end{proposition}
\begin{proof} 
		
		The idea is, we approximate the function $h(x):=\max(x,0)$, obtain the following estimate 
		$$
		\EE \int_\CO h^2(-\uk(x,t))\,dx\le \EE \int_\CO h^2(-\uk(x,0))\,dx+C \EE \int_0^ {\tau\wedge t}   \int_\CO h^2(-\uk(x,s))\,dx\, ds .
		$$
		An application of the Gr\"onwall Lemma gives then non-negativity. 
		In particular, it shows for any $t\in[0,T]$ we know that $\EE \max(-\uk(t,x))=0$ and $\EE \max(-\vk(t,x))=0$. Hence, $\PP\times \leb$-a.s. $\uk \ge 0$ and $\vk \ge 0$.

		\medskip
		\newcommand{\cpp}{\mbox{pol}_2}
		\newcommand{\cqq}{\mbox{pol}_1}
		\noindent Let us start with the construction of the approximation of $h$. For $\delta \in (0,\frac 18)$ set 
		$$
		h_\delta(r):=\begin{cases}0 &\mbox{ if } r\in (-\infty,0 ),
			\\
			\cqq(r),& \mbox{ if } r\in (0,\delta ),
			\\
			\cpp(r) , &\mbox{ if } r\in (\delta,1 ),
			\\
			\frac {r^2}{r+\delta} & \mbox{ otherwise},
		\end{cases}
		$$
		where 
		\DEQS
		\cqq(r) &=&\frac{ 1 + \delta + 2 \delta^2}{2 \delta (1 + \delta)^2}
		\,r^ 2, 
		\\
		\cpp(r)&=& \frac { \delta-2r-\delta r^2 +2\delta ^2r^2}{2(\delta-1)(1+\delta)^2}.
		\EEQS 		
		We choose the coefficients of the polynomial $\cpp$  such that 
		$$
		\cpp(r)\Bigg|_{r=1}=	\frac {r^2}{r+\delta} \Bigg|_{r=1},
		\qquad \frac d{dr}\cpp(r) \Bigg|_{r=1}=	\frac d{dr}\lk( \frac {r^2}{r+\delta}\rk)  \Bigg|_{r=1}.
		$$
		Furthermore, we  have 
		\DEQSZ\label{prop}
		\lk|	\frac {r^2}{r+\delta}-r\rk|\le \frac {\delta r}{\delta+r},\quad \lk| \frac d{dr} 	\frac {r^2}{r+\delta}-1\rk|\le  \frac {\delta ^2}{(\delta+r)^2},\quad r\ge 1,
		\\\label{prop1}
		|\cpp(r)-r|\le C\, \delta , \quad |\cpp'(r)-1|\le \delta^2, \quad \cpp''(r)=\frac {\delta(1-2\delta)}{(1-\delta)(1+\delta)^2} , \quad r\in[\delta,1].\phantom{\Bigg|}
		\EEQSZ 
		The polynomial $\cqq$ is constructed in such a way, that 
		$\cqq(0)=0$, $\cqq'(0)=0$,  $\cqq(\delta)=\cpp(\delta)$, and $\cqq'(\delta)=\cpp'(\delta)$.
		In addition, we know
		\DEQSZ\label{prop2}
		|\cqq(r)-r|\le C \delta, \quad |\cqq'(r)-1|\le \frac {C\delta }{(1 + \delta)^2},\quad { |\cqq''(r)|\le \frac {|2 \delta |}{ (1 + \delta)^2}}, \quad r\in(0,\delta).
		\EEQSZ 
		By this choice, we know that  the function $h_\delta$ is twice differentiable on $\RR$ and for $r\in[1,\infty)$
		we have
		\DEQSZ\label{prop3}
		\\\notag 
		h'_\delta(r)=\frac {r (2 \delta + r)}{(\delta + r)^2},
		\quad  h''_\delta (r)=\frac {2 \delta^2}{(\delta + r)^3},\quad |h_\delta(r)-r|\le \frac {\delta r}{\delta+r},\quad |h'_\delta(r)-1|\le \frac {\delta ^2}{(\delta +r)^2} .
		\EEQSZ 
		It follows from the estimates \eqref{prop},  \eqref{prop1}, \eqref{prop2}, and \eqref{prop3}, 
		that  $h_\delta$ approximates $\max(0,x)$ in $C_b(\RR)$ and $C^{(1)}_b(\RR)$.
		
		\medskip
		\newcommand{\ck}{\bar u_\kappa}
		\noindent Now, let us define $\phi_\delta:L^2(\CO )\to\RR$ by
		$$\phi_\delta(\ck):=\int_\CO h^2_\delta(-\ck  (x))\, dx,\quad \ck  \in L^2(\CO).
		$$
		Then $\phi_\delta$ is twice Gateau differentiable on $L^2(\CO)$ and
		\DEQS 
		D\phi_\delta [w] &=&2\int_\CO h_\delta(\zeta)h'_\delta(\zeta )w(\zeta)\, d\zeta,
		\\
		D^2\phi_\delta[v,w] &=& 2\int_\CO \lk( h_\delta(\zeta ) h''_\delta(\zeta )+(h'_\delta(\zeta))^2\rk) w(\zeta)v(\zeta )\, d\zeta ,\quad w,v\in L^2(\CO).
		\EEQS
		Since $u_0 \in L^2_{+}(\CO),$ i.e., $v_0 \in L^2(\CO)$ and $v_0 \ge 0$ a.e., we know that $\phi_\delta(\ck  (0))=0$.
		We now apply It\^o's formula to \( \phi_\delta(\ck (t)) \). However, to be precise, the Laplace operator \( \Delta \) should be replaced by its Yosida approximation defined by 
		\begin{equation}
			A_\varepsilon(x) = \frac{1}{\varepsilon} (x - J_\varepsilon (x)) = A(1+\varepsilon A)^{-1}(x), \quad \varepsilon > 0, \; x \in H,
		\end{equation}
		where \( J_\varepsilon (x) = (1+\varepsilon A)^{-1}(x) \).  
		Applying Itô's formula to \( \phi_\delta(\ck (t)) \), we first consider $(\uk ^{\eps},v_\n^{\eps})$  as the solution to system   \eqref{nonlinearcutoffu} to \eqref{nonlinearcutoffv} in which the Laplace operator is replaced  %
		by its Yosida approximation, and then taking the limit $\ep \to 0$. For convenience we omit the step and apply the
		the It\^o formula directly  to $\phi_\delta(\uk (t))$ obtaining
		\newcommand{\DeltaA}{\Delta}
\del{	\begin{align*}
			\EE \phi_\delta(\uk (t))
			&=
			\phi_\delta(\uk (0))+\EE \int_0^t \Big \langle Dh_\delta(-\uk (s)),r_1 \DeltaA \uk (s)-c_1 \uk (s) \xi ^2(s)) \Big \rangle \, ds \notag
			\\
			&
			{} +
			\EE  \sum_{k=1}^\infty \lambda_k \int_0^ t \lk( \int_{\CO}\lk(  h_\delta(-\uk(s,x) )h_\delta''(-\uk  (s,x))+(h_\delta'(-\uk(s,x))^2\rk) \rk.
			\\ & \qquad {}\lk. \times \tilde \sigma(\uk  (s,x))\phi_k(x)\, \tilde \sigma(\uk(s,x))\phi_k(x) \, dx\rk) \, ds\Big]
			\\
			&{}	+2\EE \int_0^t \sum_{k=1}^ \infty \int_\CO h_\delta(-\uk(s,x))\ h'_\delta(-c\uk(s,x))\,  \sigma(\uk (s,x))\phi_k(x)\, dx\,  d\beta_k(s)
			\\
			&=I_1+I_2+I_3+I_4.
		\end{align*}
	
}
		\begin{align*}
		\EE \phi_\delta(\uk (t))
		&=
		\phi_\delta(\uk (0))+\EE \int_0^t \Big \langle D{\phi_\delta(\uk (s))},r_1 \DeltaA \uk (s)-c_1 \uk (s) {\vk^q}(s)) \Big \rangle \, ds \notag
		\\
		&
		{} +
		\EE  \sum_{k=1}^\infty \lambda_k \int_0^ t \lk( \int_{\CO}\lk(  h_\delta(-\uk(s,x) )h_\delta''(-\uk  (s,x))+(h_\delta'(-\uk(s,x))^2\rk) \rk.
		\\ & \qquad {}\lk. \times {g_{\gamma_1}}(\uk  (s,x)){\varphi_k}(x)\, {g_{\gamma_1}}(\uk(s,x)){\varphi}_k(x) \, dx\rk) \, ds\Big]
		\\
		&{}	+2\EE \int_0^t \sum_{k=1}^ \infty \int_\CO h_\delta(-\uk(s,x))\ h'_\delta(-c\uk(s,x))\, {g_{\gamma_1}}(\uk (s,x)){\varphi_k(x)\, dx\,  dw^1_k(s)}
		\\
		&=I_1+I_2+I_3+I_4.
	\end{align*}
	\noindent As $u_0 \in L^2_{+}(\CO),$ i.e., $u_0 \in L^2(\CO),\,\,u_0 \ge 0$ a.e. So $\phi_\delta(\uk (0))=0$.
		Now using Neumann boundary conditions on $(\uk ,\,\uv )$,
		and integration by parts
		we get
		\DEQS I_1&=&	-  \EE \Big[ \int_0^t \int_{\CO} \big( (h_\delta')^2+h_\delta h''_\delta\big) (-\ck  (s,x))( \nabla  \ck (s,x) )^2\, dx\,ds
		\\
		&=& 	-  \EE \Big[ \int_0^t \int_{\CO}(h_\delta')^2(-\ck  (s,x))( \nabla  \ck (s,x) )^2\, dx\, ds \mbox{ is fine}
		\EEQS
		Now, since $h_\delta h''_\delta\ge 0$  and  $1\le (h_\delta')^2$ on $[-\delta,\infty)$, we can write 
		\DEQS\EE \Big[ \int_0^t \int_{\CO}
		\mathds{1}_{\{ c(s,x)<0\}} ( \nabla  \ck (s,x) )^2\, dx
		\le \EE \Big[ \int_0^t \int_{\CO} \big( (h_\delta')^2+h_\delta h''_\delta\big) (-\ck  (s,x))( \nabla  \ck (s,x) )^2\, dx.
		\EEQS 
		Taking intyo account that $(\nabla \ck)^2\ge 0$, and the minus sign infront of the term, it follows that
		$$
		\EE \int_0^t \langle Dh_\delta(\uk (s)),r_1 \DeltaA \uk (s)\ra \, ds \le 0
		$$
		as $\delta \to 0$.
		Let us evaluate the second term, i.e. $I_2$. {Taking into account that $|\uv|^q$ is non-negative, we can write}
		%
		\del{\DEQS
		\lqq{ \EE \int_0^t  \langle Dh_\delta(\uk (s)),c_1 \uk (s) \xi ^2(s) \Big \rangle \, ds }
		\\
&\le & 2 c_1\EE \int_0^t  \langle  h_\delta(-\uk(s))h'_\delta(-\uk(s))  \uk (s)   \vk ^2(s)   \rangle \, ds
		\\
		&\le &  c_1\EE \int_0^t  \langle  \mathds{1}_{ \{ \ck\le 0\} }\uk (s)   \vk ^2(s)   \rangle \, ds
		\\
		&\le & - c_1\EE \int_0^t  \langle  \mathds{1}_{ \{ \ck\le 0\} } |\uk (s)|   \uv ^2(s)  \rangle \, ds\le 0.
		\EEQS
		}
		\DEQS
		\lqq{ \EE \int_0^t  \langle D{\phi_\delta}(\uk (s)),c_1 \uk (s){ |\vk| ^q(s)} \Big \rangle \, ds }
		\\
		&\le & 2 c_1\EE \int_0^t  { \int_{\CO} }  h_\delta(-\uk(s,x))h'_\delta(-\uk(s,x))  \uk (s,x)  {|\vk| ^q(s,x)}   \,{dx} \, ds
		\\
		&\le &  c_1\EE \int_0^t { \int_{\CO}}  \mathds{1}_{ \{ \ck(s,x)\le 0\} }\uk (s,x)  {|\vk| ^q(s,x)}  \,{dx} \, ds
		\\
		&\le & - c_1\EE \int_0^t { \int_{\CO}} \mathds{1}_{ \{ \ck(s,x)\le 0\} } |\uk (s,x)|  {|\uv| ^q(s,x)}  {\, dx} \, ds\le 0.
		\EEQS
		Finally, let us consider $I_3$. Note, that 
		$$\lk( h^2_\delta(\ck)\rk)''=2\,\lk(  h_\delta(\ck) h'_\delta(\ck)\rk)'=
		2\,\lk( \lk( h'_\delta(\ck)\rk)^2 +  h''_\delta(\ck) h_\delta(\ck)\rk).
		$$
		Next, it follows from \eqref{prop}, \eqref{prop1}, \eqref{prop2}, and \eqref{prop3}
		\DEQS
		\lk( h'_\delta(-r)\rk)^2 &\to& \mathds{1}_{\{ r\le 0\}},
		\\
		h''_\delta(-r) h_\delta(-r) &\le   & C\mathds{1}_{\{ |r|\le \delta \}}+\delta^2 \mathds{1}_{\{ r\le -1 \}}.
		\EEQS 
		Hence, we obtain
		\del{\DEQS
		\lqq{ 
			\EE  \sum_{k=1}^\infty \lambda_k \int_0^ t \lk( \int_{\CO} \lk( \lk( h'_\delta(\ck)\rk)^2 +  h''_\delta(\ck) h_\delta(\ck)\rk)(\ck  (s,x))\sigma(\ck  (s,x))\phi_k(x)\,  \sigma(\ck (s,x))\phi_k(x) \, dx\rk) \, ds}
		\\
		&&
		\frac  12\, \lambda_k\,   \EE \sum_{k\in\mathbb{N}} \int_0^ t \int_\CO \lk( C\mathds{1}_{\{ |\ck(s,x)|\le \delta \}}+\delta^2 \mathds{1}_{\{ \ck(s,x) \le -1 \}}\rk) \,\tilde \Sigma( \ck (s,x) )\, \phi_k(x) )^2 dx\, ds
		\\
		&\le &   \frac  12  \EE  \sum_{k=1}^\infty\lambda_k  \int_0^ t \int_\CO\lk( \mathds{1}_{\{ x\in\CO: -\delta \le  \ck (s,x)\le 0\}}
		+\delta^2  \mathds{1}_{\{ x\in\CO: \ck (s,x)\le -1 \}}\rk) \, \sigma^2( \ck (s,x) )\, \phi_k^2(x)  dx\, ds
		\EEQS
		}
		%
		%
		\DEQS
		\lqq{ 
			\EE  \sum_{k=1}^\infty \lambda_k \int_0^ t \lk( \int_{\CO} \lk( \lk( h'_\delta(\ck)\rk)^2 +  h''_\delta(\ck) h_\delta(\ck)\rk)\rk.}
			\\
			&&{}\lk. \times (\ck  (s,x)){g_{\gamma_1}}(\ck  (s,x)){\varphi_k}(x)\, {g_{\gamma_1}}(\ck (s,x)){\varphi_k}(x) \, dx\rk) \, ds
		\\
		&&
		\frac  12\, \lambda_k\,   \EE \sum_{k\in\mathbb{N}} \int_0^ t \int_\CO \lk( C\mathds{1}_{\{ |\ck(s,x)|\le \delta \}}+\delta^2 \mathds{1}_{\{ \ck(s,x) \le -1 \}}\rk) \,{g_{\gamma_1}}( \ck (s,x) )\, {\varphi_k}(x) )^2 dx\, ds
		\\
		&\le &   \frac  12  \EE  \sum_{k=1}^\infty\lambda_k  \int_0^ t \int_\CO\lk( \mathds{1}_{\{ x\in\CO: -\delta \le  \ck (s,x)\le 0\}}
		+\delta^2  \mathds{1}_{\{ x\in\CO: \ck (s,x)\le -1 \}}\rk) \,{g_{\gamma_1}}^2( \ck (s,x) )\, {\varphi_k}^2(x)  dx\, ds
		\EEQS
Since $ g_{\gamma_1}$ is Lipschitz continuous and $\EE\sup_{0\le s\le T}|\bar u_\kappa(s)|_{L^2}^2$ bounded, we get
		%
		\del{\DEQS
		\lqq{ \EE  \sum_{k=1}^\infty \lambda_k \int_0^ t  \int_{\CO}  \lk( \lk( h'_\delta(\ck)\rk)^2 +  h''_\delta(\ck) h_\delta(\ck)\rk)(\ck  (s,x))}
		\\
		&&{} \sigma(\ck  (s,x))\phi_k(x)\,  \sigma(\ck (s,x))\phi_k(x) \, dx \, ds
		\le    C \delta   t \stackrel{\delta\to 0}{ \longrightarrow }\, 0.
		\EEQS}
		%
		%
		\DEQS
		\lqq{ \EE  \sum_{k=1}^\infty \lambda_k \int_0^ t  \int_{\CO}  \lk( \lk( h'_\delta(\ck)\rk)^2 +  h''_\delta(\ck) h_\delta(\ck)\rk)(\ck  (s,x))}
		\\
		&&{}{g_{\gamma_1}}(\ck  (s,x))\phi_k(x)\,{g_{\gamma_1}}(\ck (s,x))\phi_k(x) \, dx \, ds
		\le    C \delta   t \stackrel{\delta\to 0}{ \longrightarrow }\, 0.
		\EEQS
		The stochastic integral will be zero for all $\delta>0$.
		Letting $\delta \rightarrow 0$ and using Dominated Convergence Theorem, we have
		\DEQS 
		\EE [|\uk ^{-}(t)|_{L^2}^2] &\le& 0,
		\EEQS 
		which implies $$\uk ^{-}(t)=0\quad \mbox{a.e.}\quad t \in [0,T],\,x \in \CO,\,\, \omega \in \Omega. $$

		\noindent To tackle \( \vk \), we need to distinguish between the cases \( \rho \geq 0 \) and \( \rho < 0 \).
		If $\rho\ge 0$, we first look at the system
		\DEQSZ \label{nonlinearcutoffvohne}
		\\ \nonumber
		d{\vk  }(t)&=& [r_2 \,A \vk  (t)+a_2 \vk  (t) + \sigma_2 \,g_{\gamma_2}(\vk  (t)) \, dW_2(t), \nonumber
		\hfill
		\EEQSZ
		and show that the solution $\vk$ is non-negative. Here, we can go along the lines above. Then we use some the comparison principle, see e.g. Kotelenz \cite{KOT}, to include the nonlinearity. Let us consider the case where $\rho<0$. Here, let us note that, since $\rho+\frac \aleph 2\ge 0$, the process $\vk:\Omega\times [0,T]\times \CO\to \RR$ is a function, and, since $\uk: \Omega\times [0,T]\times \CO\to \RR$ is also a function, the nonlinearity
		$$
		\uk{|\vk|^q}: \Omega\times [0,T]\times \CO\to \RR
		$$
		is a non-negative function. Hence, the comparison pinciple still works.

		In case $\rho<0$, the initial condition is not a function. However, the space $L^2(\CO)$ is dense in $H^{-\rho}_2(\CO)$. Therefore we can find a sequence of functions $(v^n_0)_{n\in\NN}$ in $L^2(\CO)$ approximating $v_0$ and $v^n_0\ge 0$ a.s.. We also suppose that $v_n\le v_{n-1}$. Now we can proceed as above and show that $\PP\lk( \leb\{  x\in\CO:
		\vk^n(t,x)< 0\}=0\rk)=1$ for any $n\in\NN$. By the comparison result (Lemma 1 of \cite{KOT}), we know that $\vk^n\le \vk^{n+1}$. Hence, 
		the limit $\vk$ is non-negative a.s. Now, the comparison principle gives that $\vk$, here $\vk$ is a solution to \eqref{nonlinearcutoffv} is non-negative.

\end{proof}

\bigskip
\section{\textbf{Uniform bounds of the solutions $\bar u_\kappa$ and $\bar v_\kappa$}}\label{s:uniform}

\renewcommand{\uk}{\bar u_\kappa}
\renewcommand{\vk}{\bar v_\kappa}


\noindent In this section, we establish uniform bounds for the cutoff solution \((\uk, \vk)\) of \eqref{nonlinearcutoffu} and \eqref{nonlinearcutoffv} on $[0, T \wedge \bar{\tau}\kappa^\ast))$, as well as for \eqref{eq1} and \eqref{eq2} on $[T \wedge \tau\kappa^\ast, T]$.
In Proposition \ref{reg_u2}, we derive a uniform bound for \(\uk\) in \(\n \in \mathbb{N}\). These calculations are straightforward, as the nonlinear term is negative. Then, in Proposition \eqref{propvarational}, we establish a uniform bound for \(\vk\) in \(\n \in \mathbb{N}\), where the Technical Lemma \ref{dasauch} plays a crucial role.
We use variational methods to establish a uniform bound on \(\uk\) for \(\n \in \mathbb{N}\) in terms of the initial condition  $u_0.$ 
For the uniform bound on \(\vk\) for \(\n \in \mathbb{N}\), we apply semigroup methods, depending on both initial conditions  $u_0$   and  $v_0$.
.

\begin{proposition} \label{reg_u2}
Fix $T>0$.
For any $\n\in\NN$, any $p_0{{>}}2$, and any \tesfalem{$r\ge 1$} 
there exists a constant $C=C(T,p_0,\tesfalem{r}.\n)>0$ 
such that for any 	initial condition $u_0$ with $\displaystyle \mathbb{E}|u_0|^{\tesfalem{r}p_0}_{L^{p_0}}<\infty$,  
\del{$\displaystyle \mathbb{E}|u_0|^{mp_0}_{L^{p_0}}<\infty$,}
and any solution $\uk$ to \eqref{sysu}  we have for all $2\le p\le p_0$
\del{\DEQSZ \label{e:boundeduk}
	\lqq{	\mathbb{E}\sup\limits_{0\leq t\leq T}|\uk (t)|_{L^p}^{pm}\,dx + c_1 p\mathbb{E}\lk[ \int_0^T\int_\CO \uk ^{p}(s, x)\bar v_\kappa ^q(s,x)\,dx\,ds\rk]^m }
	&&
	\\
	&&{}+ r_1 p(p-1)\mathbb{E}\lk[\int_0^T\int_\CO|\uk (t,x)|^{p-2}|\nabla \uk (t,x)|^2\,dx\,ds \rk]^m \leq   C\,(\mathbb{E}|u_0 |^{pm}_{L^{p}}+1). \nonumber 
	\EEQSZ }
\DEQSZ \label{e:boundeduk}
\lqq{	\mathbb{E}\sup\limits_{0\leq t\leq T}|\uk (t)|_{L^p}^{p\tesfalem{r}}\,dx + c_1 p\mathbb{E}\lk[ \int_0^T\int_\CO \uk ^{p}(s, x)\bar v_\kappa ^q(s,x)\,dx\,ds\rk]^{\tesfalem{r}} }
&&
\\
&&{}+ r_1 p(p-1)\mathbb{E}\lk[\int_0^T\int_\CO|\uk (t,x)|^{p-2}|\nabla \uk (t,x)|^2\,dx\,ds \rk]^{\tesfalem{r}} \leq   C\,(\mathbb{E}|u_0 |^{p\tesfalem{r}}_{L^{p}}+1). \nonumber 
\EEQSZ 

\end{proposition}

We show uniform bounds of $\uk(t)$ for $t\in[0,T]$ in terms of the initial condition $u_0$ and $v_0$ 
by variational methods.  

\newcommand{\mast}{m^\ast}

\begin{proof}
First, we prove \tesfalem{Proposition \ref{reg_u2}} 
for   $\rr=1$. 
 The case of  $\rr\ge 1$,  is shown similarly.
Set $\Phi(u)=|u|_{L^p}^p$. Note, that for $w,h,h_1,h_2\in L^p(\CO)$, $w$ non--negative, we have 
$$
D\Phi(w)[h]=p\int_\CO |w(x)|^{p-2}w(x) h(x)\, dx,\quad  D^2\Phi(w)[h_1,h_2]=p(p-1)\int_\CO |w(x)|^{p-2}h_1(x)h_2(x)\, dx.
$$
Next, we fix the stopping time
$\tau := \tau_{n, \kappa}:=\inf\{0< t\leq T: {} |u_{\kappa}(t)|_{L^p}^p\ge n \}$. 
For notational convenience, we drop the dependence on $\kappa$. 
By the It\^o formula in \cite[Theorem 2.1]{krylov2010} it follows for all $t\in [0,T\wedge \tau]$ $\mathbb{P}-$a.s.
\DEQS
\lqq{|\uk (t)|^p_{L^p}=|u_0|^p_{L^p}+p{\sigma_1}\sum_{k=1}^\infty\int_0^t\int_\CO|\uk (s, x)|^{p-2}\uk^2 (s, x)  \lambda_k^{-\gamma_1/2}\varphi_k (x)\,dx\,d \mathbf{w}^1_k(s)}
&&
\\
&&-c_1 p\int_0^t\int_\CO |\uk(s, x)|^{p-2} \uk(s, x) \uk(s, x) \vk^q (s, x)\,dx\,ds
\\
&&{}-r_1p(p-1)\int_0^t\int_\CO|\uk (s, x)|^{p-2}|\nabla \uk (s ,x)|^2\,dx\,ds
\\
&&+ p\int_0^t(a_1 |\uk (s)|_{L^p}^p+ b_1\la \uk ^{p-1}(s),1\ra)\, ds
\\
&&{}+\frac{\sigma_1^2}{2}\int_0^t   
\sum_{k=1}^\infty  D^2\Phi(\uk (s, x))[g(\uk ) \varphi_k,g(\uk )\varphi_k]  \,ds.
\EEQS
Here, we have used the spectral representation of the Wiener process given in \eqref{noise_sr} and the representation of $g$ in \eqref{g_sr}.	Using the lower bound of the Stroock-Varopoulos inequality of Proposition~\ref{runst1} to the gradient term, 
Young's inequality and rearranging the terms, we obtain a constant $C>0$ such that 
\DEQSZ
\lqq{|\uk (t)|^p_{L^p} + c_1 p\int_0^t\int_\CO |\uk(s, x)|^{p-1} \uk(s, x)  \vk^q (s, x)\,dx\,ds }
\nonumber\\
&&+ { } \frac{r_1p(p-1)}{C}\int_0^t |\uk (s)|_{H^\theta_p}^{p}\,ds
\nonumber\\
&&\leq \lqq{|\uk (t)|^p_{L^p} + c_1 p\int_0^t\int_\CO |\uk(s, x)|^p \ \vk^q (s, x)\,dx\,ds }
\nonumber\\
&&+ { } r_1p(p-1)\int_0^t\int_\CO|\uk (s, x)^{[\frac{p}{2}-1]} \nabla \uk (x, x)|^2\,dx\,ds
\nonumber\\
&\leq &|u_0|^p_{L^p} +p{\sigma_1}\sum_{k=1}^\infty\int_0^t\int_\CO|\uk (s, x)|^{p} \lambda_k^{-\gamma_1/2}\varphi_k(x)\,dx\,d\mathbf{w}^1_k(s)
\nonumber\\
&& + { } \frac{\sigma_1^2}{2}\int_0^t   
\sum_{k=1}^\infty  D^2\Phi(\uk (s, x))[g(\uk ) \varphi_k,g(\uk )\varphi_k]  \,ds
+(pa_1+1)\int_0^t |\uk (s)|_{L^p}^{p}\,ds+b_1^p t.\label{e:naked}
\EEQSZ
Next, we go over to the supremum, use the non-negativity, 
and take the expectation. With the help \eqref{itotrace} we estimate the trace term for any $\eps>0$ and the associated constant $C(\eps)>0$ 
\DEQSZ \label{tosimplify}
\lqq{\mathbb{E}\sup_{0\le s\le t\wedge \tau}|\uk (s)|^p_{L^p} + c_1 p\mathbb{E}\int_0^{t\wedge \tau}\int_\CO |\uk(s, x)|^{p-1} \uk(s, x) \vk^q (s, x)\,dx\,ds }
\\
&&  + { }r_1 p(p-1)\mathbb{E}\int_0^{t\wedge \tau}\int_\CO|\uk (s,x)|^{p-2}|\nabla \uk (s,x)|^2\,dx\,ds \nonumber\\
&\le &\mathbb{E}|u_0|_{L^p}^p +p{\sigma_1}\mathbb{E}\sup_{0\le s\le {t\wedge \tau}}\Big|\sum_k\int_0^s\int_\CO|\uk (s_1,x)|^{p} \lambda_k^{-\gamma}\psi_k(x)\,dx\,d\mathbf{w}^1_k(s_1)\Big| \nonumber
\\
&& +C(\eps) \mathbb{E}\Big(\int_0^{t\wedge \tau}|\uk (s)|_{L^p}^{p	}\,ds\Big) +\eps \mathbb{E}\Big(\int_0^{t\wedge \tau}|\uk (s)|_{H^\theta_p}^{p}\,ds\Big)
.\nonumber
\EEQSZ
To deal with the stochastic integral, 
we apply  \eqref{hierendets} and get some other constant $C(\eps) >0$ such that for $\theta<\frac 2p$ 
\DEQSZ\label{sint}
\lqq{
p{\sigma_1}\mathbb{E}\sup_{0\le s\le {t\wedge \tau}}\Big|\sum_k\int_0^s\int_\CO|\uk (s_1,x)|^{p} \lambda_k^{-\gamma}\psi_k(x)\,dx\,d\mathbf{w}^1_k(s_1)\Big|}\nonumber&&
\\
\nonumber &
\le&  C(\eps) \,\EE \Big(\int_0^{t\wedge \tau} \lk|\uk(s)\rk|_{L^p}^{p}\, ds\Big) 
+2\eps \EE \sup_{0\le s\le {t\wedge \tau}} \lk|\uk(s)\rk|_{L^p}^{p}
%
+\eps \EE\lk(
\int_0^{t\wedge \tau} |\uk(s)|_{H^{\theta}_p} ^{p} \,ds\rk).
\EEQSZ
Wrapping up, we have 
\begin{align*}
&\mathbb{E}\Big(\sup_{0\le s\le {t\wedge \tau}}|\uk (s)|^p_{L^p}\Big) + c_1 p\mathbb{E}\int_0^{t\wedge \tau}\int_\CO \tesfalem{|}\uk (s, x)\tesfalem{|}^{p-1}\uk(s,x) \vk^q (s, x)\,dx\,ds\\
&\qquad +  { }\frac{r_1 p(p-1)}{C} \mathbb{E}\Big(\int_0^{t\wedge \tau} |\uk(s)|_{H^\theta_p}^{p}\,ds\Big)\\
&\leq C_1 \Big\{  \mathbb{E}|u_0|_{L^p}^{p} + 2 C(\eps) \EE\Big(\int_0^{t\wedge \tau} |\uk(s)|^p_{L^p} ds\Big) 
+ 2\eps \EE\Big(\sup_{0\le s\le {t\wedge \tau}} |\uk (s)|^p_{L^p}\Big)\\
&\qquad + 2 \eps \mathbb{E}\Big(\int_0^{t\wedge \tau} |\uk(s)|_{H^\theta_p}^{p}\,ds\Big) + b_1^p (t\wedge \tau)\Big\}.
\end{align*}
For $2 \eps < \min\{\frac{r_1 p(p-1)}{C}, 1\}$ we cancel the $H^\theta_p$-term on the left-hand side.
Furthermore, we send the $L^p$-term to the left, renormalise the inequality, and obtain a new constant $C_2, C_3>0$ 
\begin{align}
&\mathbb{E}\Big(\sup_{0\le s\le {t\wedge \tau}}|\uk (s)|^p_{L^p}\Big) + C_2\mathbb{E}\int_0^{t\wedge \tau}\int_\CO \tesfalem{|}\uk (s, x)\tesfalem{|}^{p-1} \uk(s, x) \vk^q (s, x)\,dx\,ds\nonumber\\
&\leq C_3 \Big\{  \mathbb{E}|u_0|_{L^p}^{p} + \int_0^{t} \EE\Big(\sup_{\sigma\in [0, s\wedge \tau]} |\uk(\sigma)|^p_{L^p}\Big) ds  + b_1^p T\Big\}.\label{e:preGron}
\end{align}
First, ignoring the $C_2$-term on the left, we obtain with the help of Grownwall's inequality 
an upper bound 
\[
\mathbb{E}\Big(\sup_{0\le s\le {t\wedge \tau}}|\uk (s)|^p_{L^p}\Big) \leq C_4 \Big(\mathbb{E}|u_0|_{L^p}^{p}+1\Big),
\]
which is subsequently inserted under the integral in \eqref{e:preGron}, which yields a constant $C_5 = C_5(T)$ such that for all $n\in \NN$ we have 
\DEQS			
\mathbb{E}\sup\limits_{0\leq t\leq T\wedge \tau_{n}}|\uk (t)|_{L^p}^{p}\,dx 
+ c_1 p\mathbb{E}\lk[ \int_0^{T\wedge \tau_n} \int_\CO \tesfalem{|}\uk (s, x)\tesfalem{|}^{p-1} \uk(s, x) \vk^q (s, x)\,dx\,ds\rk] 
\leq   C\,(\mathbb{E}|u_0 |^{p}_{L^{p}}+1).
\EEQS
By construction $\tau_n$ a.s.\ monotonically in $n$, the integrands follow suit, and monotone convergence allows us to pass to the limit under the expectation on the left. 

\noindent In the case $\rr>1$, we apply the power $\rr$ on both sides of \eqref{e:naked}, 
we arrive at similar calculations at
\DEQSZ \label{tosimplify2}
\lqq{\mathbb{E}\sup_{0\le s\le t\wedge \tau}|\uk (s)|^{p\rr}_{L^p}+\mathbb{E}\lk\{ c_1 p\int_0^{t\wedge \tau}\int_\CO \tesfalem{|}\uk (s,x)\tesfalem{|}^{p-1} \uk(s, x) \vk^q (s, x)\,dx\,ds\rk\}^{\rr} }
\\
&& { }( p(p-1)r_1)^{\rr}\mathbb{E}\lk\{ \int_0^{t\wedge \tau}\int_\CO|\uk (s,x)|^{p-2}|\nabla \uk (s,x)|^2\,dx\,ds \rk\}^{\rr}\nonumber\\
&\le & C_{\rr}\Big\{  \mathbb{E}|u_0|_{L^p}^{p\rr}
+ C_{\rr} \mathbb{E}\sup_{0\le s\le t\wedge \tau} \lk| \sum_k\int_0^{t\wedge \tau} \int_\CO|\uk (s,x)|^{p} \lambda_k^{-\gamma}\varphi_k(x)\,dx\,d\mathbf{w}^1_k(s) \rk|^{\rr}\nonumber
\\
&& +C\mathbb{E}\lk\{ \int_0^{t\wedge \tau} |\uk (s)|_{L^p}^{p}\,ds\rk\}^{\rr} + Tb_1^p.\nonumber
\EEQSZ
Applying  \eqref{hierendets},  Proposition \ref{runst1} and choose $\eps$ small enough, and renormalize we obtain 
\DEQS
\lqq{ \mathbb{E}\sup_{0\le s\le t\wedge \tau}|\uk (s)|^{p\rr}_{L^p}+C_r \mathbb{E}\lk\{\int_0^{t\wedge \tau}\int_\CO \uk ^{p-1}(s, x) \uk(s, x) \vk^q (s, x)\,dx\,ds\rk\}^{\rr} }
\\
&\le & C_{\rr}\Big\{  \mathbb{E}|u_0|_{L^p}^{p\rr} + b_1^p T+ C(\rr,T) \int_0^{t} \EE\sup_{0\le s_1\le s\wedge \tau} |\uk (s)|_{L^p}^{p\rr}\, ds\Big\} .\nonumber
\EEQS
Again, 	setting $\psi(t):= \mathbb{E}\sup_{0\le s\le t\wedge \tau}|\uk (s)|^{p\rr}_{L^p}$ we can apply the same Grownwall procedure to $\psi(t)$ to get the desired uniform bound on $\psi$. 
Finally sending $n\to \infty$ (in $\tau = \tau_{n, \kappa}$) the estimate above gives \eqref{e:boundeduk} for  $r>1$.

\newcounter{cond}
\stepcounter{cond}
\bigskip

\end{proof} 

\noindent The following Technical Lemma \ref{dasauch} is similar to Proposition \ref{tetaxi}, the key difference is that our goal here is to obtain a uniform bound that is independent of the cut-off.  This provides us with stronger conditions on the parameters. 

\begin{tlemma}\label{dasauch}
For all $1\le m<p<\frac{m}{m-1}$ with $p>q(p-1)$, 
$$
d\lk(\frac 12-\frac 1q\rk) -\frac{\FE}{2} \frac 2{q} \frac {p-m} {m(p-1)} \le \alpha
$$
and
for any $r^\ast\ge 1$,
there exists a generic constant ${C}(T)>0$
and numbers $\delta_1,\delta_2\in(0,1)$ with $\delta_1+\delta_2=1$, such that for any pair of nonnegative processes $\eta,\xi:[0,T]\times \Omega\times \CO \rightarrow[0,+\infty)$ we have
for $p' = \frac{p}{p-1}$
\DEQSZ \label{Eq:Productestimate} \hspace{-8ex}
\EE \| \eta\,\xi^q\|^ {r^\ast}_{L^{m}(0,T;L^1)}
&\le&
C(T)
\lk(\mathbb{E}\|\eta^p\xi^q\|_{L^1(0,T;L^1)}^{\frac{r^\ast}{p-q(p-1)}}\rk)^{\delta_1}\cdot
\lk(\mathbb{E} \|\xi \|_{\BH_{\alpha, \aleph}}^{r^\ast}\rk)^{\delta_2}, 
\EEQSZ
where $\delta_1 = \frac{p - q(p-1)}{p}$ and $\delta_2=\frac {q}{p'}$. In particular, 
\DEQSZ\label{prod:estimate}
\mathbb{E}\| \eta\xi^q\|_{L^\alneu (0,T;L^1)}^{r^\ast}\le \,
C(\ep)\,\mathbb{E}\|\eta ^{p} \xi^\ns \|_{L^1(0,T;L^1)}^{\frac{r^\ast}{p-q(p-1)}}
+\ep \mathbb{E} \|\xi\|_{\mathbb{H}_{\alpha, \aleph}}^{r^\ast} .
\EEQSZ
\end{tlemma}

\begin{proof}
\renewcommand{\pst}{p}
\renewcommand{\ns}{q}
\renewcommand{\LLm}{{r^\ast}}

Observe 
that we have for any $\beta\in[0,1]$
\begin{align*}
\eta (s) \xi ^q(s)
= \eta (s) \xi ^{q\beta } (s)\cdot \xi ^ {q(1-\beta)}(s),\quad s\in [0,T].
\end{align*}
Assume $m<p$ and fix $\beta=\frac 1p$. 
Let $\unbekannt \ge 1$ such that $\frac 1\unbekannt= m \beta = \frac{m}{p}$ and $\unbekannt'=(\frac{1}{m \beta})' = \frac {p}{p-m}$. Since we will need it later, we also note the following identity $\beta p = (1-\beta) p' = 1$.
Then, H\"older's inequality for the conjugate
exponents $\frac{1}{p}+ \frac{1}{p'}\le 1$ yields
\DEQS\notag 
\lk\| \eta  \xi ^ {q}\rk\|_{L^{m}(0,T;L^1)}^\LLm
&= & \lk(\int_0^T \lk( \int_\oO \lk| \eta(s, x) \xi^q(s, x)\rk|\, dx \rk)^{m} ds\rk)^{\frac{r^*}{m}} \\
& \le & \lk( \int_0^ T \lk( \int_\CO  \lk| \eta ^ {p} (s,x) \xi ^ {q\beta p }(s,x)\rk|\, dx\rk)^\frac{m}{p} \lk( \int_\CO  \lk|
\xi^{q(1-\beta )p'}(s,x)\rk|\, dx\rk)^\frac{m}{p'} \, ds\rk)^\frac{\LLm}{ m}.
\EEQS
Since $(1-\beta) p' = 1$, 
we obtain 
\DEQS
\lk\| \eta  \xi ^ {q}\rk\|_{L^{m}(0,T;L^1)}^\LLm
& \le \lk( \int_0^ T \lk( \int_\CO  \lk| \eta ^ {p} (s,x) \xi ^ {q}(s,x)\rk|\, dx\rk)^\frac{m}{p} \lk( \int_\CO  \lk|
\xi^{q}(s,x)\rk|\, dx\rk)^\frac{m}{p'} \, ds\rk)^\frac{\LLm}{ m}\\
& = \lk( \int_0^ T  \lk| \eta ^ {p}(s) \xi^q(s)\rk|_{L^1}^\frac{m}{p} \cdot
\lk|
\xi(s)\rk|_{L^q}^\frac{mq}{p'} \, ds\rk)^\frac{\LLm}{ m}.
\EEQS
An application of  H\"older's  inequality and bearing in mind that $\frac{m}{p} = \frac{1}{\unbekannt}$ then yields
\DEQSZ\label{start001}
\lk\| \eta  \xi ^ {q}\rk\|_{L^{m}(0,T;L^1)}^\LLm
& \le  & \lk(\lk( \int_0^ T \lk| \eta ^ {p}(s) \xi^q(s)\rk|_{L^1}ds \rk)^\frac{m}{p} \cdot
\lk(\int_0^T \lk|
\xi(s)\rk|_{L^q}^\frac{mq \mu'}{p'} \, ds\rk)^\frac{1}{\mu'}\rk)^\frac{\LLm}{ m}\notag 
\\
\notag 
& = &
\lk( \int_0^ T \lk| \eta ^ {\pst   }(s) \xi ^ { {\ns }}(s)\rk|_{L^1}
\, ds \rk)^\frac{r^*}{p}
\cdot\lk(
\int_0^T \lk|
\xi(s)\rk|_{L^q}^\frac{mq \mu'}{p'} \, ds\rk)^\frac{\LLm}{ m \mu'}
\\
&=
& \,\lk\| \eta ^ {\pst   } \xi ^ { {\ns }}\rk\|_{L^1(0,T;L^1)}^\frac {\LLm}{p}\cdot
\lk\|\xi\rk\|_{L^ {m q\frac {p-1}{p-m}} (0,T;L^ {q})}^{\frac{\LLm q}{p'}}.
\EEQSZ
For the last line, we used the following identity
$$
\frac{r^*}{m \mu'} \frac{mq \mu'}{p'} = \frac{r^*q}{p'} . 
$$
Here, it is essential that 
$m< p$ and $q < p'$. Due to the fact that the following inequality holds:
$$
\frac d2 -\alpha <\FEplus \frac {2({p-m})} {qm (p-1)}+\frac dq , \checkmark
$$
we can apply Proposition~\ref{interp_rho} which gives   the embedding 
$$L^{m q\frac {p-1}{p-m}} (0,T;L^{q}(\mathcal{O})) \hookrightarrow 
\mathbb{H}_{\alpha, \aleph} = L^\infty(0,t; H^\rho_2(\CO))\cap L^2(0,t;H^{\rho +\FEplus}_2(\CO)) 
.
$$
Going back to \eqref{start001}, using the embedding above, we get
\DEQS
\lk\| \eta  \xi ^ {q}\rk\|_{L^{m}(0,T;L^1)}^\LLm
& \le&   C \,\lk\| \eta ^ {\pst   } \xi ^ { {\ns }}\rk\|_{L^1(0,T;L^1)}^{\frac {\LLm}{p} }
\lk\|\xi\rk\|_{\BH_{\alpha, \aleph}}^{\LLm\frac{q}{p'}}.
\EEQS
Taking the expectation and applying H\"older's inequality, we have
\begin{align*}
\EE\lk\| \eta  \xi ^ {q}\rk\|_{L^{m}(0,T;L^1)}^\LLm
\leq C \bigg(\EE\Big(\|\eta^p\xi^q\|_{L^1(0,T;L^1)}^{\frac{r^*}{p-q(p-1)}}\Big)\bigg)^\frac{p-q(p-1)}{p}
\cdot \bigg(\EE\Big(\|\xi\|_{\BH_{\alpha, \aleph}}^{r^\ast}\Big)\bigg)^{\frac{q}{p'}}
\end{align*}
we get estimate \ref{Eq:Productestimate}. Finally, using Young's  inequality with $\frac{1}{\delta_1}$ and its convex conjugate, we obtain the following condition, which satisfies:
\begin{align*}
\Big(\frac{1}{\delta_1}\Big)' 
&=  \Big(\frac{p}{p-q(p-1)}\Big)' 
= \frac{\frac{p}{p-q(p-1)}}{\frac{p}{p-q(p-1)}-1}\\
&= \frac{p (p-q(p-1))}{(p-q(p-1))(p-p-q(p-1))}
= \frac{p }{q(p-1)} = \frac{1}{\delta_2}.
\end{align*}
This finishes the proof of \eqref{prod:estimate}.

\end{proof}

\noindent In the next proposition, we establish integrability and regularity of $\vk $, where $(\uk,\vk)$ solves the system \eqref{nonlinearcutoffu}-\eqref{nonlinearcutoffv}.

\begin{proposition}\label{propvarational}
	Let us assume that the tuple $(q,\FE,\alpha)$ satisfy
	$$
	q<\frac {2d}{2d-\FE}, \quad \alpha < \frac \aleph 2  -\frac{d}{2}. 
	$$
	Next, in case $2\le q< {\FE+1}$ and $d=1$, let us put $\tau:=\frac 12-\frac 1q-\alpha$ and $p_1^\ast =\frac \FE{q\tau}$.
In case, $1\le q<2$ let us put $p_1^\ast= 1/(2-q)$.
Then, for any $r^*\geq 1$ there exists positive constant $C>0$ such that for all $\kappa$ 
\DEQS
\mathbb{E}\| \vk \|_{\mathbb{H}_{\alpha, \aleph}}^{r^\ast}
&\leq C\bigg(1 +  \mathbb{E}|\Delta^ {\frac{\alpha}{2} }v_0|_{L^2}^{2r^\ast}
+ \Big(\EE|u_0|_{L^p}^{p_1^\ast {r^\ast}}\Big)^2\bigg)^\frac{1}{2},
\EEQS
where $\Delta^{\frac{\alpha}{2}}$ is a shorthand for the fractional Laplacian $(-\Delta^{-\frac{\alpha}{2}})$.
In case $q=2$, $d=2$, and $\aleph=2$, we get
\begin{equation*}
\begin{split}
&\frac 12\EE \sup_{0\le s\le T} \lvert \vk(s)\rvert^{\pp }_{L^{\pp }}  
\le \left(1+KTe^{KT}\right)  \left( \frac12 \EE  \lvert \uk(0)+\vk(0)\rvert^{\pp }_{L^{\pp }} + 	\frac 12 (r_1+r_2)^2 \frac {c _2^2}{r_2} \EE |u_0|_{L^2}^2+C\right)
.
\end{split}
\end{equation*}

\end{proposition}

\begin{proof}[Proof of Proposition \ref{propvarational}:]
Applying the It\^o formula to
$\phi(x)=| \Delta^ {\frac{\alpha}{2} } x|^2_{L^2}$
we get
\DEQSZ\label{amanfang}
\lqq{ \qquad \frac 12 \la \Delta^ {\frac{\alpha}{2} } \vk (t) , \Delta^ {\frac{\alpha}{2} } \vk (t)\ra-\frac 12 \la \Delta^ {\frac{\alpha}{2} } v_0 , \Delta^ {\frac{\alpha}{2} } v_0\ra
+r_2\int_0^t |\Delta^ {\frac{\alpha}{2} +\frac{\aleph}{2} } \vk (s)|^2_{L^2}\, ds}
&&
\\\notag 
&=&c_2\int_0^ t\la \Delta^ {\frac{\alpha}{2} } \vk (s), \Delta^ {\frac{\alpha}{2} }(\uk(s)\vk^{q}(s))\ra \, ds +a_2\int_0^t |\Delta^ {\frac{\alpha}{2} }\vk (s)|_{L^2}^2\, ds
+b_2\int_0^ t |\Delta^ {\frac{\alpha}{2} }\vk (s)|_{L^1}
\\
&&{}+\sigma_2  \int_0^ t \la\Delta^ {\frac{\alpha}{2} } \vk (s), \Delta^ {\frac{\alpha}{2} } \vk (s)dW_2(s)\ra
+ \frac{\sigma_2 ^2}{2} \int_0^ t \mbox{Tr}^2[ D^2\phi( \vk (s))]\, ds.\notag 
\EEQSZ
We treat the first term on the right-hand side. 
Duality gives for any {$\Gamma\in[0,1)$} 
\del{\DEQSZ\label{verweis}
\big| \la \Delta^{\frac{\alpha}{2} } \vk (s), \Delta^{\frac{\alpha}{2} }(\uk(s)\vk^{q}(s))\ra \big|
&\le& | \Delta^ {\frac{\alpha}{2} +\Gamma \aleph } \vk (s)|_{L^2} | \Delta^ {\frac{\alpha}{2} -\Gamma \aleph }(\bar u_\kappa (s) \bar v_{\kappa}^q(s))|_{L^2}
\EEQSZ}
\DEQSZ\label{verweis}
\big| \la \Delta^{\frac{\alpha}{2} } \vk (s), \Delta^{\frac{\alpha}{2} }(\uk(s)\vk^{q}(s))\ra \big|
&\le& | \Delta^ {\frac{\alpha}{2} +\Gamma {\frac{\aleph}{2}} } \vk (s)|_{L^2} | \Delta^ {\frac{\alpha}{2} -\Gamma {\frac{\aleph}{2}} }(\bar u_\kappa (s) \bar v_{\kappa}^q(s))|_{L^2}
\EEQSZ
Integration and H\"older's inequality gives
\DEQS
\lqq{ \int_0^T \big| \la \Delta^{\frac{\alpha}{2} } \vk (s), \Delta^{\frac{\alpha}{2} }(\uk(s)\vk^{q}(s))\ra \big|\, ds}
\\
&\le&\int_0^T  | \Delta^ {\frac{\alpha}{2} +\Gamma \frac \aleph 2} \vk (s)|_{L^2} | \Delta^ {\frac{\alpha}{2} -\Gamma {\frac{\aleph}{2}} }(\bar u_\kappa (s) \bar v_{\kappa}^q(s))|_{L^2}\, ds
\\
\del{&\le&\int_0^T  | \Delta^ {\frac{\alpha}{2} +\Gamma   \tesfalem{\aleph }} \vk (s)|_{L^2} | \Delta^ {\frac{\alpha}{2} -\Gamma \aleph }(\bar u_\kappa (s) \bar v_{\kappa}^q(s))|_{L^2}\, ds}
\\
&\le&
\lk( \int_0^T  | \Delta^ {\frac{\alpha}{2} +\Gamma \frac \aleph 2} \vk (s)|^{m'}_{L^2} \,ds\rk)^\frac 1{m'}  \lk(\int_0^ T | \Delta^ {\frac{\alpha}{2} -\Gamma {\frac{\aleph}{2}} }(\bar u_\kappa (s) \bar v_{\kappa}^q(s))|_{L^2}^m\rk)^\frac 1m .
\del{&\le&
\lk( \int_0^T  | \Delta^ {\frac{\alpha}{2} +\Gamma  { \aleph }} \vk (s)|^{m'}_{L^2} \,ds\rk)^\frac 1{m'}  \lk(\int_0^ T | \Delta^ {\frac{\alpha}{2} -\Gamma \aleph }(\bar u_\kappa (s) \bar v_{\kappa}^q(s))|_{L^2}^m\rk)^\frac 1m .}
\EEQS
Young's inequality gives 
\DEQSZ\label{verweis2}\notag 
\lqq{ \int_0^T \big| \la \Delta^{\frac{\alpha}{2} } \vk (s), \Delta^{\frac{\alpha}{2} }(\uk(s)\vk^{q}(s))\ra \big|\, ds}
\\
&\le&
\ep \lk( \int_0^T  | \Delta^ {\frac{\alpha}{2} +\Gamma \frac \aleph 2} \vk (s)|^{m'}_{L^2} \,ds\rk)^\frac 2{m'} +C(\ep)  \lk(\int_0^ T | \Delta^ {\frac{\alpha}{2} -\Gamma {\frac{\aleph}{2}} }(\bar u_\kappa (s) \bar v_{\kappa}^q(s))|_{L^2}^m\rk)^\frac 2m .
\del{&\le&
\ep \lk( \int_0^T  | \Delta^ {\frac{\alpha}{2} +\Gamma   \aleph  } \vk (s)|^{m'}_{L^2} \,ds\rk)^\frac 2{m'} +C(\ep)  \lk(\int_0^ T | \Delta^ {\frac{\alpha}{2} -\Gamma \aleph }(\bar u_\kappa (s) \bar v_{\kappa}^q(s))|_{L^2}^m\rk)^\frac 2m .}
\EEQSZ
For the first term on the right-hand in \eqref{verweis2}, it follows by interpolation that 
for
$\Gamma m'=2$
 there exists a constant $C>0$ such that
$$
\| \Delta^ {\frac{\alpha}{2}+\Gamma{\frac{ \aleph}{2}} }\vk \|_{L^{m'}(0,T;L^ 2)}^2 \le C \|\vk \|_{\BH_{\alpha, \aleph} }^2.
$$
For the second term on the right-hand in \eqref{verweis2}, it follows by the Sobolev embedding
that  for any
$-(\alpha -\Gamma \aleph) + \frac{d}{2} > d$ 
we have  $L^1 (\CO )\hookrightarrow H^{\alpha-\Gamma \aleph}_2(\CO )$.
In particular, 
there exists a constant $C>0$ such that
\DEQSZ \label{e: SobL1est}
| \Delta^ {\frac{\alpha}{2}- \Gamma {\frac{\aleph}{2}} }(\uk (s) \vk ^q(s))|_{L^ 2} &\le& C |\uk (s) \vk^q(s)|_{L^1 }.
\EEQSZ 
%
%
Secondly, first applying H\"older's inequality and then Young's inequality, we obtain
\DEQS
\lqq{ \int_0^t | \Delta^{\frac{\alpha}{2}+\Gamma{\frac{\aleph}{2}} } \vk (s)|_{L^ 2 } | \Delta^ {\frac{\alpha}{2} -\Gamma{\frac{\aleph}{2}} }(\uk (s) \vk^q(s))|_{L^ 2}\, ds}
&&
\\
&\le & C \cdot \Big(  \int_0^t | \Delta^ {{\frac{\alpha}{2}+\Gamma {\frac{\aleph}{2}}} } \vk (s)|_{L^ 2 }^{m'}\,ds \Big)^\frac{1}{m'} \cdot \lk(\int_0^ t |\uk (s) \vk^q(s)|_{L^ 1}^ {m} \, ds\rk)^\frac 1 {m}
\\[5mm]
&\le & C \cdot \| \vk \|_{\BH_{\alpha, \aleph} }  \cdot \|\uk  (s) \vk^q(s)\|_{L^{ {m} }(0,T;{L^1 })}
\\[5mm]
&\le &
\ep_1 \| \vk \|_{\BH_{\alpha, \aleph} } ^{2}  + C(\ep_1)  \|\uk  (s) \vk^q(s)\|_{L^{ {m} }(0,T;{L^1 })}^2.
\EEQS
To handle the first term we have to verify that the parameters $m'$ and $\Gamma $ satisfies $\Gamma  m'=2$.
To have the embedding $L^1(\CO)\hookrightarrow H^{\alpha-\Gamma  \frac \FE2}_2(\CO)$ we need $-(\alpha -\Gamma \frac \FE2)+\frac d2\ge d$, which gives 
\begin{align}\label{e: desigaldades}
	(i)~\alpha < \frac \aleph 2\Gamma  -\frac{d}{2}\quad
	(ii)~m'\Gamma =2. 
\end{align}
To apply the Technical Proposition \ref{dasauch} to the second term we have to verify that the parameters $m$ and $\Gamma$ with  $1<m<\infty$, $\alpha$, and $\Gamma  $  satisfy the following system:
\begin{align}\label{e: desigaldades1}
\quad (iii)~m<p<\frac {m}{m-1}
\quad (iv)~d\lk(\frac 12 -\frac 1q\rk) -\FEplus \frac 2q\frac {p-m}{m(p-1)}\le \alpha.
\end{align}

\noindent Note that since $m'$ and $p'$ are the conjugates to $m$ and $p$ and vice versa, we can write 
\begin{align}\label{rechnung}
\frac {p-m}{m(p-1)} 
=\frac{\frac{p' (m'-1)- m'(p'-1)}{(p'-1)(m'-1)}}{\frac{m'}{m'-1} \frac{1}{p'-1}}
=\frac{p' (m'-1)- m'(p'-1)}{m'} = \frac{m'-p'}{m'} = 1- \frac{p'}{m'}.
\end{align}
Now, we have to fix the constants $p,m,m',\Gamma $.
Observe, the parameter $\alpha$ has to be larger or equal to $\rho$. In case $\alpha$ is smaler than $\rho$, we cannot control the stopping time 
$\tau_\n:=\min\{ \| \vk \|_{\mathbb{H}_{\rho,\aleph}}\ge n\}$, and, conseqeuntly, we would not be able to globalise the solution. That means, we chose a parameter $\Gamma $ close to one, or, which is equivalent, $m'>2$ close to two. 
	
\textbf{Case I: $d=1,2$, $1<q<2$:}
	So, let us start with $m'=2+\ep$, which gives $m=\frac {2+\ep}{1+\ep}=2-\frac \ep{1+\ep}$ and $p=p'=2$.
Then, $1<m<p<\frac m{m-1}$, $q<p'<m'$, and  $1-\frac {p'}{m'}=\frac \ep{1+\ep}$. That means, if $d(\frac 12-\frac 1q)<\alpha$, then there exists some $\ep>0$ such that (iv) is true.
Since $\ep>0$ can be chosen arbitrary small, (i) and (ii) are satisfied.
In this case we have $1/(p-q(p-1))=1/(2-q)$. So, we can put $p^\ast_1:=1/(2-q)$.

\textbf{Case II: $d=1$, $2\le q\le \frac {\FE+d}d$:}
Here, we need to be satisfied $1<m<p<\frac q{q+1}$, i.e., $q<p'<m'$. Let us put $p'=q+\ep$ and $m'=q+2\ep$, where $\ep>0$. Later on, we will see that $\ep=\frac {(q-1)q\tau}{\FE-q\tau}$ is the best choice, where  $\tau$ is given by
$\tau=d(\frac 12 -\frac 1q)<\alpha$.
Then, $1-\frac {p'}{m'}=\frac \ep{q+2\ep}$.  Hence, the relation between $\tau$ and $\ep$ is given by 
$$
\frac \FE q (\frac \ep{q+2\ep})=\tau,
$$
which gives $\ep=\frac {(q-1)q\tau}{\FE-q\tau}$. Finally, since $p'=q+\ep$ and $p=\frac {q+\ep}{q-1+\ep}$, we get 
$$
p-q(p-1)=\frac \ep{q+\ep-1}=\frac {q\tau}\FE.
$$
Hence $p_1^\ast=\frac \FE{q\tau}$.
 {
Note that $(iii)$ implies that $m\in (1, 2)$ and hence $m'>2$, such that the embedding $L^1(\CO)\hookrightarrow H^{\alpha-\frac \FE 2}_2(\CO)$ is valid. 
Also, that, if 
$$
(iv)'\quad d\lk(\frac 12-\frac 1q\rk) \le \alpha, 
$$
then there exists some $\ep>0$ such that (iv) is satisfied.  Moreover, (i) and (iv) gives
$$
d\, \frac {q-1}q<\frac \aleph 2 \frac 2{2+\ep}.
$$
In particular, if 
$
d<\frac \aleph 2\frac q{q-1}
$, then some $\ep>0$ exists, such that the inequality above is satisfied. 
This gives as condition for $q$}
$$
q<\frac {2d}{2d-\FE}.
$$
Hence \eqref{prod:estimate} in Technical Lemma \ref{dasauch} yields for $\eta = \uk$ and $\xi = \vk$ 
that for any $\ep>0$ there is a constant $C(\ep)>0$ and $l>1$ such that
\DEQSZ
\mathbb{E} \|\uk\vk^q\|_{L^{m}(0,T;L^1)}^{r^\ast }
& \leq &C(\ep)\EE\|\uk^p \vk^q\|_{L^1(0,T;L^1)}^{\frac{r^\ast}{p-q(p-1)}}
+ \ep  \EE\|\vk\|_{\BH_{\alpha, \aleph}}^{r^\ast}. 	\label{e:pivot1}
\EEQSZ
By Proposition \eqref{reg_u2}, we get
\begin{equation}\label{e:herzstueck}
\EE\Big(\int_0^T |\bar u_\kappa^p(s) \bar v_\kappa^q(s)|_{L^1} ds\Big)^{\frac{r^\ast}{p-q(p-1)}}
\le C(\ep) \EE|u_0|_{L^p}^{{\frac{r^\ast}{p-q(p-1)}}} +\ep  \EE\|\vk\|_{\BH_{\alpha, \aleph}}^{r^\ast}.   .
\end{equation}
Hence, the first term in \eqref{e:pivot1} is bounded independently of $\kappa$. 
We will see later on that substituting the estimate \eqref{e:herzstueck} in \eqref{amanfang}, we will see that the term 
$\ep  \EE\|\vk\|_{\BH_{\alpha, \aleph}}^{r^\ast}$ can be cancelled.
We continue with a $\kappa$-independent bound on $\EE\|\vk\|_{\BH_{\alpha, \aleph}}^{ r^\ast}$. The stochastic integral can be tackled by the
Burkholder-Davis-Gundy inequality and the Young inequality.
We use \eqref{Hrhoburkholderqqq}. In doing so, there exists for any $\ep>0$ 
a constant $C=C(\ep)>0$ such that we have
\DEQSZ\label{fortrace}
&& \EE \lk\{ \sup_{s\in [0, t]} \int_0^s \la\Delta^ {\frac{\alpha}{2} } \vk (\sigma), \Delta^ {\frac{\alpha}{2} } \vk (\sigma)dW_2(\sigma)\ra\rk\}^{r^\ast}\nonumber\\
&&\le C \EE \lk\{ \int_0^ t |\Delta^ {\frac{\alpha}{2} } \vk (s)|_{L^2}^4 ds\rk\}^\frac {r^\ast}2
\nonumber\\
&&\le C \EE \lk\{\sup_{s\in [0, t]} |\Delta^ {\frac{\alpha}{2} } \vk (s)|_{L^2}^2  \int_0^ t |\Delta^ {\frac{\alpha}{2} } \vk (s)|_{L^2}^2 ds\rk\}^\frac {r^\ast}2
\nonumber\\
&&= C \EE \sup_{s\in [0, t]} |\Delta^ {\frac{\alpha}{2} } \vk (s)|_{L^2}^{r^\ast} \lk\{\int_0^ t |\Delta^ {\frac{\alpha}{2} } \vk (s)|_{L^2}^2 ds\rk\}^\frac {r^\ast}2
\nonumber\\
&&\le C \Big(\EE \sup_{s\in [0, t]} |\Delta^ {\frac{\alpha}{2} } \vk (s)|_{L^2}^{2r^\ast}\Big)^\frac{1}{2} \Big(\EE\lk\{\int_0^ t |\Delta^ {\frac{\alpha}{2} } \vk (s)|_{L^2}^2 ds\rk\}^{r^\ast}\Big)^\frac{1}{2}
\nonumber\\
&&   \lqq{ {\le \ep	 \EE \sup_{0\le s\le t} |\Delta^ {\frac{\alpha}{2} } \vk (s)|_{L^2}^{2r^\ast}
+ C(\eps) 
\,  \EE \lk\{ \int_0^ t |\Delta^ {\frac{\alpha}{2} } \vk (s)|_{L^2}^2 ds\rk\}^ {r^\ast}.}
}\nonumber
\EEQSZ
The third term, i.e. $  \int_0^ t \mbox{Tr}^2[ D^2\phi( \vk (s))]\, ds$ can be handled by \eqref{itotrace} and the Young inequality. Recall the It\^o formula from the beginning of the proof. 
Taking the supremum in time $s\in [0, t]$ we get 
\DEQS
\lqq{ \frac 12 \sup_{s\in [0, t]} |\Delta^ {\frac{\alpha}{2} } \vk (t)|^2 -\frac 12 \la \Delta^ {\frac{\alpha}{2} } v_0 , \Delta^ {\frac{\alpha}{2} } v_0\ra
+r_2\int_0^t |\Delta^ {\frac{\alpha}{2} +\frac{\aleph}{2} } \vk (s)|^2_{L^2}\, ds}
&&
\\
&{=}&c_2\int_0^ t\la \Delta^ {\frac{\alpha}{2} } \vk (s), \Delta^ {\frac{\alpha}{2} }(\uk(s)\vk^{q}(s))\ra \, ds +a_2\int_0^t |\Delta^ {\frac{\alpha}{2} }\vk (s)|_{L^2}^2\, ds
+b_2\int_0^ t |\Delta^ {\frac{\alpha}{2} }\vk (s)|_{L^1}
\\
&&{}+\sigma_2  \sup_{s\in [0, t]}\int_0^s \la\Delta^ {\frac{\alpha}{2} } \vk (s), \Delta^ {\frac{\alpha}{2} } \vk (s)dW_2(s)\ra
+ \frac{\sigma_2 ^2}{2} \int_0^ t \mbox{Tr}^2[ D^2\phi( \vk (s))]\, ds.
\EEQS
Collecting altogether, we have
\DEQS
\lqq{\frac 1{2}\mathbb{E}\sup_{0\le s\le t}  | \Delta^ {\frac{\alpha}{2} } \vk (t)|_{L^2}^{2{r^\ast}}-\frac 12\mathbb{E}| \Delta^ {\frac{\alpha}{2} }v_0|_{L^2}^{{2r^\ast}}
+\frac {r_2^{r^\ast}}2\mathbb{E}\lk( \int_0^ t | \Delta^ {\frac{\alpha}{2} +\frac \aleph2} \vk (s) |_{L^2}^{2}\, ds\rk)^{r^\ast}}
\\
& {\le} & 
C \mathbb{E}\Big\{\frac 12 \sup_{s\in [0, t]}| \Delta^ {\frac{\alpha}{2} } \vk (s)|_{L^2}^{2} -\frac {1}2 | \Delta^ {\frac{\alpha}{2} }v_0|_{L^2}^{2}
+\frac {r_2}2\int_0^ t | \Delta^ {\frac{\alpha}{2} +\frac \aleph2} \vk (s) |_{L^2}^{2}\, ds \Big\}^{r^\ast} 
\\
&\le& c_2^{r^\ast} \ep  \Big(\EE\|\vk\|_{\BH_{\alpha, \aleph}}^{ r^\ast}\Big) 
+ C(\EE|u_0|_{L^p}^{\frac{r^\ast}{p-q(p-1)}} +1)\\
&&+a_2^{r^\ast}\EE\int_0^t |\Delta^ {\frac{\alpha}{2} }\vk (s)|_{L^2}^{2r^\ast}\, ds
+b_2^{r^\ast}\EE\int_0^ t |\Delta^ {\frac{\alpha}{2} }\vk (s)|_{L^2}^{r^\ast}ds
\\
&&{}+
\ep	 \EE \sup_{0\le s\le t} |\Delta^ {\frac{\alpha}{2} } \vk (s)|_{L^2}^{2r^\ast}
+ C(\eps) 
\,  \EE \lk\{ \int_0^ t |\Delta^ {\frac{\alpha}{2} } \vk (s)|_{L^2}^{2r^\ast} ds\rk\}
\\
&\le& {c_2^{r^\ast} \ep  \Big(\EE\|\vk\|_{\BH_{\alpha, \aleph}}^{ r^\ast}\Big) 
+ C(\EE|u_0|_{L^p}^{\frac{r^\ast}{p-q(p-1)}} +1)}\\
&&{{}+a_2^{r^\ast}\int_0^t \EE\sup_{\sigma \in [0, s]} |\Delta^ {\frac{\alpha}{2} }\vk (\sigma)|_{L^2}^{2r^\ast}\, ds
+\frac 12 tb_2^{2r^\ast}
+ \frac 12 \int_0^ t \EE |\Delta^ {\frac{\alpha}{2} }\vk (s)|_{L^2}^{2r^\ast}ds}
\\
&&{}+
{ \ep	 \EE \sup_{0\le s\le t} |\Delta^ {\frac{\alpha}{2} } \vk (s)|_{L^2}^{2r^\ast}
+ C(\eps) 
\,   \int_0^ t \EE \sup_{\sigma\in [0, s]} |\Delta^ {\frac{\alpha}{2} } \vk (\sigma)|_{L^2}^{2r^\ast} ds.}
\\
\EEQS
A rearrangement of terms gives for an appropriate choice of $\ep>0$ some $\kappa$-independent constants $C_1, \dots, C_6>0$ such that  we have 
\DEQSZ\label{heretolook0}
&&\mathbb{E}\sup_{0\le s\le t}  | \Delta^ {\frac{\alpha}{2} } \vk (t)|_{L^2}^{2r^\ast}\nonumber\\
&&\le\mathbb{E}\sup_{0\le s\le t}  | \Delta^ {\frac{\alpha}{2} } \vk (t)|_{L^2}^{2r^\ast}
+C_1\mathbb{E}\lk( \int_0^ t | \Delta^ {\frac{\alpha}{2} +\frac \aleph2} \vk (s) |_{L^2}^{2}\, ds\rk)^{r^\ast}
\nonumber\\
&&\le
C_3\mathbb{E}|\Delta^ {\frac{\alpha}{2} }v_0|_{L^2}^{r^\ast}
+ \ep C_4 \Big(\EE\|\vk\|_{\BH_{\alpha, \aleph}}^{ r^\ast}\Big)
+ C_5(\EE|u_0|_{L^p}^{\frac{r^\ast}{p-q(p-1)}} +1)\nonumber\\
&&\qquad+ C_6\, \int_0^ t \EE  |\Delta^{\frac{\alpha}{2} } \vk (s)|_{L^2}^{2r^\ast} \, ds.\EEQSZ
Rearranging gives for an appropriate choice of $\ep_1$ and $\ep_4$ and for any arbitrary $\tilde \ep_2,\tilde \ep_3>0$
\DEQSZ\label{heretolook1}
&&\frac 18\mathbb{E}\sup_{0\le s\le t}  | \Delta^ {\frac{\alpha}{2} } \vk (t)|_{L^2}^{r^\ast}
+\frac {r_2^{r^\ast}}8\mathbb{E}\lk( \int_0^ t | \Delta^ {\frac{\alpha}{2} +\frac \aleph2} \vk (s) |_{L^2}^{2}\, ds\rk)^{\frac {r^\ast}2}
\nonumber\\
&&\le
\frac 12\mathbb{E}|\Delta^ {\frac{\alpha}{2} }v_0|_{L^2}^{r^\ast}
+C_1 \,\mathbb{E}|u_0|_{L^{p}}^{\frac{pr^\ast}{p-q(p-1)}}
+\tilde \epsilon_2  \EE  \|\bar v_\kappa\|_{\mathbb{H}_{\alpha, \aleph}}^{r^\ast }
+ C_0\,   \int_0^ t \EE \sup_{0\le s_1\le s} |\Delta^{\frac{\alpha}{2} } \vk (s_1)|_{L^2}^{r^\ast} \, ds.
\nonumber\EEQSZ
We neglect the second term on the left-hand side and apply the Grownwall Lemma to $\Phi(t):= \EE \sup_{0\le s_1\le t} |\Delta^ {\frac{\alpha}{2} } \vk (s_1)|_{L^2}^{r^\ast}$,
it follows that for any $\tilde \ep_1,\tilde \ep_2>0$ there exists some $C_0,C_1,C_2>0$ such that
we know
\DEQS
\EE \sup_{0\le s\le T} |\Delta^ {\frac{\alpha}{2} } \bar v_\kappa (s)|_{L^2}^{r^\ast}
&\le& e^{ C_0\, T}\,\Big\{  \frac 12\mathbb{E}|\Delta^{\frac{\alpha}{2}} v_0|_{L^2}^{r^\ast}
+C_1 \,\mathbb{E}|u_0|_{L^{p}}^{\frac{pr^\ast}{p-q(p-1)}}
%
+\tilde \epsilon_3  \EE  \|\bar v_\kappa \|_{\mathbb{H}_{\alpha, \aleph} }^{r^\ast }
\Big\}.
\EEQS
Substituting this estimate in \eqref{heretolook1} to estimate $ \int_0^ t \EE \sup_{0\le s_1\le s} |\Delta^ {\frac{\alpha}{2} } \vk (s_1)|_{L^2}^{2r^\ast} \, ds$, we
know, that for all $\bar \ep_1,\bar \ep_2,\bar \ep_3,\bar \ep_4>0$ there exist constant $\bar C_1,\bar C_2,\bar C_3>0$ such that
\DEQS
c\mathbb{E}\| \vk \|_{\mathbb{H}_{\alpha, \aleph}}^{r^\ast}
&&\le \EE \essup_{0\le s\le T}|\Delta^{\frac{\alpha}{2} }\vk (s)|_{L^2}^{r^\ast}+ r_2\mathbb{E}\lk\{\int_0^T|\Delta^{\frac{\alpha}{2}
+\frac \aleph2}\vk (s)|_{L^2}^2 ds\rk\}^{\frac {r^\ast}2}\\
&&\le  C_1 \mathbb{E}|\Delta^{\frac{\alpha}{2} }v_0|_{L^2}^2
+C_2 \,\mathbb{E}|u_0|_{L^{p}}^{\frac{pr^\ast}{p-q(p-1)}}
+ \bar \ep_2\EE  \|\vk \|_{\mathbb{H}_{\alpha, \aleph}  }^{r^\ast}
\\
&&{} + Te^{C_0T}\, \Big\{ C_4\,\mathbb{E}|u_0|_{L^{p}}^{\frac{pr^\ast}{p-q(p-1)}}
+\bar  \epsilon_4  \EE  \|\vk \|_{\mathbb{H}_{\alpha, \aleph}}^{r^\ast }
\Big\}.
\EEQS
Taking $\bar \ep_4>0$ sufficiently small, the assertion follows.

\newcommand{\gann}{c}
\newcommand{\mO}{\CO}
\renewcommand{\dd}{\,d}
\renewcommand{\pp}{2}

\noindent \textbf{Limit Case: $d=2$, $q=2$ and  $\aleph=2$.}
In this setting,  the following uniform bound can be shown. 	Let $w= \gann_2 \uk  + \gann_1\vk $. 
For any  number $t\in[0,T]$, there exist constants $C,K>0$ such that we have 
\DEQS		
\lqq{ \frac 12\EE\sup_{0\le s\le T} \lvert w(s)\rvert^{\pp }_{L^{\pp }}  
}
\\
&&{}	+\frac 12 r_1\gann ^2_2\EE \int_{0}^{t} \int_{\mO}(\nabla  \uk (r,x))^2\, \dd x \dd r  +\frac 14 r_2\gann ^2_1\EE \int_{0}^{s}\int_{\mO}(\nabla  \vk (r,x))^2\, \dd x\dd r
\\
&	\le &\left(1+KTe^{KT}\right) \left(\frac12 \EE  \lvert w(0)\rvert^{\pp }_{L^{\pp }} + 	\frac 12 (r_1+r_2)^2 \frac {\gann _2^2}{r_2} \EE |u_0|_{L^2}^2+C\right)
.
\EEQS
We fix a number $t\geq 0 $, $t\le \tau^\ast$. For simplicity we omit the stochastic terms. 
We see that $w$ satisfies the following equation
\begin{equation}
\label{eq:bound1}
\begin{split}
\dot{w}(t,x) =  &r_1\gann _2\Delta  \uk (t,x) + r_2\gann _1 \Delta \vk (t,x)+(a_1c_2 \uk+a_2c_1\vk)+(b_1c_2+b_2c_1)
.
\end{split}
\end{equation}
with $w(0,x)=\gann _2 u_0(x)+\gann _1v_0(x)$. It follows that
\begin{equation}
\label{eq:bound2}
\begin{split}
&\lvert w(s)\rvert^{\pp }_{L^{\pp }} = \lvert w(0)\rvert^{\pp }_{L^{\pp }} +\pp (b_1c_2+b_2c_1)\int_{0}^{s}\int_{\mO} w(r,x)\dd x\dd r
\\
& +\pp \int_{0}^{s}
\int_{\mO}\big(r_1\gann _2\Delta  \uk (r,x) + r_2\gann _1 \Delta \vk (r,x)\big)(w(r,x))\dd x\dd r
\\
&{}+2\int_{0}^{s}\int_{\mO}(\gann _2a_1\uk (r,x)+\gann _1a_2\vk(r,x)) (w(r,x))\dd x\dd r
.		\end{split}
\end{equation}
Let us investigate the second term on the right hand side. Integration by parts gives
\begin{align*}
& \int_{\mO}\big(r_1\gann _2\Delta  \uk (r,x) + r_2\gann _1 \Delta \vk (r,x)\big)(w(r,x))\dd x
\\
&=\int_{\mO}\big(r_1\gann _2\Delta  \uk (r,x) + r_2\gann _1 \Delta \vk (r,x)\big)((\gann _2  \uk (r,x) + \gann _1  \vk (r,x))\dd x
\\
&\le -r_1\gann ^2_2\int_{\mO}(\nabla  \uk (r,x))^2\, \dd x  - r_2\gann ^2_1\int_{\mO}(\nabla  \vk (r,x))^2\, \dd x
\\
&\qquad {}		+\gann _2\gann _1( r_2+r_1)  \lk| \int_{\mO}
\nabla  \vk (r,x)\nabla  \uk (r,x)\dd x\rk|.
\end{align*}
The Young inequality gives for any $\ep>0$
\begin{align*}
& \int_{\mO}\big(r_1\gann _2\Delta  \uk (r,x) + r_2\gann _1 \Delta \vk (r,x)\big)(w(r,x))\dd x
\\
&\le -r_1\gann ^2_2\int_{\mO}(\nabla  \uk (r,x))^2\, \dd x  - r_2\gann ^2_1\int_{\mO}(\nabla  \vk (r,x))^2\, \dd x
\\ &{}
+\ep \gann _2\gann _1( r_2+r_1)\int_{\mO}
(\nabla  \vk (r,x))^2  \dd x + \frac 1{4\ep}  \gann _2\gann _1( r_2+r_1) \int_{\mO}( \nabla  \uk (r,x))^2\dd x.
\end{align*}
Choosing $\ep=\frac 12 \frac {r_2\gann _1}{ \gann _2( r_2+r_1)}$, we have
$$
\frac 1 {4\ep} \gann _2\gann _1(r_1+r_2)\le   \frac 12 (r_1+r_2)^2 \frac {\gann _2^2}{r_2} .
$$
Substituting the result in \eqref{eq:bound2},
we get
\begin{align*}
&\frac 12 \lvert w(s)\rvert^{\pp }_{L^{\pp }}  
+r_1\gann ^2_2\int_{0}^{s} \int_{\mO}(\nabla  \uk (r,x))^2\, \dd x \dd r  +\frac 12  r_2\gann ^2_1\int_{0}^{s}\int_{\mO}(\nabla  \vk (r,x))^2\, \dd x\dd r
\\
&\phantom{\lvert w(s)\rvert^{\pp }_{L^{\pp }} }
\le \frac 12	\lvert w(0)\rvert^{\pp }_{L^{\pp }}+\int_{0}^{s}\int_{\mO}(\gann _2a_1\uk (r,x)+\gann _1a_2\vk(r,x)) (w(r,x))\dd x\dd r
\\			
&\phantom{\lvert w(s)\rvert^{\pp }_{L^{\pp }}  =	}{}
+	\frac 12 (r_1+r_2)^2 \frac {\gann _2^2}{r_2}  \int_{\mO}( \nabla  \uk (r,x))^2\dd x
+ (b_1c_2+b_2c_1)\int_{0}^{s}\int_{\mO} w(r,x)\dd x\dd r		.
\end{align*}
Again, applying the Young inequality we know, that there exists a constant $$C=C(r_1,a_1,a_2,b_1,b_2,c_1,c_2)>0$$ such that 
\begin{equation}
\label{eq:bound22}
\begin{split}
&\frac 12\lvert w(s)\rvert^{\pp }_{L^{\pp }}  
+r_1\gann ^2_2\int_{0}^{s} \int_{\mO}(\nabla  \uk (r,x))^2\, \dd x \dd r  +\frac 12  r_2\gann ^2_1\int_{0}^{s}\int_{\mO}(\nabla  \vk (r,x))^2\, \dd x\dd r
\\
&\le  \lvert w(0)\rvert^{\pp }_{L^{\pp }} +  C \int_{0}^{s}|w(r)|_{L^2}^2 \dd r
+	\frac 12 (r_1+r_2)^2 \frac {\gann _2^2}{r_2}  \int_{\mO}( \nabla  \uk (r,x))^2\dd x
.
\end{split}
\end{equation}
Taking expectation, 	it follows, that
\begin{equation*}
\begin{split}
&\frac 12\EE \sup_{0\le s\le t} \lvert w(s)\rvert^{\pp }_{L^{\pp }}  
\le \EE  \lvert w(0)\rvert^{\pp }_{L^{\pp }} +  C \EE \int_{0}^{t}|w(r)|_{L^2}^2 \dd r
\\
&\phantom{=\lvert w(0)\rvert^{\pp }_{L^{\pp }}}
+\EE 	\frac 12 (r_1+r_2)^2 \frac {\gann _2^2}{r_2} \int_{0}^{t} \int_{\mO}( \nabla  \uk (r,x))^2\dd x\dd r
\\
&\le\frac 12 \EE  \lvert w(0)\rvert^{\pp }_{L^{\pp }} +  C \EE \int_{0}^{t}|w(r)|_{L^2}^2 \dd r
+\tilde C	\lk[ \frac 12 (r_1+r_2)^2 \frac {\gann _2^2}{r_2} \EE |u_0|_{L^2}^2+1\rk]
.
\end{split}
\end{equation*}
The last line and taking into account the uniform bounds on $u$, i.e.\ Proposition 
\ref{reg_u2} gives the assertion.
The Gr\"onwall Lemma gives	
\begin{equation*}
\begin{split}
&\frac 12\EE \sup_{0\le s\le T} \lvert w(s)\rvert^{\pp }_{L^{\pp }}  
\le \left(1+KTe^{KT}\right)  \left( \frac12 \EE  \lvert w(0)\rvert^{\pp }_{L^{\pp }} + 	\frac 12 (r_1+r_2)^2 \frac {\gann _2^2}{r_2} \EE |u_0|_{L^2}^2+C\right)
.
\end{split}
\end{equation*}
Together with estimate \eqref {eq:bound22}, the last estimates implies the assertion.
\end{proof}

\bigskip
\section{\textbf{The stochastic convolution process}}\label{appb}

\noindent We refer to \cite{brzezniak,brzezniaGatarek, DaPrZa,janbesov,maxjan,vanneerven}, where the precise definition of the stochastic integral in Hilbert and Banach space can be found. 
For the reader's convenience, we will shortly introduce the main results we are using within the proofs. 	For simplicity, we refrain from stating the results in full generality. The properties gathered are used in Section~\ref{Appendix_A_noise}. 

\medskip

\noindent Let $E$ be a Hilbert space and $A$ be a linear map on $E$ satisfying the following conditions.
\begin{assumption}\label{assum-1}
\begin{enumerate}
\item \label{h2-a}
$-A$ is a positive operator\footnote{See Section I.14.1 in Triebel's monograph \cite{Triebel_1995}.} on $E$ with compact resolvent. In particular, there exists $M>0$ such that
\[ \Lve
(A+\lambda \tesfalem{I})^{-1}\Rve \le \frac{M}{1+\lambda}, \mbox{ for any }\lambda\ge 0;\]
\item \label{h2-b} $A$ is an infinitesimal generator of an analytic semigroup $(e^{tA})_{t\geq 0}$ of contraction type in $E$.
\item \label{h3} $A$ has the BIP (bounded imaginary power) property, i.e. there exists some
constants $K>0$ and $\tesfalem{\vartheta}\in [0,\frac\pi2)$ such that
\DEQSZ 
\Vert A^{is} \Vert &\le& K e^{\vartheta |s|}, \; s \in \mathbb{R}.
\label{2.1}
\EEQSZ 

\end{enumerate}
\end{assumption}
\del{ Let us mention that we have for all $x\in E$, 
$$
| Ae^{tA}x|\le C\frac 1t |x|_E,
$$
and therefore, for all $\alpha\ ge 0$,
\DEQSZ\label{smoothing}}

\noindent The result will be formulated in terms of abstract interpolation
and extrapolation spaces of $D(A)$. In applications, these spaces often correspond to the Sobolev spaces, as we will explain below. Interpolation and extrapolation spaces of $D(A)$ may be defined for any operator $A$ which generates an analytic semigroup of type
$(\omega,\theta,K)$ on a Banach space $X$. For $\alpha\in(0,1)$ we write
$$
D_A(\alpha,p)=(E,D(A))_{\alpha,p},\quad D(A^\alpha):=[E,D(A)]_\alpha,
$$
for the real and complex interpolation spaces.

\noindent For $x\in E$ let us define
$$
\mathcal{I}:[0,T]\ni t\mapsto e^{tA}x\in E.
$$
Then (see Proposition 3.1 \cite{maxjan})
if $
x\in D_A(1-\frac 1p,p)$, then
\DEQSZ\label{fracine01}
\max\tesfalem{\big(}\| \mathcal{I}\|_{H^{1,p}([0,T],E)}, \| \mathcal{I}\|_{L^p(o,T;D(A))}\tesfalem{\big)}\le C\|x\|_{D_A(1-\frac 1p,p)}.
\EEQSZ

\noindent Let us fix two real numbers $q\in (1,\infty)$ and $\mu\geq 0$.
Let us now define the convolution operator \tesfalem{for}  $w:[0,T]\to E$ by
$$
{\mathfrak{F}}_A(w)(t):= \int_0^ t e^{(t-s) A} w (s)\, ds,
$$
and for $\alpha\in(0,1)$ the fractional convolution
$$
{\mathfrak{F}}_A^\alpha (w)(t):= \int_0^ t (t-s)^{-\alpha} e^{(t-s) A} w (s)\, ds,
$$
Then, one can find the following Lemma.
\begin{lemma}\label{L:reg} (Compare Theorem 3.3 \cite{maxjan}, or Lemma  2.4 of \cite{brzezniaGatarek}) Let Assumption \ref{assum-1} be satisfied for $A$.
Suppose that the positive numbers $\alpha, \beta, \gamma, \delta \in (0, 1]$ and $q>1$
satisfy
\DEQSZ 
1  -\frac1q  +  \gamma &>&\beta +\delta.
\label{cond:1}
\EEQSZ 
If $T  \in (0,\infty)$, then  the operator
\DEQSZ 
{\mathfrak{F}}_A
: L^q (0,T;D(A^\gamma)) &\to&
{{C}}^\beta_b ([0,T];D(A^\delta)) \; \label{2.13b}
\EEQSZ 
is bounded.
\end{lemma}

\noindent Finally, we present some slight modifications of Proposition 2.2 and Theorem 2.6 from~\cite{brzezniak}, and respectively, from Lemma 3.6 of \cite{Jan1}.
\begin{theorem}
\label{Th:compact}
Let Assumption \ref{assum-1} be satisfied.
Assume that
$ \alpha \in (0,1]$ and $\delta,\gamma\geq  0$ are such that
\DEQSZ 
\alpha+\gamma-\delta-\frac1q  &>&-\frac1r.
\label{cond:2}
\EEQSZ 
Then the   operator
\DEQSZ \label{cond:3} 
A ^{\delta}\, \mathfrak{F}_A^\alpha \, A ^ {-\gamma} : L^q (0,T;E)
&\to&  L^r(0,T;E)
\EEQSZ 
is bounded. Moreover, if the operator $ \mathfrak{F}_A:E\to E$ is   compact, then the
operator in \eqref{cond:3}   is compact.
\end{theorem}

\medskip 
\begin{remark}\label{Prop:2.1-remark-F} 
In view of Theorem \ref{Th:compact} 
$${\mathfrak{F}}_A:L ^ p(0,T;E)\to L ^ p(0,T;E)
$$
is a well defined bounded linear operator for $p\in [1,\infty)$.
\end{remark}

\medskip 

\noindent For the study of maximal regularity, we know that if $A$ is the Laplace operator with Neumann boundary conditions and $E$ a UMD Banach space, $p\in (1,\infty)$,  the mapping
\DEQSZ\label{maxjan0}
{\mathfrak{F}}_A:L^p(0,T;E)\longrightarrow L^p(0,T;D(A))
\EEQSZ
and (see, e.g., estimate (3.6) \cite{maxjan})
\DEQSZ\label{compastimate}
\frac d{dt} {\mathfrak{F}}_A:L^p(0,T;E)\longrightarrow L^p(0,T;E)
\EEQSZ
are bounded. Theorem 3.3 of \cite{maxjan} gives that
\DEQSZ\label{maxjan1}
{\mathfrak{F}}_A:L^p(0,T;E)&\longrightarrow & C^0(0,T;D_A({1-\frac 1p},p)).
\EEQSZ
is bounded. Finally, since we also have to verify compactness, we present the following Theorem.
\begin{corollary}\label{C:comp}
Let Assumption \ref{assum-1} be satisfied.
Assume  that  three  non-negative
numbers $\alpha,  \beta, \delta$ satisfy  the  following condition
\DEQSZ 
\alpha - \frac1q &>&\beta + \delta.
\label{cond:1a}
\EEQSZ 
Then the operator
$\mathfrak{F}_A^\alpha : L^q(0,T;E ) \to
{C}^\beta_b([0,T];D(A^\delta))$ is bounded. Moreover, if the
operator $A^{-1}:E\to E$ is  compact, then  the
operator $\mathfrak{F}_A^\alpha : L^q(0,T;E ) \to {C}^\beta_b([0,T];D(A^\delta))$  is also compact.\\
In particular, if $\alpha > \frac1 q$ and the operator $A^{-1}:E\to E$ is compact, the map
$\mathfrak{F}_A^\alpha :L^q(0,T; E) \to C([0,T];E )$ is compact.
\end{corollary}

\noindent Let $\mathfrak{A}=(\Omega, \CF,\BF,\PP)$ be a complete probability space
and filtration $\BF=(\CF_t)_{t\ge 0}$  satisfying the usual conditions.
Let $H_1$ be a Hilbert space and $E$ a UMD-Banach space of type $2$.
$W$ be a in $H_1$ cylindrical Wiener process, ${\Upsilon}(H_1,E)$ the space of linear operators from $H_1$ to $E$ being $\Upsilon$--radonifying
and
\DEQSZ\label{def.MA}
\lqq{ \CN_{\MA}^p({E}):=\Big\{ \xi:[0,T]\times \Omega\to {E},\quad \mbox{such that} \quad }
&&\\
\notag
&& \mbox{$\xi$ is progressively measurable over $\mathfrak{A}$ and}\quad \int_0^T |\xi(s)|^p_{\footnotesize {\Upsilon}(H_1,{E})}\, ds<\infty\Big\}.
\EEQSZ
To handle the stochastic convolution, let us define a process $\xi\in \CN_{\MA}^2({E})$
the operator $\mathfrak{S}_A$ by
\DEQSZ\label{def_st_con}
\mathfrak{S}_A\xi(t)=\int_0^ t e^{(t-s)A} \xi(s)dW(s)
.
\EEQSZ 

\noindent If some non--negative numbers $\beta\in [0,\frac 12)$, $\delta\ge 0$, and $\nu$ satisfy the following inequality
$$
\beta+\delta+\frac 1p<\frac 12,
$$
then any process belonging to  $\CN_{\MA}^q({E})$ possesses a modification $\tilde x$ of the process $[0,T]\ni t\mapsto \mathfrak{S}_A\xi(t)$, such that
$$
\tilde x\in C^{\beta}_b(0,T;D(A^\delta), \quad a.s.
$$
Besides,  there exists some constant $C_T>0$ independent of $\xi$ such that (see \cite[Theorem 3.5]{maxjan})
\DEQSZ\label{maxjanineq}
\EE \|x\|^p_{\CC^{\beta}_b(0,T;D(A^\delta))}   & \le & C_T \EE \int_0^T \|\xi(s)\|^p_{\Upsilon(H_1,{E})}\, ds.
\EEQSZ

\bigskip 
\section{\textbf{Embedding results}}\label{s:embedding}
\noindent {An important part of this article consists in the correct embedding of the nonlinearity, the subsequently gathered results are used for instance in Section~\ref{s:uniform}.}  

\bigskip 
\subsection{\textbf{$L^\ell$-embedding}}\hfill\\

\noindent Recall the Banach space $\mathbb{H}_\alpha:= L^2(0,T; H^{\alpha+\frac{\aleph}{2}}_2(\CO))\cap L^\infty(0,T; {H^{\alpha}_2}(\CO))$
equipped with the norm
\begin{equation}\label{eq:BZeq}
\|\xi\|_{\mathbb{H}_\alpha}:= \|\xi\|_{L^2(0,T; H^{\alpha+\frac{\aleph}{2}}_2)}+ \|\xi\|_{L^\infty(0,T; {H^{\alpha}_2})} ,\quad\xi\in \mathbb{H}_{\alpha, \aleph}.
\end{equation}
Then, we have the following embedding.

\begin{proposition}\label{interp_rho}
Let $l_1,l_2\in(2,\infty)$ with $\frac{d}{2}-\alpha\le (\FEplus) \frac 2{l_1}+\frac d{l_2}$.
Then there exists $C>0$ such that
$$
\|\xi\|_{L^{l_1}(0,T;L^{l_2})}\le C\|\xi\|_{\mathbb{H}_{\alpha, \aleph}},\quad \xi\in\mathbb{H}_{\alpha, \aleph}.
$$
\end{proposition}
\begin{proof}
First, let us note that for $\delta>0$ such that
\begin{equation}\label{eq:deltaembedding}\frac {1}{l_2}\ge \frac{1}{2}-\frac{\delta}{d},
\end{equation}
we have the embedding $H^{\delta}_2(\CO)\hookrightarrow L^{l_2}(\CO)$, and therefore
$$
\|\xi\|_{L^{l_1}(0,T;L^{l_2})}\le C\|\xi\|_{L^{l_1}(0,T;H^{\delta}_2)},\quad \xi\in L^{l_1}(0,T;H^{\delta}_2).
$$
Due to interpolation, 
\cite[Theorem 6.4.5 (7), p.\ 152-153]{bergh}, the parameter relations 
\begin{equation}\label{eq:delta_theta}
\delta\le\theta\alpha+(1-\theta)(\alpha+1) \mbox{ for some } \theta\in (0,1)
\end{equation}
and 
\[
\frac  1{l_1}\ge \frac{1}{2}(1-\theta)\qquad \Leftrightarrow \qquad 2 \geq (1-\theta)l_1
\]
yields
$$
\|\xi\|_{L^{l_1}(0,T;H^{\delta}_2)}\le  C\|\xi\|_{L^\infty(0,T;H^{\alpha}_2)}^\theta \|\xi\|_{L^2(0,T;H^{\alpha+1}_2)}^{1-\theta},\quad \xi\in\BH_{\alpha, \aleph}.
$$
\end{proof}

\bigskip 
\subsection{\textbf{An interpolated Stroock-Varopoulos inequality}}\hfill\\

\noindent Next, we present a variant of the Stroock-Varopoulos inequality and its proof.
It is used in Appendix~\ref{Appendix_A_noise} to relax conditions on the noise parameter $\gamma_i$. 
See, e.g. \cite[Lemma 3.6]{DGG2020} for another version of this result.

\begin{proposition}\label{runst1}
For any $\gamma>1$, $\theta\in(0,\frac{1}{\gamma})$, there exists a constant $C>0$ such that
\DEQS
\|\eta\|_{L^{2\gamma}(0,T;H^{\theta}_{2\gamma})}^{2\gamma} \le C\int_0^T |\eta(s)^{[\gamma-1]}\nabla \eta(s)|_{L^2}^2\, ds,
\EEQS
where $z^{[\gamma-1]}:=|z|^{\gamma-2}z$ for $z\in\RR$.
\end{proposition}

\begin{proof}
By Subsection~\ref{ss:Besov} combined with \cite[p.~365, Remark 3]{runst} and \cite[p.~365, Subsection 5.4.4, display (3)]{runst}, we have for any  $p\in(1,\infty)$, $s\in(0,1)$, $\mu\in(0,1)$ and $\ep\in(0,s\mu)$,
\begin{equation}\label{from_runst}
| |w|^\mu |_{H^{s\mu-\ep}_{\frac{p}{\mu}}}  =	| |w|^\mu|_{F^{s\mu-\ep}_{\frac p\mu,2}} \le C	| |w|^\mu |_{F_{\frac p\mu,\frac 2\mu}^{s\mu}}\le C	 |w|_{ F_{p, 2 }^s} ^\mu = C| w|_{H^{s}_p}^\mu, \quad w\in H^s_p(\CO).
\end{equation}
From \eqref{from_runst}, we know that  for any  $\gamma>1$, $p\in (1,\infty)$, $\theta\in(0,\frac 1\gamma)$, $\ep>0$, there exists a constant $C>0$ such that
$$
|w|_{H^\theta_{p\gamma}}=||w^{\gamma} |^\frac 1\gamma |_{H^\theta_{p\gamma}}\leq C ||w|^\gamma|^{\frac 1\gamma}_{H^{ \gamma(\theta+\ep)}_{p}}.$$
In particular, for any  $0<\theta<\frac 1 \gamma$, 
$\ep = (1-\theta)/\gamma$ and $p=2$, there exists a constant $C>0$ such that 
$$	
|w|_{H^\theta_{2\gamma}}^\gamma \le C||w|^\gamma|_{H^1_2}.
$$
Since we know by the chain rule and the Poincar\'e inequality,
$$\int_0^ T |\eta^{[\gamma-1]}(s)\nabla \eta(s) |_{L^2}^2\, ds=\int_0^ T |\nabla |\eta(s)| ^{\gamma}|_{L^2}^2\, ds\ge c\int_0^ T ||\eta(s)| ^{\gamma}|_{H^1_2}^2\, ds$$
we know that for any $\theta<\frac 1\gamma$,
$$\int_0^ T |\eta(s)|_{H^\theta_{2\gamma}}^{2\gamma}\, ds\le C \int_0^ T |\eta ^{[\gamma-1]}(s)\nabla \eta(s) |_{L^2}^2\, ds.$$
This finishes the proof. 
\end{proof}


\bigskip
\subsection{\textbf{Besov spaces over bounded domains}}\label{ss:Besov}\hfill\\
In this section, we present several inequalities concerning the multiplication of functions in Besov spaces. First, we summarize various inequalities from the book by Runst and Sikel \cite{runst}. It is important to note that these inequalities are valid on $\mathbb{R}^d$, while we are working on a bounded domain. Therefore, in the second step, we investigate the conditions under which these inequalities hold on a bounded domain.
\begin{theorem}\label{RS2a}
Let $s_1\le s_2$ and $s_1+s_2>d\cdot\max(0,\tesfalem{\frac 2p-1})$. 
\begin{itemize}
\item Assume $s_2>\frac dp$ and $q\ge \max(q_1,q_2)$. Then, if $s_2>s_1$,
$F_{p,q_1}^{s_1}(\RR^d) \cdot F_{p,q_2}^{s_2} (\RR^d)\hookrightarrow F_{p,q_1}^{s_1}(\RR^d)$;
if $s_2=s_1$,
$F_{p,q_1}^{s_1}(\RR^d) \cdot F_{p,q_2}^{s_2} (\RR^d)\hookrightarrow F_{p,q}^{s_1}(\RR^d)$;
\item Let $s_1=s_2=\frac dp$ and $q\ge \max(q_1,q_2)$. If $0<p\le 1$, then  $F_{p,q_1}^{s_1}(\RR^d) \cdot F_{p,q_2}^{s_2} (\RR^d)\hookrightarrow F_{p,q}^{s_1}(\RR^d)$;
\item
%
If $s_2<\frac dp$, then
$F_{p,q_1}^{s_1}(\RR^d) \cdot F_{p,q_2}^{s_2} (\RR^d)\hookrightarrow F_{p,q}^{s_1+s_2-\frac dp}(\RR^d)$.
\end{itemize}

\end{theorem}
\begin{proof}
See \cite[p.\ 190]{runst}.\end{proof}

\del{ \begin{theorem}\label{RS2b}
Let $s_1<0<s_2$, $q\ge q_1$, and 
\begin{itemize}
\item $s_1+s_2>0$,
\item 
$\frac 1p\le \frac 1{p_1}+\frac 1{p_2}$,
\item 
$ \frac 1p>\frac 1{p_1}+\frac 1d\lk( \frac d{p_2}-s_2\rk)_+$
\item $s_1+s_2>\frac d{p_1}+\frac d{p_2} -d$
\end{itemize}
Then
$F_{p,q_1}^{s_1}(\RR^d) \cdot F_{p,q_2}^{s_2} (\RR^d)\hookrightarrow F_{p,q}^{s_1}(\RR^d)$.

\end{theorem}

\begin{proof}
See \cite[p.\ 171]{runst}.\end{proof}

\begin{theorem}\label{RSc}
Let $0<s_1\le s_2$, $q\ge q_1$, and 
\begin{itemize}
\item 
$\frac 1p\le \frac 1{p_1}+\frac 1{p_2}$,
\item 
$ \frac dp-s_1>\begin{cases*}
(\frac d{p_1}-s_1)_+ +(\frac d{p_2}-s_2)_+
, &\quad  \mbox{ if }
$			\max (\frac d{p_i}-s_i)>0,$
\\
\max_{i}(\frac d{p_i}-s_i)>0 & \quad \mbox{otherwise}.
\end{cases*}
$
\item $s_1+s_2>\frac d{p_1}+\frac d{p_2} -d$
\end{itemize}
Then
$F_{p,q_1}^{s_1}(\RR^d) \cdot F_{p,q_2}^{s_2} (\RR^d)\hookrightarrow F_{p,q}^{s_1}(\RR^d)$.

\end{theorem}

\begin{proof}
See \cite[p.\ 171]{runst}.\end{proof}
}
\del{ 
\begin{theorem}\label{RSd}
Let $\mu(\frac dp-s)<d$ and $0<s<\min(\frac dp,\mu)$
and put
$$
t=\frac d{s+\mu(\frac dp-s)}.
$$
Then, there exists a constant $C>0$ such that
$$| \, |f|^\mu|_{F_{t,q}^s} \le C|f|^\mu_{F_{p,q}^s}, \quad f\in F_{p,q}^s
$$
provided $q>\frac d{d+s}$	and
$$| \, |f|^\mu|_{B_{t,q}^s} \le C|f|^\mu_{B_{p,q}^s},\quad f\in B_{p,q}^s
$$
provided $q\le \frac d{d+s}$.
\end{theorem}

\begin{proof}
See \cite[p.\ 363]{runst}.\end{proof}
}

\noindent To work on bounded domains, let us introduce the extensions. For a bounded domain $\CO\subset \RR^d$  for $g\in \mathcal{S}'(\RR^d)$, the restriction of $g$ to $\CO$ is an element of $\mathcal{D}(\CO)$ and will be denoted by $g|_{\CO}$.
Let us fix
a function $f:\CO\to\RR$. We call $g\in \mathcal{S}'(\RR^d)$ an extension of $f$ if
$g|_{\CO}=f$.
In case $\CO$ is a $C^\infty$-domain and $d\in\NN$ with $|s|<d$. Then
for $\frac 1d<p<\infty$ and $0<q\le \infty$,
there exists a unique bounded operator $\mbox{ext}$ from $F_{p,q}^s(\CO)$ to $F_{p,q}^s(\RR^d)$. Also,
for $\frac 1d<p\le \infty$ and $0<q\le \infty$,
there exists a unique bounded operator $\mbox{ext}$ from $B_{p,q}^s(\CO)$ to $B_{p,q}^s(\RR^d)$ (see \cite[p. 76]{runst}).
With this definition, we can define the Besov and Lizorin-Triebel spaces on bounded domains.

\begin{definition}
Let $s\in\RR$ and $0<q\le \infty$.
\begin{enumerate}
\item if $0<p<\infty$, then we put
$$
F_{p,q}^s(\CO)=\lk\{ f\in\mathcal{D}'(\CO):\exists \, g\in F^s_{p,q}(\RR^d)\,\mbox{ with } g\big|_{\CO}=f\rk\}
$$
and
$$
|f|_{F_{p,q}^s}:=|f|_{F_{p,q}^s(\CO)}=\inf |g|_{F_{p,q}^s(\RR^d)},
$$
where the infimum is taken over all $g\in B_{p,q}^s(\RR^d)$ such that $g\big|_{\CO}=f$.
\item if $0<p\le \infty$, then we put
$$
B_{p,q}^s(\CO)=\lk\{ f\in\mathcal{D}'(\CO):\exists \, g\in B^s_{p,q}(\RR^d)\,\mbox{ with } g\big|_{\CO}=f\rk\}
$$
and
$$
|f|_{B_{p,q}^s}:=|f|_{F_{p,q}^s(\CO)}=\inf |g|_{B_{p,q}^s(\RR^d)},
$$
where the infimum is taken over all $g\in B_{p,q}^s(\RR^d)$ such that $g\big|_{\CO}=f$.
\end{enumerate}
\end{definition}
\noindent We aim at showing that Theorem~\ref{RS2a} is also valid on a bounded domain.
First, note that the multiplication with $1_\CO$ is a bounded linear operator from ${B^s_{p,q}(\RR^d)}$ into itself.
In particular, by Theorem 3, \cite[p. 201]{runst}, we have
$$
|f \, {1}_{\CO}|_{B^s_{p,q}(\RR^d)}\le  |f|_{B^s_{p,q}(\RR^d)}\lk( | {1}_{\CO}|_{B^\frac dp_{p,\infty}(\RR^d)}+| {1}_{\CO}|_{B^\frac dp_{p,\infty}(\RR^d)}\rk).
$$
Secondly, note that  for any function $h:\CO\to\RR$, the function
$$
1_\CO(x) h(x) =\begin{cases} h(x),& \mbox{ if } x\in\CO
\\
0 & \mbox{otherwise},
\end{cases}
$$
is an extension.
Now, we can formulate the following proposition.
\begin{proposition}\label{mmmext}
If $s<\frac dp$, then there exists a number $C>0$ such that
$$|f|_{B^s_{p,q}(\CO)}\le |f1_\CO|_{B^s_{p,q}(\RR^d)}
\le C |f|_{B^s_{p,q}(\CO)}
,\quad \forall \, f\in B^s_{p,q}(\CO).
$$
\end{proposition}
\begin{proof}
The direction
$|f|_{B^s_{p,q}(\CO)}\le|f1_\CO|_{B^s_{p,q}(\RR^d)}$ is straightforward,  since
we look
at the infimum. It remains to show that $ |f1_\CO|_{B^s_{p,q}(\RR^d)}
\le C |f|_{B^s_{p,q}(\CO)}$.
Note,  since the Dirac distribution  $\delta$ belongs to $ B^{s}_{p,q}(\RR^d)$, where $s<\frac dp-d$ or $s=\frac dp-d$ and $q=\infty$ (see Remark 3 \cite[p. 34]{runst}), we know that $1_{\CO}\in B^{s}_{p,q}(\RR^d)$, where $s<\frac dp$ or $s=\frac dp$ and $q=\infty$.
Since $1_\CO f$ is an extension of $f$, we know that
\DEQS
|f|_{B^s_{p,q}(\CO)}\le |f \, {1}_{\CO}|_{B^s_{p,q}(\RR^d)}.
\EEQS
Now, let $\ep>0$. Then, due to the definition of the norm of ${B^s_{p,q}(\CO)}$,
there exists a function $g:\RR^d\to\RR$ being an extension of $f$ and
$$
|g|_{B^s_{p,q}(\RR^d)}\le (1+\ep)|f|_{B^s_{p,q}(\CO)}.
$$
Next, since $g=f=1_\CO \,f$ on $\CO$ and by the equivalence characterisation of $B_{p,q}^s(\RR^d)$ in \cite[p. 8]{triebel_1983} or \cite[p. 54]{runst}, we know that
$$
|f \, {1}_{\CO}|_{B^s_{p,q}(\RR^d)}\le |g|_{B^s_{p,q}(\RR^d)}
. $$
Hence, it follows that there exists some constant $C>0$ such that
$$
|f \, {1}_{\CO}|_{B^s_{p,q}(\RR^d)}
\le C|f|_{B^s_{p,q}(\CO)},$$
which was the assertion.
\end{proof}

\bigskip
\section{\textbf{A compactness result}}\label{dbouley-space}

\noindent {In this section, we present several results regarding compactness in both time and space for functions. The compactness criteria in Theorem~\ref{th-gutman} and Theorem~\ref{th-gutman2} are applied to show the compactness of the fixed point operator in Subsection~\ref{ss:compactness}. }

\noindent First, let us introduce the following spaces. Let $B_1$ be a separable Banach space.
Given $p\in (1,\infty)$, $\alpha\in(0,1)$, let $\WW ^ {\alpha}_p (0,T;B_1)$ be the Sobolev space
of all $u\in L^p(0,T;B_1)$ such that
$$
\int_0^T \int_0^T \,{|u(t)-u(s)|_{B_1} ^ p\over |t-s|^{\alpha p}}\,ds\,dt<\infty
$$
equipped with the norm
$$
\lk\| u\rk\|_{ \WW^ {\alpha}_p(0,T ;B_1)}:=\lk( \int_0^T \int_0^T \,{|u(t)-u(s)|_{B_1} ^ p\over |t-s| ^ {\alpha p}}\,ds\,dt\rk) ^ \frac 1p.
$$
Next, let us briefly introduce an important Lemma.
Let $B_0\subset B\subset B_1$ be Banach spaces, $B_0$ and $B_1$ reflexive, with compact embedding of $B_0$ to $B$. Let $p\in(1,\infty)$ and $\alpha\in(0,1)$ be given. Then, we know from the Aubins-Lions-Simon Lemma that the embedding
$$
L^p(0,T;B_0)\cap \WW ^{\alpha}_p (0,T;B_1)
\hookrightarrow
L^p(0,T;B)
$$
is compact. Applying the Aubins-Lions-Simon Lemma together with the following propositions will give us compactness.

\begin{theorem}\label{th-gutman}
Let $B_0\subset B\subset B_1$ be Banach spaces, $B_0$ and $B_1$ reflexive, with compact embedding of $B_0$ in $B$. Let $p\in(1,\infty)$ and $\alpha\in(0,1)$ be given. Let $X$ be the space
$$
X=L^ p([0,T];B_0)\cap \WW _p^{\alpha} (0,T;B_1).
$$
Then the embedding of $X$ in $L ^ p( [0,T];B)$
is compact.
In particular, for any $R>0$ the set
\[
\lk\{ x\in  L^p(0,T;B) : |x|_{L^p(0,T;B_0)}+|x|_{\WW ^{\alpha}_p (0,T;B_1)}\le R\rk\}
\]
is compact in $L^p(0,T;B)$.

\end{theorem}

\noindent A similar result can be derived for the space of continuous functions.
Again, we need a compact containment condition and regularity in time. Here, it is essential for us that the spaces where we have to show the compact containment condition and the spaces where we have to show the regularity in time can be different, similar to the case of $ Lp (0,T;B)$.

\begin{theorem}\label{th-gutman2}
A set $\mathcal{A}\subset  C_b^{(0)}((0,T]; B)$ is tight, if there exists two separable Banach spaces $B_0$ and $B_1$ such that
$B_0\subset B\subset B_1$ and $B_0\hookrightarrow B$ compactly, $B\hookrightarrow B_1$ densely, such that
\begin{itemize}
\item the range of $\mathcal{A}$ is bounded in $B$. In particular, there exists a constant $R>0$ such that
$$\sup_{\xi\in\mathcal{A}}\sup_{0\le s\le T}|\xi(s)|_B\le R.
$$
\item the compact containment condition holds,
In particular, 
there exists a constant there exists a constant $C>0$ such that
$$\sup_{0\le t\le T}\sup_{\xi\in\mathcal{A}} |\xi(t)|_{B_0}\le C.
$$
\item
there exist  constants $C,\alpha >0$ and such that
$$\sup_{\xi\in\mathcal{A}}|\xi|_{C^{(\alpha)}_b([0,T];B_1)}\le C.
$$
\end{itemize}
\end{theorem}
\begin{proof}
Note, that there exists some $\theta\in(0,1)$ such that $ B\hookrightarrow [B_0,B_1]_\theta$.
A set $\mathcal{A}$ is praecompact if for any $\ep>0$ a finite $\ep$-covering exists.
First, let $\{t_0<t_1<\ldots t_{n-1}<t_n=T\}$ a time grid such that
$$
\sup_{\xi\in\mathcal{A}} \,\max_{i=2,\ldots,n} \sup_{t_{i-1}\le s\le t_i}\max(|\xi(s)-\xi(t_{i})|,|\xi(t_{i-1})-\xi(s)|)\le \frac \ep4\frac 1{R+\frac \ep 4} ,
$$
and
$$ \max_{2\le i\le n}(2C)^\theta|t_i-t_{i-1}|^{\alpha(1-\theta)}\le \frac \ep4.
$$
Then, for any $i=1,\ldots,n$ let $A(i):=\{ x^i_1,\ldots ,x^i_{j_i}\}$ such that $A(i)$ is a $\frac \ep 4$-covering of the set $\{ \xi(t_i):\xi\in\mathcal{A}\}$. Now, the set of functions
$$
\hat{ \mathcal{A}}:=\lk\{ \eta: \eta(t_i)\in A(i), \, \mbox{ $\eta$ is linear on $[t_{i-1},t_i]$, $i=2,\ldots,n$ }\rk\}
$$
is an $\ep$-covering of $\mathcal{A}$.
This is simple to show. Let $\ep>0$ and $\xi\in\mathcal{A}$ and $s\in[0,T]$. Then we have to show that there exists a $\hat\eta\in  \hat{ \mathcal{A}}$ such that
$|\hat \eta(s)-\xi(s)|_B\le \ep$. First, for a given $\xi$ let $\hat\eta$ be a function, such that for any $i=1,\ldots,n$ we put
$\hat\eta(t_i):=\tilde\eta$, where $\tilde\eta\in A(i)$ with
$|\tilde\eta-\xi(t_i)|_B\le \frac \ep4$. Due to the construction of $A(i)$, such an element $\tilde \eta$ exists.
Additionally, $\hat\eta$ is linearly interpolated between the grid points.
By the triangle inequality we can write for $s\in[t_{i-1},t_i]$
\DEQS
|\hat \eta(s)-\xi(s)|_B\le|\hat \eta(s)-\hat \eta(t_i)|_B+|\hat \eta(t_i)-\xi(t_i)|_B+|\xi(t_i)-\xi(s)|_B.
\EEQS
Since $\hat\eta$ is linear, and $\mathcal{A}$ bounded in $B$ by $R$, therefore $A(i)$ and $A(i-1)$ bounded in $B$ by $R+\frac \ep4$, we can write
$$ |\hat \eta(s)-\hat \eta(t_i)|_B\le |s-t_i|\lk(R+\frac \ep 4\rk)\le \ep 4
.
$$
Since for any $i$ the set $A(i)$ is a $\frac \ep 4$ covering of the set  $\{ \xi(t_i):\xi\in\mathcal{A}\}$, we know that $|\hat \eta(t_i)-\xi(t_i)|_B\le \frac \ep 4$. Finally, $|\xi(t_i)-\xi(s)|_B\le |\xi(t_i)-\xi(s)|_{B_0}^\theta|\xi(t_i)-\xi(s)|_{B_1}^{1-\theta}$.
First, note that we have $|\xi(t_i)-\xi(s)|_{B_1}^{1-\theta}\le |s-t_i|^{\alpha(1-\theta)}$.
Secondly, we know $|\xi(t_i)-\xi(s)|_{B_0}^\theta\le (2C)^\theta$.
Collecting altogether, we have shown that $|\hat \eta(s)-\xi(s)|_B\le\ep$. Hence, for any $\ep>0$  an $\ep$--covering
exists and, therefore, $\mathcal{A}$ is precompact.
\end{proof}

Let $\mathfrak{A}=(\Omega,\CF,(\CF_t)_{t\ge 0},\PP)$ a filtered probability space, $E$ be a Banach space, and  $\{ \xi_n:[0,T]\times \Omega\longrightarrow B\mid n \in\NN\}$ be a set of processes. Here, in this section, we are interested under which condition the family of probability measure $\{ \Law(\xi_n )\mid n \in\NN\}$ over
$L^p(0,T,B)$ or $C_b^{(0)}(0,T;B)$ is tight.
In doing so, we will characterise a compact subset of $L^p$-spaces and the continuous function space. From this characterization, we can then derive
characterizations of tightness of  the family given by  $\{ \Law(\xi_n)\mid n\in\NN\}$.

%


\end{document}